\documentclass{article}%
\usepackage{amsmath}
\usepackage{amssymb}
\usepackage{amsmath}
\usepackage{amsfonts}
\usepackage{graphicx}
\usepackage[latin1]{inputenc}
\usepackage[english]{babel}
\usepackage{dsfont}
\usepackage{hyperref}
\usepackage{url}
\usepackage{mathrsfs}
\usepackage{abstract} 
\usepackage{stmaryrd}
\usepackage{authblk}  

\usepackage{vmargin}
\setmarginsrb{2.5cm}{2.5cm}{2.5cm}{2.5cm}{0cm}{0cm}{1.25cm}{1.25cm}
\providecommand{\U}[1]{\protect\rule{.1in}{.1in}}
\newtheorem{theorem}{Theorem}

\newtheorem{corollary}[theorem]{Corollary}

\newtheorem{definition}[theorem]{Definition}

\newtheorem{lemma}[theorem]{Lemma}

\newtheorem{proposition}[theorem]{Proposition}
\newtheorem{remark}[theorem]{Remark}

\newenvironment{proof}[1][Proof]{\noindent\textbf{#1.} }{\ \rule{0.5em}{0.5em}}
\begin{document}
	\title{$\mathbb{L}^p$ $(p>1)$-solutions for BSDEs with jumps and stochastic monotone generators}
	
%
	
	\author[1]{Badr ELMANSOURI\thanks{Corresponding author. Emails: \href{mailto:badr.elmansouri@edu.uiz.ac.ma}{badr.elmansouri@edu.uiz.ac.ma} \, \& \, \href{mailto:b.elmansouri@uca.ac.ma}{b.elmansouri@uca.ac.ma}}\thanks{ORCID: \href{https://orcid.org/0000-0003-2603-894X}{ORCID: 0000-0003-2603-894X}}}
	\author[2]{Mohamed EL OTMANI\thanks{Email: \href{mailto:m.elotmani@uiz.ac.ma}{m.elotmani@uiz.ac.ma}}\thanks{ORCID: \href{https://orcid.org/0000-0002-9674-6469}{ORCID: 0000-0002-9674-6469}}}
	
	\affil[1]{Cadi Ayyad University (UCA), National School of Applied Sciences of Marrakech (ENSA-M), BP 575, Avenue Abdelkrim Khattabi, 40000, Guéliz, Marrakech, Morocco.}

		\affil[2]{Faculty of Sciences Agadir, Ibn Zohr University, Laboratory of Analysis and Applied Mathematics (LAMA), Hay Dakhla, BP8106, Agadir, Morocco.}

%
	
	%
%
	\date{}
	\maketitle
\begin{abstract}
We study $\mathbb{L}^p$-solutions for multidimensional discontinuous backward stochastic differential equations in a filtration generated by a Brownian motion and an independent integer-valued random measure. Under suitable $\mathbb{L}^p$-integrability conditions on the data, we establish the existence and uniqueness of $\mathbb{L}^p$-solutions in both cases: $p \geq 2$ and $p \in (1,2)$. The generator is assumed to be stochastically monotone in the state variable $y$, stochastically Lipschitz in the control variables $(z, u)$, and to satisfy a stochastic linear growth condition, along with an appropriate $\mathbb{L}^p$-integrability requirement.
\vspace{0.3cm}


\end{abstract}

	\section{Introduction}
The notion of linear backward stochastic differential equations (BSDEs) was introduced by Bismut \cite{Bismut1973} as adjoint equations associated with the stochastic Pontryagin maximum principle in stochastic control theory. The general case of nonlinear BSDEs in a Brownian framework was later studied by Pardoux and Peng \cite{PardouxPeng1990}. Given a time horizon $T \in (0,+\infty)$ and a standard Brownian motion $W = (W_t)_{t \leq T}$, a solution of such an equation, associated with a terminal value $\xi$ and a generator (or driver) $f$, is a pair of stochastic processes $(Y_t, Z_t)_{t \leq T}$ adapted to the natural filtration generated by $W$ that satisfies:
\begin{equation}\label{basic BSDE Intro Cont}
	Y_t=\xi+\int_{t}^{T} f(s,Y_s,Z_s)ds-\int_{t}^T Z_s d W_s, \quad t \in [0,T].
\end{equation}
Pardoux and Peng \cite{PardouxPeng1990} established the existence and uniqueness of a solution for \eqref{basic BSDE Intro Cont} under suitable assumptions. Specifically, they required the square integrability of $\xi$ and the driver process $(f (t, \omega, 0, 0))_{t\leq T}$, as well as the uniform Lipschitz continuity of the generator $f$ with respect to $(y, z)$. 

Subsequently, Rong \cite{royer2006backward} extended these results to the case involving jumps, where a jump term was added to the formulation of the BSDE \eqref{basic BSDE Intro Cont}, stemming from a Poisson random measure $N$ on $[0,T] \times \mathbb{R}^\ast$. A solution to this equation, associated with data $(\xi,f)$, now consists of a triplet of processes $(Y_t, Z_t, U_t)_{t \leq T}$ adapted to the natural filtration generated by $W$ and $N$, satisfying:
\begin{equation}\label{basic BSDE Intro Disc}
	Y_t=\xi+\int_{t}^{T} f(s,Y_s,Z_s,U_s)ds-\int_{t}^T Z_s d W_s-\int_{t}^{T}\int_{E}U_s(e) \tilde{N}(ds,de),\quad t \in [0,T].
\end{equation}
Here, $\tilde{N}$ is the compensated Poisson martingale measure associated with $N$, given by $\tilde{N}(ds,de) = N(ds,de) - \lambda(de)ds$, where $\lambda$ is a $\sigma$-finite measure on $\mathbb{R}^\ast$ such that $\int_{\mathbb{R}^\ast} (1 \wedge |e|^2)\lambda(de) < +\infty$. 

Under the assumption of square integrable data and a monotonicity condition on the driver, Royer \cite{royer2006backward} studied BSDEs of the form \eqref{basic BSDE Intro Disc} and provided applications to nonlinear expectations. The connection between \eqref{basic BSDE Intro Disc} and the maximum principle for optimal control of stochastic processes was considered by Tang and Li \cite{TangLi1994}.

Since then, the theory of BSDEs has proven to be a powerful tool in addressing various mathematical problems arising in different domains, including mathematical finance \cite{Delong2013,el1997backward}, stochastic control and stochastic game theory \cite{hamadene1995backward,hamadene1995zero,OksendalSulem2019}, and the theory of partial differential equations \cite{barles1997backward,PDEPARD,peng1991probabilistic,Pardoux1999}.

Many works have established existence and uniqueness results under weaker assumptions than those of Pardoux and Peng \cite{PardouxPeng1990}. In particular, under the assumption of square integrable data, many authors have focused on relaxing the Lipschitz condition on the driver $f$. In this context, Briand and Carmona \cite{briand2000bsdes} and Pardoux \cite{Pardoux1999} considered generators that are monotonic with respect to the state variable $y$, under different growth conditions within a Brownian framework (see also \cite{peng1991probabilistic}). The generalization to the case involving Poisson jumps was developed by Royer \cite{royer2006backward} and later extended by Kruse and Popier \cite{Kruse2015} to more general filtrations (see also \cite{ElmansouriElOtmani2023}).

Beyond the square integrability of the data, relatively few papers have addressed the problem of existence and uniqueness of solutions when the data in BSDEs \eqref{basic BSDE Intro Cont} or \eqref{basic BSDE Intro Disc} are not square integrable. In the continuous case within a Brownian filtration, Briand and Carmona \cite{briand2000bsdes}, Briand et al. \cite{Briand2003}, Chen \cite{chen2010p}, El Karoui et al. \cite{el1997backward}, Fan and Jiang \cite{fan2012p}, and Xiao et al. \cite{xiao2015lp} proved the existence and uniqueness of a solution to \eqref{basic BSDE Intro Cont} when the data belong only to $\mathbb{L}^p$ for some $p \in [1,+\infty)$. The case with jumps was studied by Yao \cite{yao2017lp}, where the author proved the existence and uniqueness of a solution to \eqref{basic BSDE Intro Disc} with $\mathbb{L}^p$-integrable data for $p \in (1,2)$, assuming a monotonic driver with at most linear growth in $y$. In the same framework, Li and Wei \cite{LI20141582} investigated the existence and uniqueness of fully coupled forward-backward SDEs, where the noise is generated by $W$ and $N$, and the drivers are monotonic with $\mathbb{L}^p$-integrability conditions for $p \in [2,+\infty)$. Kruse and Popier \cite{Kruse2015,Kruse2017} studied BSDEs with jumps in a general filtration supporting both the Brownian motion $W$ and the Poisson random measure $N$, where an additional orthogonal martingale term appears in \eqref{basic BSDE Intro Disc}. In their seminal work \cite{Kruse2015,Kruse2017}, they proved the existence and uniqueness of a solution under a monotonicity condition on the driver and an $\mathbb{L}^p$-integrability assumption on the data for $p \in (1,+\infty)$ (see also \cite{eddahbi2017} for the case $p \in [2,+\infty)$). Elmansouri and Marzougue \cite{ELMANSOURI2025110407} established the existence and uniqueness of an $\mathbb{L}^p$-integrable solution, for $p\in(1,2)$, for BSDEs with a default jump. More recently, El Jamali \cite{ElJamali} studied the existence and uniqueness of $\mathbb{L}^p$-solutions, with $p\in(1,2)$, for BSDEs driven by a Lévy process and with a Lipschitz generator. This result was later extended to the case of continuous and left-continuous generators by El Jamali and Elmansouri in \cite{ElJamaliElmansouri+2026+1+20}.

Motivated by these contributions, in this paper, we aim to establish the existence and uniqueness of multidimensional BSDEs with jumps of the form \eqref{basic BSDE Intro Disc} with $\mathbb{L}^p$-data for $p \in (1,+\infty)$ under more general assumptions on the generator $f$. More precisely, we investigate $\mathbb{L}^p$-solutions for discontinuous BSDEs \eqref{basic BSDE Intro Disc} where the coefficients satisfy a stochastic monotonicity assumption with respect to the state variable $y$, a stochastic Lipschitz condition with respect to the control variables $(z,u)$, and a stochastic linear growth condition. We emphasize that our stochastic parameters are general, meaning they are stochastic processes that are not necessarily bounded. Compared to the existing literature, and to the best of our knowledge, there are no similar studies developing $\mathbb{L}^p$-solutions for BSDEs of the form \eqref{basic BSDE Intro Disc} in this general framework. Let us first clarify the main differences with the closest related works. Kruse and Popier \cite{Kruse2015,Kruse2017} study BSDEs with jumps under monotonicity assumptions in $y$ and Lipschitz conditions with respect to the remaining variables. In contrast, our framework allows stochastic monotonicity in $y$ and stochastic Lipschitz coefficients with respect to $(z,u)$. The work of Yao \cite{yao2017lp} also deals with BSDEs with jumps under stochastic monotonicity conditions, but the monotonicity coefficient is assumed to satisfy a stronger boundedness requirement, namely an essential boundedness condition on its time integral, and the Lipschitz coefficients with respect to $z$ and $u$ are controlled by deterministic positive integrable functions. In the present paper, these coefficients are allowed to be stochastic processes. Moreover, Li and Wei \cite{li2024weightedlppgeq1solutionsrandom} treat a classical Brownian framework, while our setting includes jumps, which makes the analysis more delicate, especially in the $\mathbb{L}^p$-setting with $p\neq 2$.

The only related recent works in a stochastic monotonicity framework are by Li et al. \cite{li2024weightedsolutionsrandomtime} and Li and Fan \cite{li2024weightedlppgeq1solutionsrandom}, where the authors provide existence and uniqueness results for $\mathbb{L}^p$-solutions of BSDEs \eqref{basic BSDE Intro Cont} with stochastic monotonic coefficients. However, their setting is more regular and easier to handle compared to our case, since it does not include the jump component considered here.  Additionally, BSDEs with stochastic coefficients have been extensively studied in the seminal works of El Karoui and Huang \cite{ElKarouiHuang1997}, Bender and Kohlmann \cite{Bender2000BSDES}, Bahlali et al. \cite{Bahlali2004}, and more recently by Elmansouri and El Otmani \cite{Badrelotmani2024} in a more general context. It is worth noting that BSDEs with stochastic coefficients naturally arise in mathematical finance, stochastic control, and various other fields, as illustrated in \cite{badrjdg,BahlaliK,briand2008bsdes,el1997backward}, among others.

The paper is organized as follows. In Section \ref{sec2}, we introduce the necessary notations and present our general stochastic assumptions on the generator $f$, including the stochastic monotonicity condition in $y$, the stochastic Lipschitz condition in $(z,u)$, a stochastic linear growth condition, and an $\mathbb{L}^p$-integrability hypothesis on the data for $p \in (1,+\infty)$. In Section \ref{sec3}, we establish important a priori estimates for the solutions. Specifically, we derive general and interesting a priori estimates for the case $p \in [2,+\infty)$, as well as for the less regular case $p \in (1,2)$, which requires the use of a generalization of Itô-Meyer's formula since the mapping $\mathbb{R}^d \ni x \mapsto |x|^p$ is not smooth. Finally, in Section \ref{sec4}, we prove the existence and uniqueness of $\mathbb{L}^p$-solutions for BSDEs \eqref{basic BSDE Intro Disc} for $p \in (1,+\infty)$ under stochastic conditions on the driver and $\mathbb{L}^p$-integrability conditions on the data. This is achieved using various approximation techniques, including localization and truncation, starting from the smooth case $p \in [2,+\infty)$ and extending to the more challenging case $p \in (1,2)$.


\section{Preliminaries}
\label{sec2}
Let us consider a complete probability space $(\Omega,\mathcal{F},\mathbb{P})$ and a deterministic terminal time $T\in(0,+\infty)$. On this space, we consider a $k$-dimensional Brownian motion $(W_t)_{t\leq T}$ and an independent standard Poisson random measure $N$ on $[0,T]\times E$, where $E\subseteq\mathbb{R}^{\ell}\setminus\{0\}$, for some $\ell\in\mathbb{N}^{+}$, is endowed with its Borel $\sigma$-field $\mathcal{E}=\mathcal{B}(E)$. 

We define $\mathbb{F}=(\mathcal{F}_t)_{t\leq T}$ as the natural filtration generated by $W$ and $N$, that is, $\mathcal{F}_t=\mathcal{F}_t^W\vee\mathcal{F}_t^N$ for every $t\in[0,T]$, where $\mathcal{F}_t^W=\sigma(W_s:s\leq t)$ and $\mathcal{F}_t^N=\sigma\left(\int_0^s\int_A N(dr,du):s\leq t, A\in\mathcal{E}\right)$. We assume throughout that $\mathbb{F}$ satisfies the usual conditions of right-continuity and completeness. The random measure $N$ has compensator $\upsilon(dt,du)=Q(du)dt$, where $Q$ is a $\sigma$-finite measure on $E$ satisfying $\int_E(1\wedge |e|^2)Q(de)<+\infty$, and such that, for every $\mathcal{G}\in\mathcal{E}$ with $Q(\mathcal{G})<+\infty$, the process $\{\widetilde{N}([0,t]\times\mathcal{G}):=(N-\upsilon)([0,t]\times\mathcal{G})\}_{t\leq T}$ is an $\mathbb{F}$-martingale. 

To simplify the notation in the proofs, for any RCLL process $\mathcal{X}=(\mathcal{X}_t)_{t\leq T}$, we write $\mathcal{X}_\ast:=\sup_{t\in[0,T]}|\mathcal{X}_t|$. We also denote by $\mathcal{T}_{0,T}$ the set of all $[0,T]$-valued $\mathbb{F}$-stopping times.
\begin{remark}
	The recent works \cite{li2024weightedlppgeq1solutionsrandom,li2024weightedsolutionsrandomtime} study $\mathbb{L}^p$-solutions of BSDEs with stochastic monotonicity coefficients in a Brownian framework. The present paper deals with a more general setting, where the filtration is generated not only by a Brownian motion but also by a Poisson random measure. This discontinuous framework creates additional difficulties in the treatment of the jump component and in the derivation of the corresponding $\mathbb{L}^p$-estimates, especially for $p\in(1,2)$.
\end{remark}

In this paper, we denote:
\begin{itemize}
	\item ${\mathcal{P}}$: the predictable $\sigma$-field on $\Omega \times [0,T] $ and
	$$\tilde{\mathcal{P}}=\mathcal{P} \otimes \mathcal{E}$$
	On $\Omega \times [0,T] \times E$ (resp. $\Omega \times [0,T]$), a function that is $\tilde{\mathcal{P}}$-measurable (resp. $\mathcal{P}$-measurable) is called predictable process.
	
	\item $G_{loc}(N)$: the set of predictable functions $U$ on  $\Omega \times [0,T] \times E$ such that for any $t \geq 0$ 
	$$
	\int_{0}^{t} \int_{E} \left(\left|U_s(e)\right|^2 \wedge \left|U_s(e)\right|\right) Q(de)ds<+\infty.
	$$
	
	\item $\mathcal{D}$ (resp. $\mathcal{D}(0,T)$): the set of all  $\bar{\mathcal{P}}$-measurable (resp. ${\mathcal{P}}$-measurable) processes.
	
	\item $L^2_{loc}(W)$: the subspace of $\mathcal{D}$ of processes $Z$ such that for any $t \geq 0$ a.s.
	$$
	\int_{0}^{t} \left|Z_s\right|^2 ds<+\infty.
	$$
	
	\item $\mathbb{L}^p_Q$: the set of measurable function $\psi : E \rightarrow \mathbb{R}^d$ such that 
	$$
	\left\|\psi \right\|^p_{\mathbb{L}^p_Q}= \int_{{E}} \left|\psi(e)\right|^p Q(de) <+\infty.
	$$ 
\end{itemize}
As the filtration $\mathbb{F}$ is generated by the Brownian motion $W$ and the integer valued random measure $N$, we have the following predictable representation property (see \cite[Chapter II. Section 2.4]{Delong2013}).
\begin{theorem}\label{representation thm}
	Any $\mathbb{F}$-local martingale has the decomposition
	$$
	\int_{0}^{\cdot} Z_s dW_s+\int_{0}^{\cdot}\int_{E} U_s(e)\tilde{N}(ds,de)
	$$
	where $Z \in L^2_{loc}(W)$ and $U \in G_{loc}(N)$.
\end{theorem}
\begin{remark}
	Throughout this paper, $\mathfrak{C}$ will denote a positive constant that may vary from one line to another. Moreover, the notation $\mathfrak{C}_\gamma$ will be used to emphasize the dependence of $\mathfrak{C}$ on a specific set of parameters $\gamma$.
	
	We shall also use the notation $\beta_\gamma^\ast$ to denote a positive constant depending on a set of parameters $\gamma$, which will often correspond to the constants appearing in the assumptions. This constant is chosen large enough to ensure all the estimates involving the exponential weight. Thus, for a given set of parameters $\gamma$, whenever a result is stated for $\beta\geq \beta_\gamma^\ast$, all lower bounds on $\beta$ required in the proof are included in the choice of $\beta_\gamma^\ast$.
\end{remark}

In contrast with the continuous case in a Brownian filtration \cite{Briand2003,fan2012p,li2024weightedlppgeq1solutionsrandom,li2024weightedsolutionsrandomtime}, and as explained in \cite{Kruse2017}, the $\mathbb{L}^p$-estimates for $p\in(1,2)$ become more delicate in the presence of Poisson jumps. This requires the introduction of a different functional space for the jump component $u$, which is related to, but not exactly the same as, the one used in the case $p\geq2$. We recall this space below. The main reason for this change is that, when $p\in(1,2)$, the compensator of a martingale does not, in general, control the predictable projection; see \cite{SPS1} for a counterexample. This difficulty is also highlighted in \cite{Kruse2017}, where the use of the same functional space for both cases $p\geq2$ and $p\in(1,2)$ leads to uncontrolled terms in the latter case, as occurred in \cite{Kruse2015} and was subsequently corrected in \cite{Kruse2017}.

For a measurable function $\psi:E\to\mathbb{R}$ and $p\in[1,2)$, we define, following the standard construction of the sum of two Banach spaces as in \cite{Krein1982},
$$
\|\psi\|_{\mathbb{L}^p_Q+\mathbb{L}^2_Q}
=
\inf_{\psi^1+\psi^2=\psi}
\left(
\|\psi^1\|_{\mathbb{L}^p_Q}
+
\|\psi^2\|_{\mathbb{L}^2_Q}
\right),
$$
where the infimum is taken over all admissible decompositions $\psi=\psi^1+\psi^2$ with
$(\psi^1,\psi^2)\in\mathbb{L}^p_Q\times\mathbb{L}^2_Q$. The space $\mathbb{L}^p_\nu+\mathbb{L}^2_\nu$ is defined in the same way, by replacing $Q$ with the product measure $Q\otimes \mathrm{Leb}$ on $[0,T]\times E$, where $\mathrm{Leb}$ denotes the Lebesgue measure on $[0,T]$, that is, $d\mathrm{Leb}(t)=dt$.

By \cite[Remark 1]{Kruse2017}, when $p\geq2$, we have
$\mathbb{L}^p_Q+\mathbb{L}^2_Q\subset\mathbb{L}^2_Q$. Hence, in this case, the appropriate function space for the jump component is $\mathbb{L}^2_Q$, which is the standard choice in the literature on $\mathbb{L}^p$-solutions of BSDEs with jumps; see, for instance, \cite{eddahbi2017,Kruse2015}. On the other hand, for $p\geq1$, and in particular for $p\in(1,2)$, it follows from \cite[Lemma 3]{Kruse2017} that
$\mathbb{L}^p_Q+\mathbb{L}^2_Q\subset\mathbb{L}^1_Q+\mathbb{L}^2_Q$ (see also Remark \ref{Matlo rmk} below). Therefore, for $p\in(1,2)$, a suitable integrability space for the jump component $u$ is $\mathbb{L}^1_Q+\mathbb{L}^2_Q$, which contains all the spaces $\mathbb{L}^p_Q+\mathbb{L}^2_Q$.

In order to treat the cases $p\in(1,2)$ and $p\geq2$ simultaneously, we introduce the following $p$-dependent space for the jump component:
$$
\mathscr{L}^p_Q
:=
\begin{cases}
	\mathbb{L}^1_Q+\mathbb{L}^2_Q, & \text{if } p\in(1,2),\\[0.2cm]
	\mathbb{L}^2_Q, & \text{if } p\geq2.
\end{cases}
$$
We endow it with the norm
$$
\|u\|_{\mathscr{L}^p_Q}
:=
\begin{cases}
	\|u\|_{\mathbb{L}^1_Q+\mathbb{L}^2_Q}, & \text{if } p\in(1,2),\\[0.2cm]
	\|u\|_{\mathbb{L}^2_Q}, & \text{if } p\geq2.
\end{cases}
$$

Let $p>1$. We consider the following BSDE:
\begin{equation}\label{basic BSDE}
	Y_t=\xi+\int_{t}^{T}f(s,Y_s,Z_s,U_s)ds-\int_{t}^T Z_s d W_s-\int_{t}^{T}\int_{E}U_s(e) \tilde{N}(ds,de),\quad t \in [0,T].
\end{equation}
Here,
\begin{itemize}
	\item $\xi$ is an $\mathcal{F}_T$-measurable random variable,
	\item $f : \Omega \times [0,T] \times \mathbb{R}^d \times \mathbb{R}^{d \times k} \times \mathscr{L}^p_Q$ is an $\mathcal{F} \otimes \mathcal{B}([0,T])\otimes \mathcal{B}(\mathbb{R}^d)\otimes \mathcal{B}(\mathbb{R}^{d \times k})\otimes \mathcal{B}(\mathscr{L}^p_Q)$ measurable random function such that for any $(y,z,u) \in \mathbb{R}^d \times \mathbb{R}^{d \times k } \times \mathscr{L}^p_Q$, the process $(\omega,t) \mapsto f(\omega,t,y,z,u)$ is progressively measurable. 
\end{itemize} 

Let $\beta>0$ and $(\zeta_t)_{t \leq T}$ an progressively measurable positive process. Then, we consider the increasing continuous process $A_t:=\int_{0}^{t}\zeta_s^2 ds$ for $t \in [0,T]$. To define the $\mathbb{L}^p$-solution of our BSDE \eqref{basic BSDE} for $p>1$, we need to introduce the following spaces:
\begin{itemize}
	\item $\mathcal{S}^p_{\beta}$: The space of $\mathbb{R}^d$-valued and $\mathcal{F}_{t}$-adapted RCLL processes $(Y_{t})_{t\leq T}$ such that
	$$
	\|Y\|_{\mathcal{S}^{p}_\beta}=\left(\mathbb{E}\left[\sup\limits_{t \in [0,T]}e^{\left(\frac{p}{2} \wedge 1\right)\beta A_t}|Y_{t}|^{p}\right]\right)^{\frac{1}{p}} < +\infty.
	$$
	
	\item[$\bullet$] $\mathcal{S}^{p,A}_\beta$ is the space of $\mathbb{R}^d$-valued and $\mathcal{F}_{t}$-adapted RCLL processes $(Y_{t})_{t\leq T}$ such that
	$$
	\|Y\|_{\mathcal{S}^{p,A}_\beta}=\left(\mathbb{E}\left[\int_{0}^{T}e^{\left(\frac{p}{2} \wedge 1\right)\beta A_s}|Y_{s}|^{p}dA_s\right]\right)^\frac{1}{p} < +\infty.
	$$
	
	\item[$\bullet$] $\mathcal{H}^{p}_\beta$ is the space of $\mathbb{R}^{d\times k}$-valued and $\mathcal{P}$-measurable processes $(Z_{t})_{t\leq T}$ such that
	$$\|Z\|_{\mathcal{H}^{p}_\beta}=\left(\mathbb{E}\left[\left(\int_{0}^{T}e^{\beta A_s}|Z_{s}|^{2}ds\right)^\frac{p}{2}\right]\right)^\frac{1}{p} < +\infty.$$
	
	\item[$\bullet$] $\mathcal{L}^{p}_{Q,\beta}$ is the space of $\mathbb{R}^d$-valued and $\tilde{\mathcal{P}}$-measurable processes $(U_t)_{t\leq T}$ such that
	$$
	\|U\|_{\mathcal{L}^{p}_{Q,\beta}}=\left(\mathbb{E}\left[\left(\int_0^Te^{\beta A_s}\|U_s\|_{\mathbb{L}^2_Q}^2ds\right)^\frac{p}{2}\right]\right)^\frac{1}{p} < +\infty.
	$$
	
	\item[$\bullet$] $\mathcal{L}^{p}_{N,\beta}$ is the space of $\mathbb{R}^d$-valued and $\mathcal{P}\otimes \mathbb{U}$-predictable processes $(U_t)_{t\leq T}$ such that
	$$
	\|U\|_{\mathcal{L}^{p}_{N,\beta}}=\left(\mathbb{E}\left[\left(\int_0^T e^{\beta A_s} \int_{{E}}|U_s(e)|^2N(ds,de)\right)^\frac{p}{2}\right]\right)^\frac{1}{p} < +\infty.
	$$
	Finally, we define 
	\item[$\bullet$] $\mathfrak{B}^p_\beta:=\mathcal{S}^{p}_\beta\cap \mathcal{S}^{p,A}_\beta$ endowed with the norm
	$\|Y\|_{\mathfrak{B}^{p}_\beta}^p=\|Y\|_{\mathcal{S}^{p}_\beta}^p+\|Y\|_{\mathcal{S}^{p,A}_\beta}^p.$

	\item $\mathfrak{L}^{p}_{\beta}:=\mathcal{L}^{p}_{Q,\beta} \cup \mathcal{L}^{p}_{N,\beta}$ endowed with the norm
	$$
	\|U\|_{\mathfrak{L}^{p}_{\beta}}^p
	:=
	\begin{cases}
		\|U\|_{\mathcal{L}^{p}_{N,\beta}}^p, & \text{if } p\in(1,2),\\[0.2cm]
		\|U\|_{\mathcal{L}^{p}_{Q,\beta}}^p, & \text{if } p\geq2.
	\end{cases}
	$$
	
	\item $\mathcal{E}^p_\beta(0,T):=\mathfrak{B}^p_\beta \times \mathcal{H}^p_\beta \times \mathfrak{L}^{p}_{\beta}$ endowed with the norm $ \|(Y,Z,U)\|^p_{\mathcal{E}^p_\beta(0,T)}:=\|Y\|_{\mathfrak{B}^{p}_\beta}^p+\|Z\|^p_{\mathcal{H}^{p}_\beta}+\|U\|_{\mathfrak{L}^{p}_\beta}^p$. 
\end{itemize}

\paragraph{Conditions on the data $(\xi,f)$.} Let $p>1$. 
Throughout the rest of this paper, we consider the following conditions on the terminal value $\xi$ and the generator $f$: 
\begin{itemize}
	
	\item[\textsc{(H1)}] For all $(t,z,u)\in [0,T]\times \mathbb{R}^{d\times k}\times \mathscr{L}^p_Q$, the mapping $y\longmapsto f(t,y,z,u)$ is continuous, $\mathbb{P}$-a.s.
	
	\item[\textsc{(H2)}] There exists an $\mathbb{F}$-progressively measurable process 
	$\alpha:\Omega\times[0,T]\rightarrow\mathbb{R}$ such that, for all 
	$t\in[0,T]$, $y,y'\in\mathbb{R}^d$, $z\in\mathbb{R}^{d\times k}$ and 
	$u\in \mathscr{L}^p_Q$, $d\mathbb{P}\otimes dt$-a.e.,
	$$
	(y-y') (f(t,y,z,u)-f(t,y',z,u)) \leq
	\alpha_t |y-y'|^2.
	$$
	
	\item[\textsc{(H3)}] There exist two non-negative $\mathbb{F}$-progressively measurable processes 
	$\eta,\delta:\Omega\times[0,T]\to\mathbb{R}_{+}$ such that, for all 
	$y\in\mathbb{R}^d$, $z,z'\in\mathbb{R}^{d\times k}$, and $u,u'\in\mathscr{L}^p_Q$, $d\mathbb{P}\otimes dt$-a.e.,
	$$
	\left| f(t,y,z,u)-f(t,y,z',u')\right|
	\leq
	\eta_t \left\|z-z'\right\|
	+
	\delta_t \left\|u-u'\right\|_{\mathscr{L}^p_Q}.
	$$
	Equivalently,
	\begin{itemize}
		\item[(i)] if $p\in(1,2)$, then, for all $u,u'\in\mathbb{L}^1_Q+\mathbb{L}^2_Q$,
		$$
		\left| f(t,y,z,u)-f(t,y,z',u')\right|
		\leq
		\eta_t \left\|z-z'\right\|
		+
		\delta_t \left\|u-u'\right\|_{\mathbb{L}^1_Q+\mathbb{L}^2_Q},
		$$
		$d\mathbb{P}\otimes dt$-a.e.;
		
		\item[(ii)] if $p\geq2$, then, for all $u,u'\in\mathbb{L}^2_Q$,
		$$
		\left| f(t,y,z,u)-f(t,y,z',u')\right|
		\leq
		\eta_t \left\|z-z'\right\|
		+
		\delta_t \left\|u-u'\right\|_{\mathbb{L}^2_Q},
		$$
		$d\mathbb{P}\otimes dt$-a.e.
	\end{itemize}
	
	\item[\textsc{(H4)}] There exist two progressively measurable processes 
	$\varphi:\Omega\times[0,T]\rightarrow[1,+\infty)$ and 
	$\phi:\Omega\times[0,T]\rightarrow(0,+\infty)$ such that, in particular, 
	$\varphi\geq 1$, and, for all $y\in\mathbb{R}^d$, $d\mathbb{P}\otimes dt$-a.e.,
	$$
	\left|f(t,y,0,0)\right|
	\leq
	\varphi_t+\phi_t |y|.
	$$
	
	\item[\textsc{(H5)}] We set $a_s^2=\phi_s+\eta_s^2+\delta_s^2$, where $(a_s)_{s\leq T}$ is a positive process,\footnote{The positivity of $a$ is assumed only to simplify the notation. Otherwise, one may replace $a_\cdot$ by $|a_\cdot|$, and the same arguments remain valid.} and define
	$\zeta_s^2=a_s^2\mathds{1}_{\{p\in[2,+\infty)\}}+a_s^q\mathds{1}_{\{p\in(1,2)\}}$, with $q=\frac{p}{p-1}$. The increasing process $A$ is then defined by
	$A_t=\int_0^t \zeta_s^2 ds$, $t\in[0,T]$. Moreover, we assume that there exists a constant $\epsilon>0$ such that $a_s^2\geq\epsilon$ for all $s\in[0,T]$.
	
	\item[\textsc{(H6)}] The terminal condition $\xi$ and the process $\varphi$ satisfy
	$$
	\mathbb{E}\left[
	e^{\left((p-1)\vee \frac{p}{2}\right)\beta A_T}
	|\xi|^p
	\right]
	<+\infty
	$$
	and
	$$
	\mathbb{E}\int_0^T
	e^{((p-1)\vee 1)\beta A_s}
	|\varphi_s|^p ds
	<+\infty.
	$$
\end{itemize}

The following remark clarifies the rationale for selecting the component $U$ in the space $\mathfrak{L}^{p}_{\beta}$ with respect to the associated norm $\|\cdot\|_{\mathfrak{L}^p_\beta}$.
\begin{remark}\label{rmq p}
	Let $\beta>0$ and $p>1$. Consider $U \in G_{loc}(N)$ and $\mathfrak{M}_t:=\int_{0}^{t} \int_E U_s(e)\tilde{N}(ds,de)$ for $t \in [0,T]$. Note that $\mathfrak{M}$ is an RCLL local martingale (see \cite[Chapter XI. Theorem 21]{he1992}). Moreover, the process $\mathcal{A}_t:=\int_{0}^{t}e^{\beta A_s} \int_E \left|U_s(e)\right|^2{N}(ds,de)$ is an RCLL locally integrable increasing process. Indeed, as $A$ is continuous by definition, for $\mathbb{P}$-almost every $\omega \in \Omega$, the mapping $[0,T] \ni t \mapsto A_t(\omega)$ is bounded. Then, we can consider $\sigma_n:=\inf\left\{ t \geq 0 : 	A_t \geq n \right\} \wedge T$, which is an increasing sequence of stopping times such that $\sigma_n \nearrow T$ a.s. Next, we define $\left\{\tau_n\right\}_{n \geq 1}$ to be the sequence of random times defined by 
	$$
	\tau_n=\inf\left\{ t \geq 0 : 	\int_{0}^{t} \int_{E} \left\|U_s\right\|^2_{\mathbb{L}^2_Q} ds \geq n \right\} \wedge T.
	$$
	Clearly, $\left\{\tau_n\right\}_{n \geq 1}$ is an increasing sequence of stopping times such that $\tau_n \nearrow T$ a.s. as $n \rightarrow +\infty$ since $U \in G_{loc}(N)$. Additionally,  
	the stopped process $\mathfrak{M}^{\tau_n}:=\big(\mathfrak{M}_{t \wedge \tau_n}\big)_{t \leq T}$ is a square-integrable martingale (see \cite[Chapter I. Theorem  4.40]{jacodshiryaev2003}). By setting $\widehat{\tau}_n:=\tau_n \wedge \sigma_n$, we have 
	$$
	\mathbb{E}\left[ \mathcal{A}_{t \wedge \widehat{\tau}_n}\right]  = \mathbb{E}\int_{0}^{t \wedge \widehat{\tau}_n}e^{\beta A_s} \left\|U_s\right\|^2_{\mathbb{L}^2_Q} ds \leq n e^{\beta n}.
	$$
	Finally, as the filtration $\mathbb{F}$ satisfies the usual condition, we are whiten the framework of Section 4 in \cite{SPS1}. Therefore;
	\begin{itemize}
		\item In the case $p \geq 2$,
		\begin{itemize}
			\item[$\ast$]  We have $\frac{p}{2} \geq 1$, then using the convexity of the function $\mathbb{R}^+ \ni x \overset{F}{\mapsto} x^{p/2}$, the fact that $\sup_{x >0} \frac{x F^{\prime}(x)}{F(x)}=\frac{p}{2}$, and \cite[Theorem 4.1-(1)]{SPS1}, we get 
			\begin{equation}\label{CE}
				\mathbb{E}\left[\left(\int_{0}^{T} e^{\beta A_s}\big\|{U}_s\big\|^2_{\mathbb{L}^2_Q}ds\right)^{\frac{p}{2}}\right] \leq \left(\frac{p}{2}\right)^{\frac{p}{2}} \mathbb{E}\left[\left(\int_{0}^{T} e^{\beta A_s}  \int_{E} \big| {U}_s(e)\big|^2  N(ds,de)\right)^{\frac{p}{2}}\right].
			\end{equation}
		Therefore, $\mathcal{L}^{p}_{N,\beta} \subset \mathcal{L}^{p}_{Q,\beta}$ and then $\mathfrak{L}^{p}_{\beta}=\mathcal{L}^{p}_{Q,\beta}$. In particular, for $p=2$ we have the standard equality
		$$\mathbb{E}\left[\int_0^Te^{\beta A_s}\|U_s\|_{\mathbb{L}^2_Q}^2dt\right]=\mathbb{E}\left[\int_0^T\int_{E}e^{\beta A_s}|U_s(e)|^2N(ds,de)\right],$$
		since $\mathfrak{M}$ is a square-integrable martingale. In this specific case, which is frequently encountered in the literature on BSDEs, we have $\mathfrak{L}^{2}_{\beta}=\mathcal{L}^{2}_{N,\beta}=\mathcal{L}^{2}_{Q,\beta}$.
			
			\item[$\ast$] There exists two universal constants $\mathfrak{c}_p$ and $\mathfrak{C}_p$ such that (see \cite[Chapter IV. Theorem 48 ]{Protter2004})
					\begin{equation*}
						\begin{split}
							\mathfrak{c}_p\mathbb{E}\left[\left(\int_{0}^{T} e^{\beta A_s}  \int_{E} \big| {U}_s(e)\big|^2  N(ds,de)\right)^{\frac{p}{4}}\right] &\leq  \mathbb{E}\left[\sup_{0\leq t\leq T} \left|\mathfrak{M}_t\right|^\frac{p}{2}\right]\\ 
							&\leq  \mathfrak{C}_p \mathbb{E}\left[\left(\int_{0}^{T} e^{\beta A_s}  \int_{E} \big| {U}_s(e)\big|^2  N(ds,de)\right)^{\frac{p}{4}}\right].
							\end{split}
					\end{equation*}
		\end{itemize}
	
	\item In the case $p \in  (1,2)$;
	\begin{itemize}
			\item[$\ast$] The compensator of the martingale $\mathfrak{M}$ cannot be controlled by it's quadratic variation, meaning that \eqref{CE} does not holds in general. More precisely, using this time the concavity of the function $\mathbb{R}^+ \ni x \overset{F}{\mapsto} x^{p/2}$ and \cite[Theorem 4.1-(2)]{SPS1} and that $U \in \mathcal{L}^{p}_{N,\beta}$, we get 
		\begin{equation*}\label{E}
			\mathbb{E}\left[\left(\int_{0}^{T} e^{\beta A_s}  \int_{E} \big| {U}_s(e)\big|^2  N(ds,de)\right)^{\frac{p}{2}}\right] \leq 2 \mathbb{E}\left[\left(\int_{0}^{T} e^{\beta A_s}\big\|{U}_s\big\|^2_{\mathbb{L}^2_Q}ds\right)^{\frac{p}{2}}\right].
		\end{equation*}
		Therefore, $\mathcal{L}^{p}_{Q,\beta} \subset \mathcal{L}^{p}_{N,\beta} $ and then $\mathfrak{L}^{p}_{\beta}=\mathcal{L}^{p}_{N,\beta}$.
		
	To justify this inequality, we note that, for $p\in(1,2)$, we have $\frac{p}{2}\in(0,1)$, and the function $F(x)=x^{\frac{p}{2}}$, $x\geq0$, is concave. Set
	$$ 
	B_t:=\int_0^t e^{\beta A_s}\int_E |U_s(e)|^2 N(ds,de),\qquad t\in[0,T]. 
	$$
	Then, since $U$ is predictable and $A$ is continuous and adapted, the dual predictable projection of $B$ is given by
	$$ 
	{P}_t:=\int_0^t e^{\beta A_s}\|U_s\|^2_{\mathbb{L}^2_Q}ds,\qquad t\in[0,T]. 
	$$
	Moreover, since $U \in \mathcal{L}^p_{N,\beta}$, the process $B$ is integrable, and hence locally integrable. Thus, the hypotheses of \cite[Theorem 4.1-(2)]{SPS1} are satisfied. Therefore, we obtain
	$\mathbb{E}\left[F(B_T)\right] \leq 2 \mathbb{E}\left[F({P}_T)\right]$,
	which yields the desired estimate.
		
			\item[$\ast$] The BDG inequality is not applicable to the $\frac{p}{2}$-th power of the martingale $\mathfrak{M}$ since $\frac{p}{2} < 1$, which poses a difficulty in deriving the $\mathbb{L}^p$ estimate of the solution for BSDE (\ref{basic BSDE}) when the generator $f$ depends on the $u$-variable under the \textsc{(H3)} condition.
	\end{itemize}
	
	\end{itemize}
\end{remark}
Note that $\mathfrak{B}^p_\beta:=\mathcal{S}^{p}_\beta\cap\mathcal{S}^{p,A}_\beta$ is a Banach space endowed with the norm $\|Y\|_{\mathfrak{B}^{p}_\beta}^p=\|Y\|_{\mathcal{S}^{p}_\beta}^p+\|Y\|_{\mathcal{S}^{p,A}_\beta}^p$; moreover, we define $\mathfrak{L}^{p}_{\beta}:=\mathcal{L}^{p}_{N,\beta}$ for $p\in(1,2)$ and $\mathfrak{L}^{p}_{\beta}:=\mathcal{L}^{p}_{Q,\beta}$ for $p\geq2$, endowed with the corresponding norm $\|U\|_{\mathfrak{L}^{p}_{\beta}}^p:=\|U\|_{\mathcal{L}^{p}_{N,\beta}}^p$ if $p\in(1,2)$ and $\|U\|_{\mathfrak{L}^{p}_{\beta}}^p:=\|U\|_{\mathcal{L}^{p}_{Q,\beta}}^p$ if $p\geq2$; with these definitions, $\mathcal{E}^p_\beta(0,T)$ is also a Banach space for every $p>1$.

To derive optimal constants in the a priori estimates of the solutions, the following remark is used. 
\begin{remark}\label{rmq essential}
	Let $(Y,Z,U)$ be a solution of the BSDE (\ref{basic BSDE}) associated with $(\xi, f)$. For any $\varepsilon \geq 0$, we set $\Gamma^{(\alpha),\varepsilon}_t:=\exp\left\{\int_{0}^{t} \alpha_s ds+\varepsilon \int_{0}^{t}a^2_s ds +\varepsilon t\right\}$, then we define 
	\begin{equation*}
		\begin{split}
			\hat{Y}_t:=\Gamma^{(\alpha),\varepsilon}_t Y_t,\quad
			\hat{Z}_t:=\Gamma^{(\alpha),\varepsilon}_t Z_t, \quad
			\hat{U}_t:=\Gamma^{(\alpha),\varepsilon}_t U_t.
		\end{split}
	\end{equation*}
	Using an integration by parts formula, we have
	\begin{equation*}
		\begin{split}
			\hat{Y}_t=&\hat{\xi}+\int_{t}^{T}\hat{f}\big(s,\hat{Y}_s,\hat{Z}_s,\hat{U}_s\big) ds-\int_{t}^T \hat{Z}_s d W_s-\int_{t}^{T}\int_{E}\hat{U}_s(e) \tilde{N}(ds,de),
		\end{split}
	\end{equation*}
	with
	\begin{equation*}
		\begin{split}
			\hat{\xi}=\Gamma^{(\alpha),\varepsilon}_T \xi, \quad
			\hat{f}(t,y,z,u)&=\Gamma^{(\alpha),\varepsilon}_t f\left(t,\Gamma^{(-\alpha),-\varepsilon}_t y, \Gamma^{(-\alpha),-\varepsilon}_t z,\Gamma^{(-\alpha),-\varepsilon}_t u\right)-(\alpha_t+\varepsilon a^2_t+\varepsilon) y.
		\end{split}
	\end{equation*}
Thus, whenever $(Y,Z,U)$ is a solution of the BSDE (\ref{basic BSDE}) associated with $(\xi, f)$, the process $(\widehat{Y},\widehat{Z},\widehat{U})$ satisfies a similar BSDE associated with $(\widehat{\xi},\widehat{f})$.

Moreover, this change of variables is compatible with the functional space used in the sequel. Indeed, under the integrability condition imposed on the process defining $A$ in \textsc{(H5)}, the exponential factor $\Gamma^{(\alpha),\varepsilon}$ and its inverse are bounded on $[0,T]$, $\mathbb{P}$-a.s. (see also Remark \ref{rmq p}). Hence, multiplication by $\Gamma^{(\alpha),\varepsilon}$ preserves membership in the corresponding space $\mathcal{E}^p_\beta(0,T)$ on localized intervals. In the general case, as explained below and in accordance with the reduction used in \cite[Remark 2]{ELM2026}, the required integrability condition after the change of variables is imposed explicitly.

The driver $\widehat{f}$ retains the stochastic Lipschitz condition described in \textsc{(H3)} with the same stochastic processes $(\eta_t)_{t \leq T}$ and $(\delta_t)_{t \leq T}$. However, if $f$ satisfies \textsc{(H2)}, then $\widehat{f}$ satisfies a similar monotonicity condition with the real-valued stochastic process
$
\widehat{\alpha}_t=-\varepsilon a_t^2-\varepsilon.
$
Consequently, for any $t\in[0,T]$ and each $\varepsilon\geq0$, we have
$
\widehat{\alpha}_t+\varepsilon a_t^2=-\varepsilon\leq0.
$
Therefore, this change of variables reduces the problem to the case where
$
\alpha_t+\varepsilon a_t^2\leq0$, for $t\in[0,T].
$
The reverse transformation is carried out in a similar manner using an appropriate change of variable; see also Remarks 2 and 3 in \cite{ELM2026}.
\end{remark}

In order to facilitate the calculations, and in view of the change of variables described in Remark \ref{rmq essential}, following the reduction used in \cite[Remark 2]{ELM2026}, we shall work, in the remainder of the paper, under the following convention. For the fixed parameter $\varepsilon \geq 0$ used in the estimates, we assume that condition \textsc{(H2)} is satisfied with a process $(\alpha_t)_{t\leq T}$ such that
$$
\alpha_t+\varepsilon a_t^2 \leq 0,\qquad t\in[0,T].
$$
In particular, under this convention, the generator satisfies the required non-increasing condition in the $y$-variable, in the sense that the monotonicity contribution appearing in the a priori estimates has a non-positive sign. More precisely, after applying the above change of variables and keeping the same notation for the transformed generator, we have, for all $t\in[0,T]$, $y,y'\in\mathbb{R}^d$, $z \in  \mathbb{R}^{d\times k}$, $u \in \mathscr{L}^p_Q$, $d\mathbb{P}\otimes dt$-a.e.,
$$
( y-y') (f(t,y,z,u)-f(t,y',z,u)) \leq 0.
$$

If this is not the case, the same change of variables as the one introduced in Remark \ref{rmq essential} may be applied in order to reduce the problem to this situation. Moreover, the condition involving the processes $(\eta_t)_{t\leq T}$ and $(\delta_t)_{t\leq T}$ remains valid with the same processes after this transformation.

We emphasize, however, that this change of variables modifies the exponential weights appearing in the integrability assumptions on the terminal condition and on the free term of the generator. Consequently, in the general stochastic case, the corresponding integrability conditions cannot be deduced from the original assumptions. Therefore, throughout the sequel, whenever this reduction is invoked, assumption \textsc{(H6)} is understood to be imposed not only on the original data, but also after the change of variables. This convention will be used without further mention in the subsequent estimates.

\begin{remark}
	When the process $A$ is bounded, the exponential factors generated by the change of variables are controlled, and the integrability conditions imposed before and after the transformation are equivalent up to multiplicative constants. In particular, one recovers the equivalence between the conditions on $\xi$ and on the transformed terminal condition $\widehat{\xi}$, and similarly for the free term of the generator. In the general case, this equivalence may fail, as illustrated by the example based on sticky reflected Brownian motion in \cite[pp. 125--126]{ELM2026}. This shows that, contrary to the classical setting with deterministic monotonicity and Lipschitz coefficients studied in \cite{Briand2003,Kruse2015,Kruse2017}, such an equivalence cannot be expected under stochastic monotonicity and stochastic Lipschitz coefficients. This explains why the integrability condition \textsc{(H6)} must be imposed explicitly after the change of variables.
\end{remark}

We end this section with the following definition:
\begin{definition}\label{Def}
	Let $p >1$. A triplet $(Y, Z, U) := (Y_t, Z_t, U_t)_{t \leq T}$ is called an $\mathbb{L}^p$-solution of the BSDE \eqref{basic BSDE} on $[0, T]$ if the following conditions are satisfied:
	\begin{itemize}
		\item $(Y, Z, U)$ satisfies \eqref{basic BSDE} $d\mathbb{P} \otimes dt$-a.s.
		\item There exists $\beta>0$ such that $(Y, Z, U)$ belongs to $\mathcal{E}^p_\beta(0, T)$.
	\end{itemize}
\end{definition}

\section{A priori estimates}
\label{sec3}
To obtain the existence and uniqueness result for $\mathbb{L}^p$-solutions, some a priori estimates are required.
\subsection{$\mathbb{L}^p$-estimates for $p \in  [2, +\infty)$}
To highlight the difference between the classical calculation techniques used in the a priori estimates widely found in the $\mathbb{L}^2$ case and the general $\mathbb{L}^p$ situation for any $p \geq 2$, as well as for some technical reasons (which will be discussed later), we first provide the a priori estimation for the case $p = 2$, followed by the case $p \in (2, +\infty)$. Recall that, in this case, $\mathscr{L}^p_Q=\mathbb{L}^2_Q$.
\subsubsection{$\mathbb{L}^p$-estimates for $p =2$}
The a priori estimates for the solutions of the BSDE \eqref{basic BSDE} in the case where $p = 2$ can be seen as a special case of the general result given in \cite[Proposition 1]{Badrelotmani2024}. By following a similar proof scheme, along with Remark \ref{rmq essential}, we can derive a similar result. Therefore, we omit the proof.

Let $(Y^1, Z^1, U^1)$ and $(Y^2, Z^2, U^2)$ be two $\mathbb{L}^2$-solutions of the BSDE \eqref{basic BSDE} associated with the parameters $(\xi_1, f_1)$ and $(\xi_2, f_2)$, respectively, satisfying assumptions \textsc{(H2)}--\textsc{(H6)}. Define $\widehat{\mathcal{R}} = \mathcal{R}^1 - \mathcal{R}^2$, where $\mathcal{R} \in \{Y, Z, U, \xi, f\}$.

\begin{proposition}\label{Propo p=2}
	For any $\beta> 0$, there exists a constant $\mathfrak{C}_{\epsilon,T}$ such that 
	\begin{equation*}
		\begin{split}
			&\mathbb{E}\left[\sup_{t \in [0,T]}e^{{\beta} A_t}\big|\widehat{Y}_{t}\big|^{2}\right]+\mathbb{E}\int_{0}^{T} e^{\beta A_s} \big|\widehat{Y}_s\big|^2 dA_s+\mathbb{E}\int_{0}^{T}e^{\beta A_s}\big\|\widehat{Z}_s\big\|^2ds+\mathbb{E}\int_{0}^{T}e^{\beta A_s} \big\|\widehat{U}_s\big\|^2_{\mathbb{L}^2_Q} ds\\
			&\leq \mathfrak{C}_{\epsilon,T} \left(\mathbb{E}e^{{\beta} A_T} \big|\widehat{\xi}\big|^2 + \mathbb{E}\int_{0}^{T}e^{{\beta} A_s} \left| \widehat{f}(s,Y^2_s,Z^2_s,U^2_s)\right|^2 ds\right).
		\end{split}
	\end{equation*}
\end{proposition}

In view of Remark \ref{rmq p}, the estimate of Proposition \ref{Propo p=2} also covers the particular case $p=2$.

\subsubsection{$\mathbb{L}^p$-estimates for $p \in  (2, +\infty)$}
Until the end of this section, we assume that $p > 2$. Before stating the a priori estimation of the $\mathbb{L}^p$-solutions, we shall need an auxiliary result borrowed from \cite[Lemma A.4]{Yao2010}. For the sake of completeness, we recall it and provide its detailed proof.

\begin{lemma}\label{**Lemma}
	Let $p \in (2,+\infty)$. For any $x,y \in \mathbb{R}^d$, we have
	$$\int_{0}^{1}(1-r)\left|x+ry\right|^{p-2}dr \geq 3^{1- p} \left|x\right|^{p-2}.$$
\end{lemma}
\begin{proof}
	For $y=0$, we have $\int_{0}^{1}(1-r)\left|x+ry\right|^{p-2}dr=\frac{1}{2}\left|x\right|^{p-2}$. As, $3^{p-2} \geq 1$, we have $3^{p-1} > 2$ then $\frac{1}{2} > 3^{1-p}$ and the result follows.\\ 
	Now, we assume that $y \neq 0$ and set $r_0:=\frac{2}{3}\frac{\left|x\right|}{\left|y\right|}$. Using the reverse triangle inequality $\left|\left|x\right|-r\left|y\right|\right|\leq \left|x+ry\right|$ and the definition of $r_0$, we get
	\begin{enumerate}
		\item \label{1}  For any $r \in [0,r_0]$, we have $r \left|y\right| \leq  r_0 \left|y\right|=\frac{2}{3}\left|x\right|$, then $r \left|y\right|-\left|x\right| \leq -\frac{1}{3}\left|x\right|$, thus $\frac{1}{3}\left|x\right| \leq \left|x\right|-r \left|y\right| \leq \left|\left|x\right|-r \left|y\right| \right| \leq \left|x+ry\right|$. 
		
		\item \label{2}   For any $r \in [2r_0,+\infty)$, we have $r \left|y\right| \geq  2 r_0 \left|y\right|=\frac{4}{3}\left|x\right|$, then $r \left|y\right|-\left|x\right| \geq \frac{1}{3}\left|x\right|$, thus $\frac{1}{3}\left|x\right| \leq r \left|y\right|-\left|x\right|\leq \left| r \left|y\right|-\left|x\right| \right| \leq \left|x+ry\right|$.
	\end{enumerate}
	From \ref{1} and \ref{2}, we derive that $3^{2-p}\left|x\right|^{p-2} \leq \left|x+ry\right|^{p-2}$ for any $ r \in [0,r_0] \cup [2r_0,+\infty)$.\\ 
	We can discuss three cases:
	\begin{itemize}
		\item If $r_0 <\frac{1}{2}$. In this case, we have 
		$
		\int_{0}^{1}(1-r)\left|x+ry\right|^{p-2}dr=\int_{0}^{r_0}(1-r)\left|x+ry\right|^{p-2}dr+\int_{r_0}^{2 r_0}(1-r)\left|x+ry\right|^{p-2}dr+\int_{2r_0}^{1}(1-r)\left|x+ry\right|^{p-2}dr\geq \int_{0}^{r_0}(1-r)\left|x+ry\right|^{p-2}dr+\int_{2r_0}^{1}(1-r)\left|x+ry\right|^{p-2}dr\geq 3^{2-p}\left|x\right|^{p-2}\left(\frac{3}{2}r_0^2-r_0+\frac{1}{2}\right)
		$. As $\frac{3}{2}r_0^2-r_0+\frac{1}{6}=\frac{3}{2}\left(r_0-\frac{1}{3}\right)^2 \geq 0$, we derive $\left(\frac{3}{2}r_0^2-r_0+\frac{1}{2}\right)\geq \frac{1}{3}$. Thus, $
		\int_{0}^{1}(1-r)\left|x+ry\right|^{p-2}dr\geq \frac{1}{3}3^{2-p}\left|x\right|^{p-2}=3^{1-p}\left|x\right|^{p-1}$.
		
		\item If $\frac{1}{2} \leq r_0 <1$. In this case, we have $
		\int_{0}^{1}(1-r)\left|x+ry\right|^{p-2}dr \geq \int_{0}^{r_0}(1-r)\left|x+ry\right|^{p-2}dr \geq \frac{1}{2} 3^{2-p}\left|x\right|^{p-2}\geq  3^{1-p}\left|x\right|^{p-2}$.
		
		\item If $r_0 \geq 1$. In this case, from \eqref{1}, we have $\int_{0}^{1}(1-r)\left|x+ry\right|^{p-2}dr \geq 3^{1-p}\left|x\right|^{p-2}$.
	\end{itemize}
	Completing the proof.
\end{proof}


Now, we are ready to provide the main result of this part.\\
Let $(Y^1, Z^1, U^1)$ and $(Y^2, Z^2, U^2)$ be two $\mathbb{L}^p$-solutions of the BSDE \eqref{basic BSDE} associated with the parameters $(\xi_1, f_1)$ and $(\xi_2, f_2)$, respectively, satisfying assumptions \textsc{(H2)}-\textsc{(H6)}. Define $\widehat{\mathcal{R}} = \mathcal{R}^1 - \mathcal{R}^2$, where $\mathcal{R} \in \{Y, Z, U, \xi, f\}$.

There is no loss of generality in assuming that $f_1$ and $f_2$ satisfy these assumptions with the same regularity coefficients. Indeed, if the two generators satisfy the assumptions (H2)--(H4) with different coefficients $(\eta^i_t,\delta^i_t,\phi^i_t)_{t \leq T}$ for $i \in \{1,2\}$, one may replace these coefficients by their corresponding maxima, for instance
$$
\eta_t=\eta_t^1\vee\eta_t^2,\qquad 
\delta_t=\delta_t^1\vee\delta_t^2,\qquad 
\phi_t=\phi_t^1\vee\phi_t^2,
$$
and similarly for the other coefficients. This provides a common family of coefficients for both generators. The only price to pay is that the associated increasing process $A$ may become larger, so that the integrability assumptions on $\xi_1$, $\xi_2$ and the corresponding free terms have to be understood with respect to this common exponential weight.

\begin{proposition}\label{Propo p sup a 2}
	Let $p \in (2,+\infty)$. For any ${\beta} >0$, there exists a constants $\mathfrak{C}_{p,\epsilon,T}$ such that  
	\begin{equation}\label{Verified}
		\begin{split}
			&\mathbb{E}\left[\sup_{t \in [0,T]}e^{{\beta} A_t}\big|\widehat{Y}_{t}\big|^{p}\right]+\mathbb{E}\int_{0}^{T} e^{\beta A_s} \big|\widehat{Y}_s\big|^p dA_s\\
			&\leq \mathfrak{C}_{p,\epsilon,T} \left(\mathbb{E}\left[ e^{{\beta} A_T} \big|\widehat{\xi}\big|^p\right]  + \mathbb{E}\int_{0}^{T}e^{{\beta} A_s} \left| \widehat{f}(s,Y^2_s,Z^2_s,U^2_s)\right|^p ds\right).
		\end{split}
	\end{equation}
	Moreover, for any ${\beta}>0$ such that \eqref{Verified} holds, we have
	\begin{equation*}
		\begin{split}
			&\mathbb{E}\left[\left(\int_{0}^{T}e^{\beta A_s}\big\|\widehat{Z}_s\big\|^2ds\right)^{\frac{p}{2}}\right]+\mathbb{E}\left[\left(\int_{0}^{T}e^{\beta A_s} \big\|\widehat{U}_s\big\|^2_{\mathbb{L}^2_Q} ds\right)^{\frac{p}{2}}\right]+\mathbb{E}\left[\left(\int_{0}^{T}e^{\beta A_s} \int_{E}\big|\widehat{U}_s(e)\big|^2N(ds,de)\right)^{\frac{p}{2}}\right]\\
			&\leq \mathfrak{C}_{p,\epsilon,T}  \left(\mathbb{E}\left[ e^{(p-1){\beta} A_T} \big|\widehat{\xi}\big|^p\right]  + \mathbb{E}\int_{0}^{T}e^{(p-1){\beta} A_s} \left| \widehat{f}(s,Y^2_s,Z^2_s,U^2_s)\right|^p ds\right).
		\end{split}
	\end{equation*}
\end{proposition}
\begin{proof}
	For $p > 2$, define the $\mathcal{C}^{1,2}$-function $(t,x) \mapsto \Theta(t,y)$  on $\left[0,T\right] \times \mathbb{R}^d$ by $\Theta(t,y)=e^{\beta A_t}\theta(y)$ and $\theta(y):=\left|y\right|^p$. Writing $\left|y\right|^p =\left(\left|y\right|^{2}\right)^{\frac{p}{2}}$ (to derive the partial derivation w.r.t. $y_j$), we point out that 
	$$
	\dfrac{\partial \Theta}{\partial t}(t,y)=\beta e^{\beta A_t}\left|y\right|^p,\quad \frac{\partial \Theta}{\partial y_j}(t,y)=pe^{\beta A_t}\left|y\right|^{p-2}y_j. 
	$$
	Similarly, with the expression $\left|y\right|^{p-2} =\left(\left|y\right|^{2}\right)^{\frac{p-2}{2}}$, we have
	$$
	\frac{\partial^2 \Theta}{\partial y_j \partial y_i}(t,y)=pe^{\beta A_t}\left|y\right|^{p-2} \mathbb{I}_{i,j}+p(p-2)e^{\beta A_t}y_j y_i\left|y\right|^{p-4}.
	$$
	Thus
	\begin{equation*}
		\begin{split}
			\nabla \Theta(t,y)&=e^{\beta A_t} y \left|y\right|^{p-2}=e^{\beta A_t} \nabla \theta(y),\\
			\quad D^2 \Theta(t,y)&=pe^{\beta A_t}\left|y\right|^{p-2}I+p(p-2)e^{\beta A_t}\left|y\right|^{p-2} y \otimes y=e^{\beta A_t} D^2 \theta(y),
		\end{split}
	\end{equation*}
	where $I$ is the $d \times d$ identity matrix, $D^2 \Theta$ (resp. $\nabla \Theta$) is the Hessian $d \times d$-matrix (resp. gradient) of $\Theta$ w.r.t. the second variable $y$, and $y \otimes y$ is the tensor product of $y$.

	As $p > 2$, we can apply It\^o's formula (see, e.g., \cite[Theorem 5.1, pp. 66-67]{IkedaWatanabe1989} or \cite[Theorem 32, pp. 78-79]{Protter2004}) with the function $\Theta(t,y)$  to the semimartingale $\big(e^{\beta A_t+\mu t} \left|\widehat{Y}_t\right|^p\big)_{t \leq T}$ for $\beta, \mu >0$, where $Y$ is the state processes given by the BSDE \eqref{basic BSDE} and the parameter $\mu$ to be chosen explicitly later. We then  have
	\begin{equation}\label{Ito p > 2}
		\begin{split}
			e^{\beta A_T+\mu T} \big|\widehat{\xi}\big|^p =&e^{\beta A_t+\mu t} \big|\widehat{Y}_t\big|^p-p\int_{t}^{T}e^{\beta A_s+\mu s}  \big|\widehat{Y}_s\big|^{p-2} \widehat{Y}_s \left(f_1(s,Y^1_s,Z^1_s,U^1_s)-f_2(s,Y^2_s,Z^2_s,U^2_s)\right) ds\\
			&+\beta \int_{t}^{T}e^{\beta A_s+\mu s} \big|\widehat{Y}_s\big|^p dA_s+\mu\int_{t}^{T}e^{\beta A_s+\mu s} \big|\widehat{Y}_s\big|^p ds +\frac{1}{2}\int_{t}^{T}e^{\beta A_s+\mu s} \mbox{Trace}(D^2\theta(\widehat{Y}_s)\widehat{Z}_s \widehat{Z}^\ast_s)ds\\
			&+p\int_{t}^{T}e^{\beta A_s+\mu s} \widehat{Y}_{s-} \big|\widehat{Y}_{s-}\big|^{p-2}  \widehat{Z}_s dW_s +p\int_{t}^{T}e^{\beta A_s+\mu s} \int_{E}\widehat{Y}_{s-} \big|\widehat{Y}_{s-}\big|^{p-2}  \widehat{U}_s(e) \tilde{N}(ds,de) \\
			&+\int_{t}^{T}e^{\beta A_s+\mu s}\int_{E}\left(\big|\widehat{Y}_{s-}+\widehat{U}_s(e)\big|^p-\big|\widehat{Y}_{s-}\big|^p-p\widehat{Y}_{s-}\big|\widehat{Y}_{s-}\big|^{p-2}\widehat{U}_s(e)\right)N(ds,de).
		\end{split}
	\end{equation}
	Using the assumptions (H2) and (H3) on the driver $f$ along with the basic inequality $ab \leq \frac{1}{2 \varepsilon}a^2+\frac{\varepsilon}{2}b^2$ for any $\varepsilon>0$, Cauchy-Schwarz inequality and Remark \ref{rmq essential}, we have
	\begin{equation}\label{generator estimation}
		\begin{split}
			&\widehat{Y}_s \left(f_1(s,Y^1_s,Z^1_s,U^1_s)-f_2(s,Y^2_s,Z^2_s,U^2_s)\right)  \\ 
			&\leq \big|\widehat{Y}_s \big|\left(\alpha_s\big|\widehat{Y}_s\big|+\eta_s \big\|\widehat{Z}_s \big\|+\delta_s \big\|\widehat{U}_s \big\|_{\mathbb{L}^2_Q} \right)+\frac{1}{2}\big| \widehat{f}(s,Y^2_s,Z^2_s,U^2_s) \big|^2 \\
			&\leq \left(\alpha_s+\dfrac{\eta^2_s+\delta_s^2}{2 \varepsilon}+\dfrac{1}{2}\right)\big| \widehat{Y}_s\big|^2  +\dfrac{\epsilon}{2}\left(\big\|\widehat{Z}_s \big\|^2+\big\|\widehat{U}_s \big\|^2_{\mathbb{L}^2_Q}+\left| \widehat{f}(s,Y^2_s,Z^2_s,U^2_s)\right|^2 \right)\\
			& \leq \dfrac{1}{2}\big| \widehat{Y}_s\big|^2 +\dfrac{\epsilon}{2}\left(\big\|\widehat{Z}_s \big\|^2+\big\|\widehat{U}_s \big\|^2_{\mathbb{L}^2_Q} \right)+\frac{1}{2}\left|\widehat{f}(s,Y^2_s,Z^2,U^2_s) \right|^2.
		\end{split}
	\end{equation}
For the formulation of the term
$$
\operatorname{Trace}\left(D^2\theta(\widehat{Y}_s)\widehat{Z}_s\widehat{Z}_s^\ast\right),
$$
we first define the $\mathbb{R}^{d\times d}$-valued predictable process $(\Lambda_s)_{s\leq T}$ by
$$
\Lambda_s:=\widehat{Z}_s\widehat{Z}_s^\ast.
$$
Since $e^{\beta A_s}$ is a positive real-valued random variable, it does not affect the matrix formulation. Clearly, for each $s\in[0,T]$, $\Lambda_s$ is a symmetric positive semi-definite matrix, and we have
	\begin{equation*}
		\begin{split}
			\mbox{Trace}(D^2\theta(\widehat{Y}_s)\widehat{Z}_s \widehat{Z}^\ast_s)
			&=\sum_{i,j=1}^d\ \left(p\big|\widehat{Y}_s\big|^{p-2} \mathbb{I}_{i,j}+p(p-2)\widehat{Y}^j_s \widehat{Y}^i_s\big|\widehat{Y}_s\big|^{p-4}\right)\Lambda^{j,i}_s \\
			&=p\big|\widehat{Y}_s\big|^{p-2} \mbox{Trace}(\Lambda_s)+p(p-2)\big|\widehat{Y}_s\big|^{p-4}\sum_{i,j=1}^d \widehat{Y}^i_s\Lambda^{i,j}_s \widehat{Y}^j_s \\
			&=p\big|\widehat{Y}_s\big|^{p-2} \big\|\widehat{Z}_s\big\|^2+p(p-2)\big|\widehat{Y}_s\big|^{p-4}\widehat{Y}^\ast_s  \Lambda_s \widehat{Y}_s.
		\end{split}
	\end{equation*}  
	Therefore,
	\begin{equation}\label{Brownian W geq}
		e^{\beta A_s} 	\mbox{Trace}(D^2\theta(\widehat{Y}_s)\widehat{Z}_s \widehat{Z}^\ast_s)\geq pe^{\beta A_s} \big|\widehat{Y}_s\big|^{p-2} \big\|\widehat{Z}_s\big\|^2.
	\end{equation}
	Let us deal with the jump part in the formula \eqref{Ito p > 2}. To this end, we use Taylor expansion with integral remainder to the function $r \mapsto \theta(x+ry)$ on $[0,1]$ for any $x,y \in \mathbb{R}^d$, which yields to the following
	\begin{equation*}
		\begin{split}
			\theta(x+y)-\theta(x)-\nabla \theta(x)y
			&=\int_{0}^{1} y^\ast D^2 \theta(x+ry)y (1-r) dr\\
			&=p\left| y \right|^2  \int_{0}^{1} \left|x+ry \right|^{p-2}(1-r)dr+p(p-2)\int_{0}^{1} \left|y(x+ry)\right|^2 \left|x+ry\right|^{p-4} (1-r) dr\\
			&\geq  p\left| y \right|^2  \int_{0}^{1} \left|x+ry \right|^{p-2}(1-r)dr
		\end{split}
	\end{equation*} 
	Now applying Lemma \ref{**Lemma} to the last integral term, we derive 
	$$
	p\left| y \right|^2  \int_{0}^{1} \left|x+ry \right|^{p-2}(1-r)dr \geq p 3^{1-p} \left|y\right|^2 \left|x\right|^{p-2}.
	$$
	Henceforth, for any $x,y \in \mathbb{R}^d$, we have
	\begin{equation*}
		\theta(x+y)-\theta(x)-\nabla \theta(x)y \geq p 3^{1-p} \left|y\right|^2 \left|x\right|^{p-2}.
	\end{equation*}
	Applying this inequality to the term $\big|\widehat{Y}_{s-}+\widehat{U}_s(e)\big|^p-\big|\widehat{Y}_{s-}\big|^p-p\widehat{Y}_{s-}\big|\widehat{Y}_{s-}\big|^{p-2}\widehat{U}_s(e)$ with $x=\widehat{Y}_{s-}(\omega)$ and $y=\widehat{U}_s(\omega,e)$, we derive 
	\begin{equation*}
		\big|\widehat{Y}_{s-}+\widehat{U}_s(e)\big|^p-\big|\widehat{Y}_{s-}\big|^p-p\widehat{Y}_{s-}\big|\widehat{Y}_{s-}\big|^{p-2}\widehat{U}_s(e) \geq  p 3^{1-p} \big|\widehat{U}_s(e)\big|^2 \big|\widehat{Y}_{s-}\big|^{p-2}~\mbox{a.s.}
	\end{equation*}
Therefore, using the decomposition of $N$ as the sum of its compensated random measure $\widetilde{N}$ and its dual predictable projection, we have
	\begin{equation}\label{ineqyality}
		\begin{split}
			&p\int_{t}^{T}e^{\beta A_s+\mu s} \int_{E}\widehat{Y}_{s-} \big|Y_{s-}\big|^{p-2}  \widehat{U}_s(e) \tilde{N}(ds,de)\\
			&+\int_{t}^{T}e^{\beta A_s+\mu s}\int_{E}\left(\big|\widehat{Y}_{s-}+\widehat{U}_s(e)\big|^p-\big|\widehat{Y}_{s-}\big|^p-p\widehat{Y}_{s-}\big|\widehat{Y}_{s-}\big|^{p-2}\widehat{U}_s(e)\right)N(ds,de)\\
			&=\int_{t}^{T}e^{\beta A_s+\mu s}\int_{E}\left(\big|\widehat{Y}_{s-}+\widehat{U}_s(e)\big|^p-\big|\widehat{Y}_{s-}\big|^p-p\widehat{Y}_{s-}\big|\widehat{Y}_{s-}\big|^{p-2}\widehat{U}_s(e)\right)Q(de)ds\\
			&+\int_{t}^{T}e^{\beta A_s+\mu s}\int_{E}\left(\big|\widehat{Y}_{s-}+\widehat{U}_s(e)\big|^p-\big|\widehat{Y}_{s-}\big|^p\right)\tilde{N}(ds,de)\\
			&\geq \int_{t}^{T}e^{\beta A_s+\mu s}\int_{E}\left(\big|\widehat{Y}_{s-}+\widehat{U}_s(e)\big|^p-\big|\widehat{Y}_{s-}\big|^p\right)\tilde{N}(ds,de)+p 3^{1-p} \int_{t}^{T}e^{\beta A_s+\mu s}\int_{E} \big|\widehat{Y}_{s-}\big|^{p-2} \big\|\widehat{U}_s\big\|^2_{\mathbb{L}^2_Q} ds.
		\end{split}
	\end{equation}
	Plugging \eqref{generator estimation}, \eqref{Brownian W geq} and \eqref{ineqyality} into \eqref{Ito p > 2}, we get
	\begin{equation}\label{Ito p > 2--1}
		\begin{split}
			&e^{\beta A_t+\mu T} \big|\widehat{Y}_t\big|^p+\beta\int_{t}^{T} e^{\beta A_s+\mu s} \big|\widehat{Y}_s\big|^p dA_s+( 3^{1-p}p+\frac{p \varepsilon}{2})\int_{t}^{T} e^{\beta A_s-\mu s}\big|\widehat{Y}_s\big|^{p-2}\left\{\big\|\widehat{Z}_s\big\|^2+\big\|\widehat{U}_s\big\|^2_{\mathbb{L}^2_Q}\right\}ds\\
			&\leq e^{\beta A_T+\mu T} \big|\widehat{\xi}\big|^p +\frac{p }{2}\int_{t}^{T}e^{\beta A_s+\mu s} \big|\widehat{Y}_s\big|^{p-2}\left|\widehat{f}(s,Y^2_s,Z^2_s,U^2_s) \right|^2ds+\left(\frac{p}{2}-\mu\right)\int_{t}^{T}e^{\beta A_s+\mu s} \big|\widehat{Y}_s\big|^{p}ds\\
			&-p\int_{t}^{T}e^{\beta A_s+\mu s} \widehat{Y}_{s-} \big|\widehat{Y}_{s-}\big|^{p-2}  \widehat{Z}_s dW_s-\int_{t}^{T}e^{\beta A_s+\mu s}\int_{E}\left(\big|\widehat{Y}_{s-}+\widehat{U}_s(e)\big|^p-\big|\widehat{Y}_{s-}\big|^p\right)\tilde{N}(ds,de).
		\end{split}
	\end{equation}
	As $p>2$, we can apply Young's inequality, which yields to the following
	\begin{equation}\label{Vendr Palourda}
		\begin{split}
			e^{\beta A_s+\mu s} \big|\widehat{Y}_s\big|^{p-2}\left| \widehat{f}(s,Y^2_s,Z^2_s,U^2_s)\right|^2ds
			&= e^{\frac{(p-2)}{p} (\beta A_s+\mu s)} \big|\widehat{Y}_s\big|^{p-2}\left(e^{ \frac{2}{p} (\beta A_s+\mu s)}\left| \widehat{f}(s,Y^2_s,Z^2_s,U^2_s)\right|^2\right)ds\\
			&\leq \frac{p-2}{p}e^{\beta A_s+\mu s} \big|\widehat{Y}_s\big|^{p} ds +\frac{2}{p}e^{\beta A_s+\mu s}\left| \widehat{f}(s,Y^2_s,Z^2_s,U^2_s)\right|^pds.
		\end{split}
	\end{equation}
	Putting \eqref{Vendr Palourda} into \eqref{Ito p > 2--1}, yields
	\begin{equation*}
		\begin{split}
			&e^{\beta A_t+\mu t} \big|\widehat{Y}_t\big|^p+\beta\int_{t}^{T} e^{\beta A_s+\mu s} \big|\widehat{Y}_s\big|^p dA_s
			+(3^{1-p}p-\frac{p \varepsilon}{2})\int_{t}^{T} e^{\beta A_s+\mu s}\big|\widehat{Y}_s\big|^{p-2}\left\{\big\|\widehat{Z}_s\big\|^2+\big\|\widehat{U}_s\big\|^2_{\mathbb{L}^2_Q}\right\}ds\\
			&\leq e^{\beta A_T+\mu T} \big|\widehat{\xi}\big|^p + \int_{t}^{T}e^{\beta A_s+\mu s} \left|\widehat{f}(s,Y^2_s,Z^2_s,U^2_s) \right|^pds+\left((p-1)-\mu\right)\int_{t}^{T} e^{\beta A_s+\mu s} \big|\widehat{Y}_s\big|^p ds\\
			&-p\int_{t}^{T}e^{\beta A_s+\mu s} \widehat{Y}_{s-} \big|\widehat{Y}_{s-}\big|^{p-2}  \widehat{Z}_s dW_s-\int_{t}^{T}e^{\beta A_s+\mu s}\int_{E}\left(\big|\widehat{Y}_{s-}+\widehat{U}_s(e)\big|^p-\big|\widehat{Y}_{s-}\big|^p\right)\tilde{N}(ds,de).
		\end{split}
	\end{equation*}
	Choosing $\varepsilon=3^{1-p}$, $\mu=p-1$, we obtain for any $\beta>0$
	\begin{equation}\label{Ito p > 2--2}
		\begin{split}
			&e^{\beta A_t+(p-1) t} \big|\widehat{Y}_t\big|^p+\int_{t}^{T} e^{\beta A_s+(p-1) s} \big|\widehat{Y}_s\big|^p dA_s
			+\frac{3^{1-p}p}{2}\int_{t}^{T} e^{\beta A_s+(p-1) s}\big|\widehat{Y}_s\big|^{p-2}\left\{\big\|\widehat{Z}_s\big\|^2+\big\|\widehat{U}_s\big\|^2_{\mathbb{L}^2_Q}\right\}ds\\
			&\leq e^{\beta A_T+(p-1) T} \big|\widehat{\xi}\big|^p + \int_{t}^{T}e^{\beta A_s+(p-1) s} \left| \widehat{f}(s,Y^2_s,Z^2_s,U^2_s) \right|^pds-p\int_{t}^{T}e^{\beta A_s+(p-1) s} Y_{s-} \big|\widehat{Y}_{s-}\big|^{p-2}  \widehat{Z}_s dW_s\\
			&-\int_{t}^{T}e^{\beta A_s+(p-1) s}\int_{E}\left(\big|\widehat{Y}_{s-}+\widehat{U}_s(e)\big|^p-\big|\widehat{Y}_{s-}\big|^p\right)\tilde{N}(ds,de).
		\end{split}
	\end{equation}
	Note that the two local martingales arising in the last line of the  estimation \eqref{Ito p > 2--2} are true martingales with zero expectation. To see this, we apply the BDG and Young's inequalities to obtain
	\begin{equation*}
		\begin{split}
			&\mathbb{E}\left[ \sup_{t \in [0,T]}\left|\int_{0}^{t}e^{\beta A_s+(p-1) s} \widehat{Y}_{s-} \big|\widehat{Y}_{s-}\big|^{p-2}  \widehat{Z}_s dW_s\right|\right] \\
			&\leq \mathfrak{c}\mathbb{E}\left[\left(\int_{0}^{T}e^{2\beta A_s+2(p-1) s}\big|\widehat{Y}_{s}\big|^{2p-2}  \big\|\widehat{Z}_s\big\|^2ds\right)^{\frac{1}{2}}\right]\\
			&\leq e^{(p-1) T} \mathfrak{c}\mathbb{E}\left[\left( \sup_{t \in [0,T]}e^{\frac{\beta(p-1)}{p} A_t}\big|\widehat{Y}_{t}\big|^{p-1}\right) \left(\int_{0}^{T}e^{\frac{2 \beta}{p} A_s}  \big\|\widehat{Z}_s\big\|^2ds\right)^{\frac{1}{2}}\right]\\
			&\leq e^{(p-1) T}\mathfrak{c}\left(\frac{p-1}{p}\mathbb{E}\left[\sup_{t \in [0,T]}e^{\beta A_s}\big|\widehat{Y}_{t}\big|^{p}\right]+\frac{1}{p}\mathbb{E}\left[\left(\int_{0}^{T}e^{\frac{2 \beta}{p} A_s}  \big\|\widehat{Z}_s\big\|^2ds\right)^{\frac{p}{2}}\right]\right),
		\end{split}
	\end{equation*}
	where $\mathfrak{c}>0$ is the universal BDG constant.\\
	As $\frac{2 \beta}{p} \leq \beta$ and from the integrability property satisfied by the processes $Y^1$, $Y^2$, $Z^1$ and $Z^2$, we have 
	\begin{equation*}
		\begin{split}
			&\mathbb{E}\left[ \sup_{t \in [0,T]}\left|\int_{0}^{t}e^{\beta A_s+(p-1) s} \widehat{Y}_{s-} \big|\widehat{Y}_{s-}\big|^{p-2}  \widehat{Z}_s dW_s\right|\right] \\
			&\leq \mathfrak{C}_{p,T}\left(\mathbb{E}\left[\sup_{t \in [0,T]}e^{\beta A_s}\big|\widehat{Y}_{t}\big|^{p}\right]+\mathbb{E}\left[\left(\int_{0}^{T}e^{ \beta A_s}  \big\|\widehat{Z}_s\big\|^2ds\right)^{\frac{p}{2}}\right]\right)<+\infty.
		\end{split}
	\end{equation*}
	By the same argument, as  $Y^1, Y^2 \in \mathcal{S}^p$ and $U^1, U^2 \in \mathfrak{L}^p_\beta$, we have
	\begin{equation*}
		\begin{split}
			&\mathbb{E}\left[ \sup_{t \in [0,T]}\left|\int_{0}^{t}e^{\beta A_s+(p-1) s}\int_{E} Y_{s-} \big|\widehat{Y}_{s-}\big|^{p-2}  \widehat{U}_s(e) \tilde{N}(ds,de)\right|\right] \\
			&\leq \mathfrak{c}\mathbb{E}\left[\left(\int_{0}^{T}e^{2\beta A_s+2(p-1) s}\big|\widehat{Y}_{s}\big|^{2p-2} \int_{E}  \big|\widehat{U}_s(e)\big|^2 N(ds,de)\right)^{\frac{1}{2}}\right]\\
			&\leq e^{\mu_p T} \mathfrak{c}\mathbb{E}\left[\left( \sup_{t \in [0,T]}e^{\frac{\beta(p-1)}{p} A_t}\big|\widehat{Y}_{t}\big|^{p-1}\right) \left(\int_{0}^{T}e^{\frac{2 \beta}{p} A_s} \int_{E} \big|\widehat{U}_s(e)\big|^2 N(ds,de)\right)^{\frac{1}{2}}\right]\\
			&\leq  \mathfrak{C}_{p,T}\left(\mathbb{E}\left[ \sup_{t \in [0,T]}e^{\beta A_s}\big|\widehat{Y}_{t}\big|^{p}\right] +\mathbb{E}\left[\left(\int_{0}^{T}e^{\beta A_s} \int_{{E}} \big|\widehat{U}_s(e)\big|^2 N(ds,de)\right)^{\frac{p}{2}}\right]\right)<+\infty.
		\end{split}
	\end{equation*}
	Coming back to \eqref{Ito p > 2--2}, taking the expectation on both sides, we get for any $\beta >0$
	\begin{equation}\label{Ito p > 2--3}
		\begin{split}
			&\mathbb{E}\int_{0}^{T} e^{\beta A_s} \big|\widehat{Y}_s\big|^p dA_s
			+\mathbb{E}\int_{0}^{T} e^{\beta A_s}\big|\widehat{Y}_s\big|^{p-2}\left\{\big\|\widehat{Z}_s\big\|^2+\big\|\widehat{U}_s\big\|^2_{\mathbb{L}^2_Q}\right\}ds\\
			&\leq \mathfrak{C}_{p,T} \left(\mathbb{E}\left[ e^{\beta A_T} \big|\widehat{\xi}\big|^p\right]  + \mathbb{E}\int_{0}^{T}e^{\beta A_s} \left| \widehat{f}(s,Y^2_s,Z^2_s,U^2_s) \right|^p ds\right).
		\end{split}
	\end{equation}
	Now writing the It\^o's formula \eqref{Ito p > 2} in a different way and using the convexity of the function $\theta$, we get 
	\begin{equation*}
		\begin{split}
			e^{\beta A_t} \big|\widehat{Y}_t\big|^p \leq &e^{\beta A_T} \big|\widehat{\xi}\big|^p+\frac{p}{2 \epsilon}\int_{t}^{T}e^{\beta A_s} \big|\widehat{Y}_s\big|^p dA_s+\int_{t}^{T}e^{\beta A_s} \left| \widehat{f}(s,Y^2_s,Z^2_s,U^2_s)\right|^pds\\
			&+\frac{3^{1-p}p}{2}\int_{t}^{T} e^{\beta A_s}\big|\widehat{Y}_s\big|^{p-2}\left\{\big\|\widehat{Z}_s\big\|^2+\big\|\widehat{U}_s\big\|^2_{\mathbb{L}^2_Q}\right\}ds-p\int_{t}^{T}e^{\beta A_s} \widehat{Y}_{s-} \big|\widehat{Y}_{s-}\big|^{p-2}  \widehat{Z}_s dW_s\\
			& -p\int_{t}^{T}e^{\beta A_s} \int_{E}\widehat{Y}_{s-} \big|\widehat{Y}_{s-}\big|^{p-2} \widehat{U}_s(e) \tilde{N}(ds,de) .
		\end{split}
	\end{equation*}
	Therefore, by passing to the supremum and then the expectation on both sides of the estimation above, we have
	\begin{equation}\label{Ito p > 2--V1}
		\begin{split}
			\mathbb{E}\left[ \sup_{t \in [0,T]} e^{\beta A_t} \big|\widehat{Y}_t\big|^p\right]  \leq &\mathbb{E}\left[ e^{\beta A_T} \big|\widehat{\xi}\big|^p\right] +\frac{p}{2\epsilon}\mathbb{E}\int_{0}^{T}e^{\beta A_s} \big|\widehat{Y}_s\big|^p dA_s+\mathbb{E}\int_{0}^{T}e^{\beta A_s} \left| \widehat{f}(s,Y^2_s,Z^2_s,U^2_s) \right|^pds\\
			&+\frac{3^{1-p}p}{2}\mathbb{E}\int_{0}^{T} e^{\beta A_s}\big|\widehat{Y}_s\big|^{p-2}\left\{\big\|\widehat{Z}_s\big\|^2+\big\|\widehat{U}_s\big\|^2_{\mathbb{L}^2_Q}\right\}ds\\
			&+p\mathbb{E}\left[ \sup_{t \in [0,T]}\left|\int_{t}^{T}e^{\beta A_s} \widehat{Y}_{s-} \big|\widehat{Y}_{s-}\big|^{p-2}  \widehat{Z}_s dW_s\right|\right] \\
			&
			+p\mathbb{E}\left[ \sup_{t \in [0,T]}\left|\int_{t}^{T}e^{\beta A_s} \int_{E}\widehat{Y}_{s-} \big|\widehat{Y}_{s-}\big|^{p-2}  \widehat{U}_s(e) \tilde{N}(ds,de)\right|\right] .
		\end{split}
	\end{equation}
	Using again the BDG inequality for the last two martingales, we derive 
	\begin{equation}\label{BDG W-p2v1}
		\begin{split}
			&p\mathbb{E}\left[ \sup_{t \in [0,T]}\left|\int_{t}^{T}e^{\beta A_s} \widehat{Y}_{s-} \big|\widehat{Y}_{s-}\big|^{p-2}  \widehat{Z}_s dW_s\right|\right] \\
			&\leq p\mathfrak{c}\mathbb{E}\left[\left(\int_{0}^{T}e^{2\beta A_s}\big|\widehat{Y}_{s}\big|^{2p-2}  \big\|\widehat{Z}_s\big\|^2ds\right)^{\frac{1}{2}}\right]\\
			&\leq p\mathfrak{c}\mathbb{E}\left[\left( \sup_{t \in [0,T]}e^{\frac{\beta}{2} A_t}\big|\widehat{Y}_{t}\big|^{\frac{p}{2}}\right) \left(\int_{0}^{T}e^{ \beta A_s} \big|\widehat{Y}_{s}\big|^{p-2} \big\|\widehat{Z}_s\big\|^2ds\right)^{\frac{1}{2}}\right]\\
			&\leq \frac{1}{4}\mathbb{E}\left[\sup_{t \in [0,T]}e^{{\beta} A_t}\big|\widehat{Y}_{t}\big|^{p}\right]+p^2 \mathfrak{c}^2\mathbb{E}\int_{0}^{T}e^{ \beta A_s} \big|\widehat{Y}_{s}\big|^{p-2} \big\|\widehat{Z}_s\big\|^2ds.
		\end{split}
	\end{equation}
	Similarly, we get
	\begin{equation}\label{BDG N-p2v1}
		\begin{split}
			&p\mathbb{E}\left[ \sup_{t \in [0,T]}\left|\int_{t}^{T}e^{\beta A_s} \int_{E}\widehat{Y}_{s-} \big|\widehat{Y}_{s-}\big|^{p-2}  \widehat{U}_s(e) \tilde{N}(ds,de)\right|\right] \\
			&\leq \frac{1}{4}\mathbb{E}\left[\sup_{t \in [0,T]}e^{{\beta} A_t}\big|\widehat{Y}_{t}\big|^{p}\right]+p^2 \mathfrak{c}^2\mathbb{E}\int_{0}^{T}e^{ \beta A_s} \big|\widehat{Y}_{s}\big|^{p-2} \big\|\widehat{U}_s\big\|^2_{\mathbb{L}^2_Q}ds .
		\end{split}
	\end{equation}
	Returning to \eqref{Ito p > 2--V1}, then by using \eqref{Ito p > 2--3}, \eqref{BDG W-p2v1} and \eqref{BDG N-p2v1}, we obtain
	\begin{equation}\label{Ito p > 2--supv1}
		\begin{split}
			&\mathbb{E}\left[\sup_{t \in [0,T]}e^{{\beta} A_t}\big|\widehat{Y}_{t}\big|^{p}\right]\leq \mathfrak{C}_{p,T,\epsilon} \left(\mathbb{E}\left[ e^{\beta A_T} \big|\widehat{\xi}\big|^p\right]  + \mathbb{E}\int_{0}^{T}e^{\beta A_s} \left| \widehat{f}(s,Y^2_s,Z^2_s,U^2_s) \right|^p ds\right).
		\end{split}
	\end{equation}

	Now, it remains to deal with the a priori estimates of the control process $(Z,U)$, especially the term
	$$
	\mathbb{E}\left[\left(\int_{0}^{T}e^{\beta A_s} \big\|\widehat{Z}_s\big\|^2ds\right)^{\frac{p}{2}}\right]+\mathbb{E}\left[\left(\int_{0}^{T}e^{\beta A_s} \big\|\widehat{U}_s\big\|^2_{\mathbb{L}^2_Q}ds\right)^{\frac{p}{2}}\right]+\mathbb{E}\left[\left(\int_{0}^{T}e^{\beta A_s} \int_{E}\big|\widehat{U}_s(e)\big|^2N(ds,de)\right)^{\frac{p}{2}}\right].
	$$
	To this end and to avoid the term $\big(\big|\widehat{Y}_t\big|^{p-2}\big)_{t \leq T}$ in the obtained estimations, we apply It\^o's formula to the semimartingale $\big(e^{\beta A_t} \left|Y_t\right|^2\big)_{t \leq T}$ with $\beta >0$. We then have
	\begin{equation}\label{Ito p=2}
		\begin{split}
			&e^{\beta A_t} \big|\widehat{Y}_t\big|^2+\beta \int_{t}^{T}e^{\beta A_s} \big|\widehat{Y}_s\big|^2 dA_s+\int_{t}^{T}e^{\beta A_s }\big\|\widehat{Z}_s\big\|^2ds+\int_{t}^{T}e^{\beta A_s} \int_{E}\big|\widehat{U}_s(e)\big|^2N(ds,de)\\
			&=e^{\beta A_T} \big|\widehat{\xi}\big|^2 +2\int_{t}^{T}e^{\beta A_s}   \widehat{Y}_s \left(f_1(s,Y^1_s,Z^1_s,U^1_s)-f_2(s,Y^2_s,Z^2_s,U^2_s)\right) ds\\
			&-2\int_{t}^{T}e^{\beta A_s} \widehat{Y}_{s-}  \widehat{Z}_s dW_s -2\int_{t}^{T}e^{\beta A_s} \int_{E}\widehat{Y}_{s-}   \widehat{U}_s(e) \tilde{N}(ds,de) 
		\end{split}
	\end{equation} 
	Next, we use a suitable modification of the estimation \eqref{generator estimation} on the generator $f$. By the same calculation performed in \eqref{generator estimation}, we have
	\begin{equation}\label{generator estimation-v1}
		\begin{split}
			&2\widehat{Y}_s \left(f_1(s,Y^1_s,Z^1_s,U^1_s)-f_2(s,Y^2_s,Z^2_s,U^2_s)\right)  ds\\ 
			&\leq 2\big|\widehat{Y}_s \big|\left(\alpha_s \big|\widehat{Y}_s \big|+ \eta_s \big\|\widehat{Z}_s \big\|+\delta_s \big\|\widehat{U}_s \big\|_{\mathbb{L}^2_Q}+\big| \widehat{f}(s,Y^2_s,Z^2_s,U^2_s) \big| \right)ds\\
			& \leq \left(\alpha_s+\dfrac{a^2_s}{\frac{1}{2} \wedge \varepsilon}\right)\big| \widehat{Y}_s\big|^2 dA_s+\left(\frac{1}{2}\big\|\widehat{Z}_s \big\|^2+\varepsilon\big\|\widehat{U}_s \big\|^2_{\mathbb{L}^2_Q}\right)ds+2\big| \widehat{Y}_s\big| \big| \widehat{f}(s,Y^2_s,Z^2_s,U^2_s) \big|ds\\
			&\leq \left(\frac{1}{2}\big\|\widehat{Z}_s \big\|^2+\varepsilon\big\|\widehat{U}_s \big\|^2_{\mathbb{L}^2_Q}\right)ds+2\big| \widehat{Y}_s\big| \big| \widehat{f}(s,Y^2_s,Z^2_s,U^2_s) \big|ds
		\end{split}
	\end{equation}
	 For the last term on the last line, from assumption (H5), we have
	\begin{equation}\label{generator estimation-v2}
		\begin{split}
			2e^{\beta A_s+\mu s}\big| \widehat{Y}_s\big| \big| \widehat{f}(s,Y^2_s,Z^2_s,U^2_s) \big|
			&\leq   e^{\frac{2}{p}(\beta A_s+\mu s)}\big| \widehat{Y}_s\big|^2  +e^{2\frac{p-1}{p}(\beta A_s+\mu s)} \big|\widehat{f}(s,Y^2_s,Z^2_s,U^2_s) \big|^2  \\
			&\leq \frac{1}{\epsilon^{\frac{4}{p}}} e^{\frac{2}{p}(\beta A_s+\mu s)}\big| \widehat{Y}_s\big|^2 a^{\frac{4}{p}}_s  +e^{2\frac{p-1}{p}(\beta A_s+\mu s)} \big|\widehat{f}(s,Y^2_s,Z^2_s,U^2_s) \big|^2
		\end{split}
	\end{equation}
      Plugging \eqref{generator estimation-v1} and \eqref{generator estimation-v2} into \eqref{Ito p=2}, we get for any $\varepsilon>0$
	\begin{equation*}
		\begin{split}
			&\beta\int_{0}^{T}e^{\beta A_s}\big| \widehat{Y}_s\big|^2 dA_s+ \frac{1}{2}\int_{0}^{T}e^{\beta A_s }\big\|\widehat{Z}_s\big\|^2ds+\int_{0}^{T}e^{\beta A_s} \int_{E}\big|\widehat{U}_s(e)\big|^2N(ds,de)\\
			&\leq e^{\beta A_T} \big|\widehat{\xi}\big|^2 +\int_{0}^{T}e^{2\frac{p-1}{p}\beta A_s} \big| \widehat{f}(s,Y^2_s,Z^2_s,U^2_s) \big|^2ds+\frac{1}{\epsilon^{\frac{4}{p}}} \int_{0}^{T} e^{\frac{2}{p}\beta A_s}\big| \widehat{Y}_s\big|^2 a^{\frac{4}{p}} ds \\
			&+\varepsilon\mathbb{E}\int_{0}^{T} e^{\beta A_s }\big\|\widehat{U}_s\big\|^2_{\mathbb{L}^2_Q}ds-2\int_{t}^{T}e^{\beta A_s} \widehat{Y}_{s-}  \widehat{Z}_s dW_s -2\int_{t}^{T}e^{\beta A_s} \int_{E}\widehat{Y}_{s-}   \widehat{U}_s(e) \tilde{N}(ds,de).
		\end{split}
	\end{equation*} 
      Using the basic inequality 
      $$
      \left(	\sum_{i=1}^n a_i\right)^{\frac{p}{2}} \leq n^{\frac{p-2}{2}} \sum_{i=1}^n a^{\frac{p}{2}}_i,\quad \forall (n, \left\lbrace a_i \right\rbrace_{i=1,\cdots,n} ) \in \mathbb{N}^\ast \times \big(\mathbb{R}_+\big)^n, 
      $$
      we get after taking the expectation on both sides
	\begin{equation}\label{Ito p = 2}
		\begin{split}
			&\left(\dfrac{1}{2}\right)^{\frac{p}{2}}\mathbb{E}\left[\left(\int_{0}^{T}e^{\beta A_s}\big\|\widehat{Z}_s\big\|^2ds\right)^{\frac{p}{2}}\right]+\mathbb{E}\left[\left(\int_{0}^{T}e^{\beta A_s} \int_{E}\big|\widehat{U}_s(e)\big|^2N(ds,de)\right)^{\frac{p}{2}}\right]\\
			&\leq 2 \cdot 6^{\frac{p-2}{2}} \left(\mathbb{E}\left[  e^{\frac{p}{2}\beta A_T} \big|\widehat{\xi}\big|^p\right] +\mathbb{E} \left[\left(\int_{0}^{T}e^{2\frac{p-1}{p}\beta A_s} \big| \widehat{f}(s,Y^2_s,Z^2_s,U^2_s) \big|^2ds\right)^{\frac{p}{2}}\right]\right.\\
			&\qquad\qquad\left.+\frac{1}{\epsilon^{2}}\mathbb{E}\left[\left(\int_{0}^{T} e^{\frac{2}{p}\beta A_s}\big| \widehat{Y}_s\big|^2 a^{\frac{4}{p}} ds\right)^{\frac{p}{2}}\right]+\varepsilon^{\frac{p}{2}}\mathbb{E}\left[\left(\int_{0}^{T} e^{\beta A_s}\big\|\widehat{U}_s\big\|^2_{\mathbb{L}^2_Q}ds\right)^{\frac{p}{2}}\right]\right.\\
			&\qquad\qquad\left.+2^{\frac{p}{2}}\mathbb{E}\left[ \left|\int_{0}^{T}e^{\beta A_s} \widehat{Y}_{s-}  \widehat{Z}_s dW_s\right|^{\frac{p}{2}}\right]  +2^{\frac{p}{2}}\mathbb{E}\left[ \left|\int_{0}^{T}e^{\beta A_s} \int_{E}\widehat{Y}_{s-}   \widehat{U}_s(e) \tilde{N}(ds,de) \right|^{\frac{p}{2}}\right] \right) 
		\end{split}
	\end{equation}
	By applying the BDG inequality for the two stochastic integrals on the last line (see Remark \ref{rmq p}) to have
	\begin{equation*}
		\begin{split}
			2^{\frac{p+2}{2}}\cdot6^{\frac{p-2}{2}}\mathbb{E}\left[ \left|\int_{0}^{T}e^{\beta A_s} \widehat{Y}_{s-}  \widehat{Z}_s dW_s\right|^{\frac{p}{2}}\right]  &\leq 	\mathfrak{C}_p \mathbb{E}\left[ \left(\int_{0}^{T}e^{2\beta A_s} \big| \widehat{Y}_{s}\big|^2   \big\| \widehat{Z}_s\big\|^2  ds\right)^{\frac{p}{4}}\right] \\
			&= 	\mathfrak{C}_p \mathbb{E}\left[ \left(\int_{0}^{T}e^{2(1-\frac{1}{p})\beta A_s} \big| \widehat{Y}_{s}\big|^2  e^{\frac{2}{p}\beta A_s} \big\| \widehat{Z}_s\big\|^2  ds\right)^{\frac{p}{4}}\right] \\
			&  \leq \mathfrak{C}_p\mathbb{E}\left[\sup_{t \in [0,T]}e^{(p-1)\beta A_s}\big| \widehat{Y}_{s}\big|^p\right]+\frac{1}{2}\mathbb{E}\left[ \left(\int_{0}^{T}e^{\beta A_s}   \big\| \widehat{Z}_s\big\|^2  ds\right)^{\frac{p}{2}}\right].
		\end{split}
	\end{equation*}
	Similarly, we get
	\begin{equation*}
		\begin{split}
			&2^{\frac{p+2}{2}}\cdot6^{\frac{p-2}{2}} \mathbb{E}\left[ \left|\int_{0}^{T}e^{\beta A_s} \int_{E}\widehat{Y}_{s-}  \widehat{U}_s(e) \tilde{N}(ds,de)  \right|^{\frac{p}{2}}\right] \\ &\leq 	\mathfrak{C}_p \mathbb{E}\left[ \left(\int_{0}^{T}e^{2\beta A_s} \int_{E} \big| \widehat{Y}_{s}\big|^2   \big| \widehat{U}_s(e)\big|^2 N(ds,de)\right)^{\frac{p}{4}}\right] \\
			&  \leq \mathfrak{C}_p\mathbb{E}\left[\sup_{t \in [0,T]}e^{(p-1)\beta A_s}\big| \widehat{Y}_{s}\big|^p\right]+\frac{1}{2}\mathbb{E}\left[ \left(\int_{0}^{T}e^{\beta A_s}  \int_{E} \big| \widehat{U}_s(e)\big|^2  N(ds,de)\right)^{\frac{p}{2}}\right].
		\end{split}
	\end{equation*}
	Additionally, using Remark \ref{rmq p} (see \eqref{CE}), we derive 
	\begin{equation}\label{compensator}
		\mathbb{E}\left[\left(\int_{0}^{T} e^{\beta A_s}\big\|\widehat{U}_s\big\|^2_{\mathbb{L}^2_Q}ds\right)^{\frac{p}{2}}\right] \leq \left(\frac{p}{2}\right)^{\frac{p}{2}} \mathbb{E}\left[\left(\int_{0}^{T} e^{\beta A_s}  \int_{E} \big| \widehat{U}_s(e)\big|^2  N(ds,de)\right)^{\frac{p}{2}}\right].
	\end{equation}
	On the other hand, using Jensen's inequality, we have
	$$
	\left(\int_{0}^{T}e^{2\frac{p-1}{p}\beta A_s} \big| \widehat{f}(s,Y^2_s,Z^2_s,U^2_s) \big|^2ds\right)^{\frac{p}{2}} \leq T^{\frac{p}{2}-1}\int_{0}^{T}e^{(p-1) \beta A_s} \big| \widehat{f}(s,Y^2_s,Z^2_s,U^2_s) \big|^pds
	$$
	and
	$$
	\left(\int_{0}^{T} e^{\frac{2}{p}\beta A_s}\big| \widehat{Y}_s\big|^2 a^{\frac{4}{p}} ds\right)^{\frac{p}{2}} \leq T^{\frac{p-2}{2}}\int_{0}^{T}e^{\beta A_s}\big| \widehat{Y}_s\big|^p dA_s.
	$$
Plugging the four estimations above into \eqref{Ito p = 2}, and after taking the expectation and applying \eqref{Ito p > 2--3} and \eqref{Ito p > 2--supv1} with $\beta$ replaced by $(p-1)\beta$, we derive the following:
	\begin{equation*}
		\begin{split}
			&\frac{1}{2}\mathbb{E}\left[\left(\int_{0}^{T}e^{\beta A_s}\big\|\widehat{Z}_s\big\|^2ds\right)^{\frac{p}{2}}\right]+\left(\frac{1}{2}-\frac{1}{3}\left(3 p \varepsilon\right)^{\frac{p}{2}}  \right)\mathbb{E}\left[\left(\int_{0}^{T}e^{\beta A_s} \int_{E}\big|\widehat{U}_s(e)\big|^2N(ds,de)\right)^{\frac{p}{2}}\right]\\
			&\leq \mathfrak{C}_{p,\epsilon,T}  \left(\mathbb{E}\left[ e^{(p-1){\beta} A_T} \big|\widehat{\xi}\big|^p\right]  + \mathbb{E}\int_{0}^{T}e^{(p-1){\beta} A_s} \left| \widehat{f}(s,Y^2_s,Z^2_s,U^2_s)\right|^p ds\right)
		\end{split}
	\end{equation*}
	Then, we choose $\varepsilon < \frac{1}{3p} \cdot \left(\frac{3}{2}\right)^{\frac{2}{p}}$, using \eqref{compensator} and the fact that $p-1 \geq \frac{p}{2}$ for any $p > 2$, so that for any $\beta > 0$, we have
	\begin{equation*}
		\begin{split}
			&\mathbb{E}\left[\left(\int_{0}^{T}e^{\beta A_s}\big\|\widehat{Z}_s\big\|^2ds\right)^{\frac{p}{2}}\right]+\mathbb{E}\left[\left(\int_{0}^{T}e^{\beta A_s} \big\|\widehat{U}_s\big\|^2_{\mathbb{L}^2_Q} ds\right)^{\frac{p}{2}}\right]+\mathbb{E}\left[\left(\int_{0}^{T}e^{\beta A_s} \int_{E}\big|\widehat{U}_s(e)\big|^2N(ds,de)\right)^{\frac{p}{2}}\right]\\
			&\leq \mathfrak{C}_{p,\epsilon,T}  \left(\mathbb{E}\left[ e^{(p-1){\beta} A_T} \big|\widehat{\xi}\big|^p\right]  + \mathbb{E}\int_{0}^{T}e^{(p-1){\beta} A_s} \left| \widehat{f}(s,Y^2_s,Z^2_s,U^2_s) \right|^p ds\right)
		\end{split}
	\end{equation*}

Completing the proof.
\end{proof}

\subsection{$\mathbb{L}^p$-estimates for $p \in  (1,2)$}
\label{hrss}
Before presenting the a priori estimates for the solution of the BSDE \eqref{basic BSDE} for $p \in (1,2)$, we first recall a necessary result. Unlike the case where $p \in [2, +\infty)$, we cannot apply It\^o's formula directly to the function $y \mapsto \big|y\big|^{p}$, as this function lacks sufficient smoothness. Therefore, an alternative extension of It\^o's formula, adapted to this context, is required. This formula initially appeared in the Brownian setting by Briand et al.\ \cite{Briand2003} and was later extended to general filtrations supporting both a Brownian motion and an independent Poisson random measure by Kruse and Popier \cite{Kruse2015}. Consequently, we omit its proof here as it follows similarly, and we refer interested readers to \cite[Lemma 7]{Kruse2015}.\\
Throughout the remainder of the paper, for any $x \in \mathbb{R}^d$, we denote by $\check{x}$ the vector $\check{x} = \frac{x}{\left|x\right|} \mathds{1}_{x \neq 0}$.
\begin{lemma}\label{Lemma Ito for p less than 2}
	We consider the $\mathbb{R}^d$-valued semimartingale $(X_t)_{t\leq T}$ defined by
	$$ X_{t}=X_0+\int_{0}^{t}F_sds+\int_{0}^{t}Z_{s}dW_{s}+\int_0^t\int_{E}U_s(e)\tilde{N}(ds,de)$$
	such that 
	$$\int_0^T\left\{F_t+|Z_t|^2+\|U_t\|_{\mathbb{L}^2_Q}^2\right\}dt<+\infty\quad \mathbb{P}-a.s.$$
	Then, for any $p\geq 1$ there exists a continuous and non-decreasing process $(\ell_t)_{t\leq T}$ such that for all $\beta>0$ and $\mu \in \mathbb{R}$,
	\begin{eqnarray*}\label{e0}
		e^{\frac{p}{2}\beta A_t+\mu t}|X_t|^p&=&|X_0|^p+\frac{1}{2}\int_0^te^{\frac{p}{2}\beta A_s+\mu s}\mathds{1}_{\{p=1\}}d\ell_s+\frac{p}{2}\beta\int_0^te^{\frac{p}{2}\beta A_s+\mu s}|X_s|^pdA_s\nonumber\\
		&&+\mu \int_0^te^{\frac{p}{2}\beta A_s+\mu s}|X_s|^pds+p\int_0^te^{\frac{p}{2}\beta A_s+\mu s}|X_s|^{p-1}\check{X}_sF_s ds\nonumber\\
		&&+p\int_0^te^{\frac{p}{2}\beta A_s+\mu s}|X_s|^{p-1}\check{X}_sZ_sdW_s+p\int_0^t\int_{E}e^{\frac{p}{2}\beta A_s-\mu s}|X_{s-}|^{p-1}\check{X}_{s-}U_s(e)\tilde{N}(ds,de)\nonumber\\
		&&+\int_0^t\int_{E}e^{\frac{p}{2}\beta A_s+\mu s}\left[|X_{s-}+U_s(e)|^{p}-|X_{s-}|^p-p|X_{s-}|^{p-1}\check{X}_{s-}U_s(e)\right]N(ds,de)\nonumber\\
		&&+\frac{p}{2}\int_0^te^{\frac{p}{2}\beta A_s+\mu s}|X_s|^{p-2}\left[(2-p)\left(|Z_s|^2-\check{X}_s^*,Z_sZ_s^*\check{X}_s\right)+(p-1)|Z_s|^2\right]
		\mathds{1}_{\{X_s\neq0\}}ds.
	\end{eqnarray*}
	The process $(\ell_t)_{t\leq T}$ increases only on the boundary of the random set $\{t\leq T,\;\; X_{t-}=X_t=0\}$.
\end{lemma}

After applying Lemma \ref{Lemma Ito for p less than 2} to the state process $Y$ of the BSDE \eqref{basic BSDE}, we can derive a result that controls the jump component $U$ associated with the random measure $N$. This result, which was presented in \cite[Lemma 9]{Kruse2015} for a local martingale orthogonal to $W$ and $\tilde{N}$, is now provided here with a concise and explicit proof specifically for the jumps of $N$.
\begin{lemma}\label{Lemma betap for inequality}
	Let $p\in(1,2)$ and $\mathfrak{b}_p:=\frac{p(p-1)}{2}$. Then we have
	\begin{equation*}
		\begin{split}
			&\int_{0}^{t} \int_{E}\left\{\big|{Y}_{s-}-{U}_s(e)\big|^p+\big|{Y}_{s-}\big|^p-p\big|{Y}_{s-}\check{{Y}}_{s-}{U}_s(e)\big|\right\}N(ds,de)\\
			& \geq\mathfrak{b}_p\int_0^t\int_{E}e^{\frac{p}{2}\beta A_s}|{U}_s(e)|^2\left(|{Y}_{s-}|^{2}\vee|{Y}_{s-}+U_s(e)|^2\right)^{\frac{p-2}{2}}\mathds{1}_{\{|{Y}_{s-}|\vee|{Y}_{s-}+U_s(e)|\neq 0\}}N(ds,de).
		\end{split}
	\end{equation*}
\end{lemma}
\begin{proof}
	Let $\varepsilon >0$ and define $\nu_\varepsilon(y)=\left(\big|y\big|^2+\varepsilon^{2}\right)^{\frac{1}{2}}$ for any $y \in \mathbb{R}^d$. Note that this function is the one used to derive the result of Lemmas \ref{Lemma Ito for p less than 2} in \cite{Kruse2015} (or in \cite[Lemma 2.2]{Briand2003}).\\
	Using Taylor expansion with integral rest, we have
	\begin{equation*}
		\begin{split}
		&\int_0^t e^{\frac{p}{2}\beta A_s}\int_{E}\left[\nu_\varepsilon(Y_{s-}+U_s(e))^{p}-\nu_\varepsilon(Y_{s-})^{p}-pY_{s-}\nu_\varepsilon(Y_{s-})^{p-2}U_s(e)\right] N(ds, de)\\
		&=p\int_0^1 e^{\frac{p}{2}\beta A_s}\int_{E}\int_0^t(1-r)|U_s(e)|^2\nu_\varepsilon(Y_{s-}+r U_s(e))^{p-2} N(ds, de) dr\\
		&\qquad+p(p-2)\int_0^1 e^{\frac{p}{2}\beta A_s}\int_{E}\int_0^t (1-r)|U_s(e)|^2|Y_{s-}+r U_s(e)|^2 \nu_\varepsilon(Y_{s-}+r U_s(e))^{p-4} N(ds, de) dr\\
		&\geq p(p-1)\int_0^1 e^{\frac{p}{2}\beta A_s}\int_{E}\int_0^t (1-r)|U_s(e)|^2 \nu_\varepsilon(Y_{s-}+r U_s(e))^{p-2} N(ds, de) dr
	\end{split}
	\end{equation*}
	since \( p<2 \) and \( |y|^2 \nu_\varepsilon(y)^{p-4} \leq \nu_\varepsilon(y)^{p-2} \). On the other hand, observe that
	$$|Y_{s-}+r U_s(e)|=|(1-r)Y_{s-}+r(Y_{s-}+U_s(e))|\geq|Y_{s-}|\vee|Y_{s-}+U_s(e)|.$$
	Hence
	\begin{equation}\label{e15}
		\begin{split}
		&\int_0^t e^{\frac{p}{2}\beta A_s}\int_{E}\left[\nu_\varepsilon(Y_{s-}+U_s(e))^{p}-\nu_\varepsilon(Y_{s-})^{p}-pY_{s-}\nu_\varepsilon(Y_{s-})^{p-2}U_s(e)\right] N(ds, de)\nonumber\\
		&\geq p(p-1)\int_0^1 e^{\frac{p}{2}\beta A_s}\int_{E}\int_0^t(1-r)|U_s(e)|^2\left(|Y_{s-}+r U_s(e)|^2+\varepsilon^2\right)^{\frac{p-2}{2}} N(ds, de) dr\nonumber\\
		&\geq p(p-1)\int_0^t e^{\frac{p}{2}\beta A_s}\int_{E}|U_s(e)|^2\left(|Y_{s-}|^2\vee|Y_{s-}+U_s(e)|^2+\varepsilon^2\right)^{\frac{p-2}{2}} N(ds, de)\int_0^1(1-r) dr\nonumber\\
		&\geq \frac{p(p-1)}{2}\int_0^t e^{\frac{p}{2}\beta A_s}\int_{E}|U_s(e)|^2\left(|Y_{s-}|^2\vee|Y_{s-}+U_s(e)|^2+\varepsilon^2\right)^{\frac{p-2}{2}} N(ds, de).
	\end{split}
	\end{equation}
As we pass to the limit as $\varepsilon$ goes to zero, we obtain:
	\begin{equation*}
		\begin{split}
		&\int_0^t\int_{E}e^{\frac{p}{2}\beta A_s}\left[|Y_{s-}+U_s(e)|^{p}-|Y_{s-}|^p-p|Y_{s-}|^{p-1}\hat{Y}_{s-}U_s(e)\right] N(ds, de)\\
		&\geq\mathfrak{b}_p\int_0^t\int_{E}e^{\frac{p}{2}\beta A_s}|U_s(e)|^2\left(|Y_{s-}|^{2}\vee|Y_{s-}+U_s(e)|^2\right)^{\frac{p-2}{2}}\mathds{1}_{\{|Y_{s-}|\vee|Y_{s-}+U_s(e)|\neq 0\}} N(ds, de).
		\end{split}
	\end{equation*}
Completing the proof.
\end{proof}

Let $H : \Omega \times [0,T] \rightarrow \mathbb{R}$ be a $\mathcal{P}$-measurable process such that
$$
\int_0^T e^{\beta A_s}|H_s|^2ds<+\infty,
\qquad \mathbb{P}\text{-a.s.}
$$
Then, since $x\mapsto x^{\frac{p}{2}}$ is concave on $\mathbb{R}_+$ for $p \in (1,2)$, Jensen's inequality yields, for every $t\in(0,T]$,
\begin{equation}\label{Impo_Ine}
\left(\int_{0}^{t} e^{\beta A_s} |H_s|^2 ds\right)^{\frac{p}{2}}
\geq
t^{\frac{p}{2}-1}
\int_{0}^{t} e^{\frac{p}{2}\beta A_s} |H_s|^{p} ds.
\end{equation}
In particular, the right-hand side is well defined. At $t=0$, both integrals vanish, and the inequality may be understood by convention. In addition, by Hölder's inequality, we shall use the following elementary estimate: for any non-negative numbers $a_1,\ldots,a_m$,
\begin{equation}\label{Dizzy}
\left(\sum_{i=1}^m a_i\right)^p
\leq
m^{p-1}\sum_{i=1}^m a_i^p,
\qquad a_i\geq 0.
\end{equation}

From now until the end of this section, we assume that $p\in(1,2)$; in this case, we set $A_t=\int_0^t a_s^q ds$, where $q=\frac{p}{p-1}$, $a_s>0$ $ds\otimes d\mathbb{P}$-a.e. on $[0,T]$, and $a_s^2=\phi_s+\eta_s^2+\delta_s^2$, and we keep the notation used in Lemma \ref{Lemma betap for inequality}, namely $\mathfrak{b}_p=\frac{p(p-1)}{2}$. Recall that, in this case, $\mathscr{L}^p_Q=\mathbb{L}^1_Q+\mathbb{L}^2_Q$.

\begin{remark}\label{kech}
	Assumptions (H2)--(H5) imply that the process $(a_t)_{t\le T}$ is progressively measurable and satisfies $a_t<+\infty$ for every $t<T$, $\mathbb{P}$-a.s.	For $n\ge 1$, define  
    \begin{equation}\label{SPR}
	\tau_n:=\inf\{t\ge 0: a_t\ge n\}\wedge T.
	    \end{equation}  
	Then $\tau_n$ is a stopping time because the set $\{(t,\omega):a_t(\omega)\ge n\}$ is progressively measurable (since $a$ is) and the usual conditions on the filtration are assumed.  
	
	The sequence $(\tau_n)_{n\ge1}$ is nondecreasing and $\tau_n\uparrow T$ $\mathbb{P}$-a.s. Indeed, fix $\omega$. For any $t<T$, $a_t(\omega)<\infty$; hence there exists $N(\omega,t)$ such that $a_t(\omega)<n$ for all $n\ge N$. Consequently $\tau_n(\omega)>t$ for all sufficiently large $n$. Since $t<T$ is arbitrary, $\tau_n(\omega)\to T$.  
	
	By construction, for every $s<\tau_n(\omega)$ we have $a_s(\omega)<n$; thus  
	\[
	a_s\le n\qquad\text{for all }s\in[0,\tau_n)\text{ and }\mathbb{P}\text{-a.s.}
	\]  
	The endpoint $\tau_n$ has Lebesgue measure zero, so the inequality holds $ds\otimes d\mathbb{P}$-a.e. on $[0,\tau_n]$. Moreover, from the definition $a_s^2=\phi_s+\eta_s^2+\delta_s^2$ we obtain $|\eta_s|,|\delta_s|\le a_s$. Hence  
     $$
	\eta_s\vee\delta_s\le a_s\le n\qquad ds\otimes d\mathbb{P}\text{-a.e. on }[0,\tau_n].
     $$  
	Therefore, for every nonnegative progressively measurable process $X$ and every $t\in[0,T]$,  
	$$
	\int_0^{t\wedge\tau_n}X_s\,a_s\,ds\;\le\;n\int_0^{t\wedge\tau_n}X_s\,ds,
	$$  
	and the same bound holds with $a_s$ replaced by $\eta_s$ or $\delta_s$.
\end{remark}


for $U\in G_{loc}(N)$, we introduce the following norm:
\begin{equation*}
	\begin{split}
		\|U\|_{\mathbb{L}^p(\mathbb{L}^2_{Q \otimes \mathrm{Leb}})+\mathbb{L}^p(\mathbb{L}^p_{Q \otimes \mathrm{Leb}})}
		:=
		\inf_{U^1+U^2=U}
		\Bigg\{
		&\left(
		\mathbb{E}\left[
		\left(
		\int_0^T\int_{\mathcal{U}} |U^1_s(e)|^2Q(de)ds
		\right)^{\frac{p}{2}}
		\right]
		\right)^{\frac{1}{p}}  \\
		&+
		\left(
		\mathbb{E}\left[
		\int_0^T\int_{\mathcal{U}} |U^2_s(e)|^pQ(de)ds
		\right]
		\right)^{\frac{1}{p}}
		\Bigg\},
	\end{split}
\end{equation*}
where the infimum is taken over all decompositions $U=U^1+U^2$ with
$(U^1,U^2)\in \mathbb{L}^p(\mathbb{L}^2_{Q \otimes \mathrm{Leb}})\times \mathbb{L}^p(\mathbb{L}^p_{Q \otimes \mathrm{Leb}})$. This follows the standard definition of the sum of two Banach spaces; see, for instance, \cite{Krein1982}.

Since we have changed the functional space for the jump component $u$ in the BSDE \eqref{basic BSDE}, we need some auxiliary results.

The first one is the so-called {Bichteler-Jacod} inequality. To this end, let $\beta>0$ and $U \in \mathcal{L}^{p}_{Q,\beta}$. We define the following purely discontinuous martingale:
\begin{equation}\label{def N}
	\mathcal{N}_t:=\int_{0}^{t}e^{\frac{\beta}{2}  A_s} \int_{E} U_s(e) \tilde{N}(ds,de),\qquad t \in [0,T].
\end{equation}
The quadratic variation of this martingale is given by
$$
\left[\mathcal{N}\right]_t=\int_{0}^{t}e^{\beta A_s} \int_{E} |U_s(e)|^2 N(ds,de),\qquad t \in [0,T].
$$
We can now state the following lemma.
\begin{lemma}[Bichteler-Jacod inequality]\label{Bichteler-Jacod}
	Let $\beta>0$. There exist two universal constants $\mathfrak{c}_p>0$ and $\mathfrak{C}_p>0$
	such that, for any $U \in \mathcal{L}^{p}_{Q,\beta}$, if $\mathcal{N}$ is defined by \eqref{def N}, then
	$$
	\mathfrak{c}_p \left( \mathbb{E}\left[\left[\mathcal{N}\right]^{\frac{p}{2}}_T\right]\right)^{\frac{1}{p}} \leq \| e^{\frac{\beta}{2} A_\cdot} U \|_{\mathbb{L}^p\left(\mathbb{L}^2_Q\right)+\mathbb{L}^p\left(\mathbb{L}^p_Q\right)} \leq \mathfrak{C}_p \left( \mathbb{E}\left[\left[\mathcal{N}\right]^{\frac{p}{2}}_T\right]\right)^{\frac{1}{p}}.
	$$
\end{lemma}

The proof of Lemma \ref{Bichteler-Jacod} can be found in \cite[Lemma 1]{Kruse2017}. Our formulation follows the original statement of this result given by Marinelli and Röckner \cite[Theorem 1]{Marinelli2014}, together with the comments on pages 297 and 298 concerning the case of a Poisson random measure.

Following Lemma \ref{Bichteler-Jacod}, we shall also need the following auxiliary result.
\begin{lemma}\label{add results}
	Let $\beta>0$. There exists a universal constant $\mathfrak{C}_{p,T}>0$ such that, for any $U \in \mathcal{L}^{p}_{Q,\beta}$ and $\mathcal{N}$ defined by \eqref{def N}, 
	$$
	\mathbb{E}\left[\int_{0}^{T} e^{\frac{p}{2}\beta  A_s}  \left\| U_s \right\|^p_{\mathbb{L}^p_Q+\mathbb{L}^2_Q} ds \right] \leq \mathfrak{C}_{p,T} \mathbb{E}\left[\left[\mathcal{N}\right]^{\frac{p}{2}}_T\right].
	$$
	Moreover, if a function $\varPsi$ defined on $[0,T] \times E$ belongs to $\mathbb{L}^1_{Q \otimes \textnormal{Leb}}+\mathbb{L}^2_{Q \otimes \textnormal{Leb}}$, then for any $\varrho>0$ we have
	$$
	\int_{0}^{T} e^{\frac{\varrho}{2} \beta A_s} \left\| \varPsi_s \right\|_{\mathbb{L}^1_{Q }+\mathbb{L}^2_{Q}}ds \leq \mathfrak{C}_T \|e^{\frac{\varrho}{2} \beta A_\cdot} \varPsi \|_{\mathbb{L}^1_{Q \otimes \textnormal{Leb}}+\mathbb{L}^2_{Q \otimes \textnormal{Leb}}}\quad \mathbb{P}\text{-a.s.}
	$$
\end{lemma}

The statement of Lemma \ref{add results} is similar to that given in \cite[Lemma 2]{Kruse2017}. We follow an analogous proof adapted to our setting.

\begin{proof}
	Let $U^1 \in \mathbb{L}^p \left(\mathbb{L}^2_{Q \otimes \textnormal{Leb}}\right)$ and $U^2 \in \mathbb{L}^p \left(\mathbb{L}^p_{Q \otimes \textnormal{Leb}}\right)$ be such that $U=U^1+U^2$. Since the function $x \mapsto x^{\frac{p}{2}}$ is concave, we obtain, by the definition of the norm in $\mathbb{L}^p$ and Jensen's inequality on $[0,T]$,
	$$
	\mathbb{E}\left[ \int_{0}^{T} e^{\frac{p}{2}\beta  A_s}  \left\| U^1_s \right\|^p_{\mathbb{L}^2_Q} ds\right] =	\mathbb{E}\left[ \int_{0}^{T}  \left(\int_{{E}} \left| e^{\frac{\beta}{2}  A_s} U^1_s(e) \right|^2 Q(de) \right)^{\frac{p}{2}} ds\right]  \leq T^{1-\frac{p}{2}}\| e^{\frac{\beta}{2} A_\cdot} U^1  \|^p_{\mathbb{L}^p \left(\mathbb{L}^2_{Q \otimes \textnormal{Leb}}\right)}
	$$
	and
	$$
	\mathbb{E}\left[ \int_{0}^{T} e^{\frac{p}{2}\beta  A_s}  \left\| U^2_s \right\|^p_{\mathbb{L}^p_Q} ds\right] =	\mathbb{E}\left[ \int_{0}^{T}  e^{\frac{p}{2}\beta  A_s}\int_{{E}} \left| U^2_s(e) \right|^p Q(de) ds\right]  = \| e^{\frac{\beta}{2} A_\cdot} U^2  \|^p_{\mathbb{L}^p \left(\mathbb{L}^p_{Q \otimes \textnormal{Leb}}\right)}.
	$$
	Finally, it remains to apply the Bichteler-Jacod inequality from Lemma \ref{Bichteler-Jacod}.
	
	We now prove the second inequality. First, we point out that, by Remark \ref{rmq p}, the process $A$ is locally bounded. Therefore, using a monotone convergence argument, it is enough to prove the result when $A$ is bounded. Hence, without loss of generality, we may assume that $A_T$ is $\mathbb{P}$-a.s. bounded. Let us take an arbitrary decomposition of $\varPsi_s$, for $s \in [0,T]$, in the space $\mathbb{L}^1_{Q }+\mathbb{L}^2_{Q}$, namely $\varPsi_s=\varPsi_s^1+\varPsi_s^2$ with $\left(\varPsi_s^1, \varPsi_s^2\right) \in \mathbb{L}^1_{Q } \times \mathbb{L}^2_{Q}$. Then we have
	$$
	\int_{0}^{T} e^{\frac{\varrho}{2} \beta A_s} \left\| \varPsi_s \right\|_{\mathbb{L}^1_{Q }+\mathbb{L}^2_{Q}}ds \leq \int_{0}^{T} e^{\frac{\varrho}{2} \beta A_s} \left\| \varPsi^1_s \right\|_{\mathbb{L}^1_{Q}}ds +\int_{0}^{T} e^{\frac{\varrho}{2} \beta A_s} \left\| \varPsi^2_s \right\|_{\mathbb{L}^2_{Q}}ds.
	$$
	Again, by Jensen's inequality, we obtain
	\begin{equation*}
		\begin{split}
			\int_{0}^{T} e^{\frac{\varrho}{2} \beta A_s} \left\| \varPsi_s \right\|_{\mathbb{L}^1_{Q }+\mathbb{L}^2_{Q}}ds &\leq \| e^{\frac{\varrho}{2} \beta A_\cdot}\Psi^1\|_{\mathbb{L}^1_{Q \otimes \textnormal{Leb}}} +\sqrt{T}\int_{0}^{T}\| e^{\frac{\varrho}{2} \beta A_s} \varPsi^2_s \|_{\mathbb{L}^2_{Q}}ds\\
			& \leq 1 \vee \sqrt{T} \left(\| e^{\frac{\varrho}{2} \beta A_\cdot} \varPsi \|_{\mathbb{L}^1_{Q \otimes \textnormal{Leb}}+\mathbb{L}^2_{Q \otimes \textnormal{Leb}}}+c\right),
		\end{split}
	\end{equation*}
	for any $c>0$, where the last inequality follows from the definition of the infimum in the norm $\left\| \cdot \right\|_{\mathbb{L}^1_{Q \otimes \textnormal{Leb}}+\mathbb{L}^2_{Q \otimes \textnormal{Leb}}}$.
	
	This completes the proof of Lemma \ref{add results}.
\end{proof}

\begin{remark}\label{Matlo rmk}
	In view of Remark \ref{rmq p}, the martingale \eqref{def N} is well-defined whenever its quadratic variation is suitably controlled. Moreover, by Lemma \ref{add results}, we deduce that $U_t$ belongs to $\mathbb{L}^p_Q+\mathbb{L}^2_Q$ for $t \in [0,T]$, since the quadratic variation can again be controlled in $\mathbb{L}^{\frac{p}{2}}(\Omega)$ for $p \in (1,2)$. Finally, by Lemma 3 in \cite{Kruse2017}, Lemma \ref{add results}, and Proposition 3.68 in \cite{Jacod1979}\footnote{Following the proof of Lemma \ref{add results} and the result given in \cite[Lemma 3]{Kruse2017}, such an estimate can be obtained.}, we obtain that, for any $\varrho>0$,
	$$
	\mathbb{E}\left[\int_{0}^{T}  e^{\frac{p}{2}\beta  A_s} \int_{E} \left(|U_s(e)|^2 \mathds{1}_{\{|U_s(e)| \leq \varrho\}}+|U_s(e)| \mathds{1}_{\{|U_s(e)| > \varrho\}}\right)Q(de)ds\right]<+\infty.
	$$
Additionally, by \cite[Lemma 3]{Kruse2017}, we have
$\mathbb{L}^p_Q+\mathbb{L}^2_Q \subset \mathbb{L}^1_Q+\mathbb{L}^2_Q$ for $p \geq 1$. The same conclusion remains valid when $Q$ is replaced by $Q \otimes \textnormal{Leb}$. In particular, applying this inclusion to the weighted function, we obtain the corresponding weighted version. More precisely, if
$e^{\frac{\beta}{2}A_\cdot}U \in \mathbb{L}^p_{Q \otimes \textnormal{Leb}}+\mathbb{L}^2_{Q \otimes \textnormal{Leb}}$, for instance by means of the Bichteler-Jacod inequality in Lemma \ref{Bichteler-Jacod} whenever the quadratic variation is controlled in $\mathbb{L}^{\frac{p}{2}}(\Omega)$, then
$e^{\frac{\beta}{2}A_\cdot}U \in \mathbb{L}^1_{Q \otimes \textnormal{Leb}}+\mathbb{L}^2_{Q \otimes \textnormal{Leb}}$.
\end{remark}

\begin{corollary}\label{cor L1L2 control}
	Let $\beta>0$. There exists a constant 
	$\mathfrak{C}_{p,T}>0$ such that, for any 
	$U \in \mathcal{L}^{p}_{Q,\beta}$ and for the martingale 
	$\mathcal{N}$ defined by \eqref{def N}, we have
	$$
	\mathbb{E}\left[\int_{0}^{T} e^{\frac{p}{2}\beta A_s}
	\left\| U_s \right\|^p_{\mathbb{L}^1_Q+\mathbb{L}^2_Q} ds \right]
	\leq
	\mathfrak{C}_{p,T}
	\mathbb{E}\left[\left[\mathcal{N}\right]^{\frac{p}{2}}_T\right].
	$$
\end{corollary}

\begin{proof}
	By Lemma \ref{add results}, we have
	$$
	\mathbb{E}\left[\int_{0}^{T} e^{\frac{p}{2}\beta A_s}
	\left\| U_s \right\|^p_{\mathbb{L}^p_Q+\mathbb{L}^2_Q} ds \right]
	\leq
	\mathfrak{C}_{p,T}
	\mathbb{E}\left[\left[\mathcal{N}\right]^{\frac{p}{2}}_T\right].
	$$
	Moreover, by Lemma 3 in \cite{Kruse2017}, recalled in Remark 
	\ref{Matlo rmk}, we have the continuous embedding
	$$
	\mathbb{L}^p_Q+\mathbb{L}^2_Q
	\subset
	\mathbb{L}^1_Q+\mathbb{L}^2_Q.
	$$
	Hence, up to a multiplicative constant depending only on $p$, we have
	$$
	\left\| U_s \right\|_{\mathbb{L}^1_Q+\mathbb{L}^2_Q}
	\leq
	C_p
	\left\| U_s \right\|_{\mathbb{L}^p_Q+\mathbb{L}^2_Q},
	\qquad ds\otimes d\mathbb{P}\text{-a.e.}
	$$
	Therefore,
	$$
	\mathbb{E}\left[\int_{0}^{T} e^{\frac{p}{2}\beta A_s}
	\left\| U_s \right\|^p_{\mathbb{L}^1_Q+\mathbb{L}^2_Q} ds \right]
	\leq
	C_p^p
	\mathbb{E}\left[\int_{0}^{T} e^{\frac{p}{2}\beta A_s}
	\left\| U_s \right\|^p_{\mathbb{L}^p_Q+\mathbb{L}^2_Q} ds \right].
	$$
	Combining this estimate with Lemma \ref{add results} and absorbing 
	$C_p^p$ into the constant $\mathfrak{C}_{p,T}$ gives the desired inequality.
\end{proof}

Finally, we conclude the preliminary results of this section with an estimate that will be useful in the next proposition.
\begin{lemma}\label{lemma mrd}
	Let $k>0$, $\rho \in (0,+\infty)$, and $(a,b) \in \mathbb{R} \times \mathbb{R}$. Set  
	$$
	\varrho(\rho,p):=\sqrt{\frac{1}{2}\left(\frac{p-1}{2 \rho}\right)^{\frac{2}{2-p}} +\frac{1}{2}}-1.
	$$
	Then, there exists $\rho_{p,k}\in\left(0,\frac{p-1}{2}\right)$\footnote{The admissible range of the constants $\rho_{p,k}$, $k>0$, is independent of $k$ and depends only on $p\in(1,2)$ (see the proof of Lemma 5 in \cite{Kruse2017} for more details).} such that
	$$
	2k p|a|^{p-1} |b| \mathds{1}_{|b| \geq \varrho(\rho_{p,k},p)|a| }+p\rho_{p,k} |a|^{p-2} |b|^2 \mathds{1}_{|b| < \varrho( \rho_{p,k},p) |a|} \leq  |a+b|^p-|a|^p-p|a|^{p-2} a b \mathds{1}_{a \neq 0}.
	$$
\end{lemma}
The proof of Lemma \ref{lemma mrd} can be found in \cite[Lemma 5]{Kruse2017}. We also refer to \cite[Lemma 2]{Marinelli2014} for a similar result.

Let $(Y^1,Z^1,U^1)$ and $(Y^2,Z^2,U^2)$ be two $\mathbb{L}^p$-solutions of the BSDE \eqref{basic BSDE}, associated respectively with the data $(\xi_1,f_1)$ and $(\xi_2,f_2)$, and satisfying assumptions \textsc{(H2)}--\textsc{(H6)}. For $\mathcal{R}\in\{Y,Z,U,\xi,f\}$, we set $\widehat{\mathcal{R}}=\mathcal{R}^1-\mathcal{R}^2.$
\begin{proposition}\label{Estimation p less stricly than 2}
	There exist constants $\beta^\ast_{p,\epsilon}$ and $\mathfrak{C}_{p,T}$ such that, for every $\beta \geq \beta^\ast_{p,\epsilon}$, we have
	\begin{equation*}
		\begin{split}
			&\mathbb{E}\left[\sup_{t \in [0,T]} e^{\frac{p}{2}\beta A_{t  }}\big|\widehat{Y}_{t}\big|^p \right]+\mathbb{E}\int_{0}^{T}e^{\frac{p}{2}\beta A_s} \big|\widehat{Y}_s\big|^p dA_s+\mathbb{E}\left[\left(\int_{0}^{T}e^{\beta A_s} \big\|\widehat{Z}_s\big\|^2 ds\right)^{\frac{p}{2}}\right]\\
			&+\mathbb{E}\left[\left(\int_{0}^{T}e^{\beta A_s} \big\|\widehat{U}_s\big\|^2_{\mathbb{L}^2_Q} ds\right)^{\frac{p}{2}}\right]+\mathbb{E}\left[\left(\int_{0}^{T}e^{\beta A_s}\int_{E} \big|\widehat{U}_s(e)\big|^2 N(ds,de)\right)^{\frac{p}{2}}\right]\\
			&\leq  \mathfrak{c}_{p,T} \left(\mathbb{E}\left[ e^{\frac{p}{2}\beta A_{T}}\big|\widehat{\xi}\big|^p\right]  +\mathbb{E}\int_{0}^{T}e^{\beta A_s}\left|\widehat{f}(s,Y^2_s,Z^2_s,U_s^2)\right|^pds\right).
		\end{split}
	\end{equation*}
\end{proposition}

\begin{proof}
	Let $t \in [0,T]$ and $\tau \in \mathcal{T}_{0,T}$ and $\mu > 0$. By applying Lemma \ref{Lemma Ito for p less than 2} to the semimartingale $(e^{\frac{p}{2}\beta A_{t}-\mu t}\big|\widehat{Y}_{t}\big|^p)_{t \leq T}$ on $[t \wedge \tau,\tau]$, we obtain
	\begin{equation}\label{basic ito for p<2 with mu}
		\begin{split}
			& e^{\frac{p}{2}\beta A_{\tau}+\mu \tau}\big|\widehat{Y}_{ \tau}\big|^p\\
			&=e^{\frac{p}{2}\beta A_{t \wedge \tau}+\mu (t \wedge \tau) }\big|\widehat{Y}_{t \wedge \tau}\big|^p+\frac{p}{2}\beta\int_{t \wedge \tau}^{\tau}e^{\frac{p}{2}\beta A_s+\mu s} \big|\widehat{Y}_s\big|^p dA_s+\mu \int_{t \wedge \tau}^{\tau}e^{\frac{p}{2}\beta A_s+\mu s} \big|\widehat{Y}_s\big|^p ds\\
			&-p\int_{t \wedge \tau}^{\tau}e^{\frac{p}{2}\beta A_s+\mu s} \big|\widehat{Y}_s\big|^{p-1}\check{\widehat{Y}}_s  \left(f_1(s,Y^1_s, Z^1_s, U^2_s)-f_2(s,Y^2_s, Z^2_s, U^2_s)\right)ds\\
			&+p\int_{t \wedge \tau}^{\tau}e^{\frac{p}{2}\beta A_s} \big|\widehat{Y}_s\big|^{p-1} \check{\widehat{Y}}_s \widehat{Z}_s dW_s+p\int_{t \wedge \tau}^{\tau}e^{\frac{p}{2}\beta A_s+\mu s}\int_{E} \big|\widehat{Y}_{s-}\big|^{p-1} \check{\widehat{Y}}_{s-} \widehat{U}_s(e)\tilde{N}(ds,de)\\
			&+\frac{p}{2}\int_{t \wedge \tau}^{\tau}e^{\frac{p}{2}\beta A_s+\mu s}\big|{\widehat{Y}}_{s}\big|^{p-2}\mathds{1}_{\{Y_{s}\neq 0\}}\left\{(2-p)\left(\big\|Z_s\big\|^2-\check{\widehat{Y}}^\ast \widehat{Z}_s \widehat{Z}_s^\ast \check{\widehat{Y}}_s\right)+(p-1)\big\|\widehat{Z}_s\big\|^2\right\}ds\\
			&+\int_{t \wedge \tau}^{\tau} e^{\frac{p}{2}\beta A_s+\mu s}\int_{E}\left\{\big|\widehat{Y}_{s-} +\widehat{U}_s(e)\big|^p -\big|\widehat{Y}_{s-}\big|^p-p \big|\widehat{Y}_{s-}\big|^{p-1}\check{\widehat{Y}}_{s-}\widehat{U}_s(e)\right\}N(ds,de).
		\end{split}
	\end{equation}
	Using assumptions \textsc{(H2)} and \textsc{(H3)} on the generator $f$ along with Remark \ref{rmq essential} (for the the last line)
	\begin{equation}\label{O4}
		\begin{split}
			&p\big|\widehat{Y}_s\big|^{p-1}\check{\widehat{Y}}_s \left(f_1(s,Y^1_s, Z^1_s, U^1_s)-f_2(s,Y^2_s, Z^2_s, U^2_s)\right)\\
			& = p\mathds{1}_{\{\widehat{Y}_s \neq 0\}}\big|\widehat{Y}_s\big|^{p-2} \widehat{Y}_t\left(f_1(s,Y^1_s, Z^1_s, U^1_s)-f_2(s,Y^2_s, Z^2_s, U^2_s)\right)\\
			& \leq p\mathds{1}_{\{\widehat{Y}_s \neq 0\}}\big|\widehat{Y}_s\big|^{p-2}\left(\alpha_s \big|\widehat{Y}_s\big|^2 + \eta_s \big|\widehat{Y}_s\big| \big\|\widehat{Z}_s\big\|+\delta_s \big|\widehat{Y}_s\big| \big\|\widehat{U}_s\big\|_{\mathbb{L}^1_{Q}+\mathbb{L}^2_{Q}}+\big|\widehat{Y}_s\big|^{}\left|\widehat{f}(s,Y^2_s,Z^2_s,U_s^2)\right| \right)\\
			& \leq p\alpha_s \big|\widehat{Y}_s\big|^p+p\big|\widehat{Y}_s\big|^{p-1}\left(\eta_s  \big\|\widehat{Z}_s\big\|+\delta_s  \big\|\widehat{U}_s\big\|_{\mathbb{L}^1_{Q}+\mathbb{L}^2_{Q}}  +\left|\widehat{f}(s,Y^2_s,Z^2_s,U_s^2)\right|\right)\mathds{1}_{\{\widehat{Y}_s \neq 0\}}\\
			&\leq  p\left(\alpha_s+\frac{1}{2\varepsilon} a^2_s\right) \big|\widehat{Y}_s\big|^p +p\frac{\varepsilon}{2}\big|\widehat{Y}_s\big|^{p-2}\big\|\widehat{Z}_s\big\|^2\mathds{1}_{\{\widehat{Y}_s\neq 0\}} +p \big|\widehat{Y}_s\big|^{p-1} \big\|\widehat{U}_s\big\|_{\mathbb{L}^1_{Q}+\mathbb{L}^2_{Q}}\delta_s+p\big|\widehat{Y}_s\big|^{p-1}\left|\widehat{f}(s,Y^2_s,Z^2_s,U_s^2)\right|\\
			&\leq p\frac{\varepsilon}{2}\big|\widehat{Y}_s\big|^{p-2}\big\|\widehat{Z}_s\big\|^2\mathds{1}_{\{\widehat{Y}_s\neq 0\}} +p \big|\widehat{Y}_s\big|^{p-1} \big\|\widehat{U}_s\big\|_{\mathbb{L}^1_{Q}+\mathbb{L}^2_{Q}}\delta_s+p\big|\widehat{Y}_s\big|^{p-1}\left|\widehat{f}(s,Y^2_s,Z^2_s,U_s^2)\right|.
		\end{split}
	\end{equation}
On the other hand, from the Cauchy-Schwarz inequality we have $\big\|Z_s\big\|^2-\check{\widehat{Y}}^\ast \widehat{Z}_s \widehat{Z}_s^\ast \check{\widehat{Y}}_s \geq 0$, then
$$
\frac{p}{2}\left((2-p)\left(\big\|Z_s\big\|^2-\check{\widehat{Y}}^\ast \widehat{Z}_s \widehat{Z}_s^\ast \check{\widehat{Y}}_s\right)+(p-1)\big\|\widehat{Z}_s\big\|^2\right) \geq \mathfrak{b}_p \big\|\widehat{Z}_s\big\|^2.
$$
This, with taking $\varepsilon=\frac{\mathfrak{b}_p}{p}$, for any $t \in [0,T]$ and each $\tau\in\mathcal{T}_{0,T}$ and any $\mu\in\mathbb{R}$, we have:
\begin{equation}\label{Ito for p<2--v1.1}
	\begin{split}
		&e^{\frac{p}{2}\beta A_{t \wedge \tau}+\mu (t \wedge \tau) }\big|\widehat{Y}_{t \wedge \tau}\big|^p+\frac{p}{2}\beta\int_{t \wedge \tau}^{\tau}e^{\frac{p}{2}\beta A_s+\mu s} \big|\widehat{Y}_s\big|^p dA_s
		+\frac{\mathfrak{b}_p}{2}\int_{t \wedge \tau}^{\tau}e^{\frac{p}{2}\beta A_s+\mu s} \big|\widehat{Y}_s\big|^{p-2}\big\| \widehat{Z}_s \big\|^2 \mathds{1}_{\{\widehat{Y}_{s} \neq 0\}} ds\\
		&+\int_{t \wedge \tau}^{\tau} e^{\frac{p}{2}\beta A_s+\mu s}\int_{E}\left\{\big|\widehat{Y}_{s-} +\widehat{U}_s(e)\big|^p -\big|\widehat{Y}_{s-}\big|^p-p \big|\widehat{Y}_{s-}\big|^{p-1}\check{\widehat{Y}}_{s-}\widehat{U}_s(e)\right\}N(ds,de)\\
		\leq& e^{\frac{p}{2}\beta A_{\tau}+\mu \tau}\big|\widehat{Y}_{\tau}\big|^p-\mu \int_{t \wedge \tau}^{\tau}e^{\frac{p}{2}\beta A_s+\mu s} \big|\widehat{Y}_s\big|^p ds+p\int_{t \wedge \tau}^{\tau}e^{\frac{p}{2}\beta A_s-\mu s} \big|\widehat{Y}_s\big|^{p-1} \big\|\widehat{U}_s\big\|_{\mathbb{L}^1_{Q}+\mathbb{L}^2_{Q}}\delta_sds\\
		&+p\int_{t \wedge \tau}^{\tau}e^{\frac{p}{2}\beta A_s-\mu s} \big|\widehat{Y}_s\big|^{p-1}\left|\widehat{f}(s,Y^2_s,Z^2_s,U_s^2)\right|ds -p\int_{t \wedge \tau}^{\tau}e^{\frac{p}{2}\beta A_s+\mu s} \big|\widehat{Y}_s\big|^{p-1} \check{\widehat{Y}}_s \widehat{Z}_s dW_s \\
		&
		-p\int_{t \wedge \tau}^{\tau}e^{\frac{p}{2}\beta A_s+\mu s}\int_{E} \big|\widehat{Y}_{s-}\big|^{p-1} \check{\widehat{Y}}_{s-} \widehat{U}_s(e)\tilde{N}(ds,de).
	\end{split}
\end{equation}	
In particular, by using the convexity of the function $\mathbb{R}_+\ni x\mapsto x^p$, we obtain
\begin{equation}\label{N control new}
	\begin{split}
		0 & \leq \int_{t \wedge \tau}^{\tau} e^{\frac{p}{2}\beta A_s+\mu s}\int_{E}\left\{\big|\widehat{Y}_{s-} +\widehat{U}_s(e)\big|^p -\big|\widehat{Y}_{s-}\big|^p-p \big|\widehat{Y}_{s-}\big|^{p-1}\check{\widehat{Y}}_{s-}\widehat{U}_s(e)\right\}N(ds,de)\\
	    & \leq e^{\frac{p}{2}\beta A_{\tau}+\mu \tau}\big|\widehat{Y}_{\tau}\big|^p-\mu \int_{t \wedge \tau}^{\tau}e^{\frac{p}{2}\beta A_s+\mu s} \big|\widehat{Y}_s\big|^p ds+p\int_{t \wedge \tau}^{\tau}e^{\frac{p}{2}\beta A_s-\mu s} \big|\widehat{Y}_s\big|^{p-1} \big\|\widehat{U}_s\big\|_{\mathbb{L}^1_{Q}+\mathbb{L}^2_{Q}}\delta_sds\\
	    &\quad+p\int_{t \wedge \tau}^{\tau}e^{\frac{p}{2}\beta A_s-\mu s} \big|\widehat{Y}_s\big|^{p-1}\left|\widehat{f}(s,Y^2_s,Z^2_s,U_s^2)\right|ds -p\int_{t \wedge \tau}^{\tau}e^{\frac{p}{2}\beta A_s+\mu s} \big|\widehat{Y}_s\big|^{p-1} \check{\widehat{Y}}_s \widehat{Z}_s dW_s \\
	    &\quad
	    -p\int_{t \wedge \tau}^{\tau}e^{\frac{p}{2}\beta A_s+\mu s}\int_{E} \big|\widehat{Y}_{s-}\big|^{p-1} \check{\widehat{Y}}_{s-} \widehat{U}_s(e)\tilde{N}(ds,de).
	\end{split}
\end{equation}
Now, since $(\widehat{Y},\widehat{U})\in\mathcal{S}^p_{\beta}\times\mathcal{L}^{p}_{N,\beta}$, Young's inequality $ab\leq \frac{p-1}{p}a^{\frac{p}{p-1}}+\frac{1}{p}b^p$, together with Lemmas \ref{Bichteler-Jacod} and \ref{add results}, as well as Remark \ref{Matlo rmk}, yields
\begin{equation}\label{O0}
	\begin{split}
		&p \mathbb{E}\left[\int_{0}^{T} e^{\frac{p}{2} \beta A_s} \big|\widehat{Y}_s\big|^{p-1} \big\|\widehat{U}_s\big\|_{\mathbb{L}^1_{Q}+\mathbb{L}^2_{Q}}\delta_sds\right] \\
		& \leq p \mathbb{E}\left[\left(\sup_{0\leq t\leq T} e^{\frac{p-1}{2}\beta A_t} |\widehat{Y}_t|^{p-1}\right)\left(\int_{0}^{T} e^{\frac{1}{2} \beta A_s}\big\|\widehat{U}_s\big\|_{\mathbb{L}^1_{Q}+\mathbb{L}^2_{Q}}ds \right)\right] \\
		&\leq (p-1)\mathbb{E}\left[\sup_{0\leq t\leq T}e^{\frac{p}{2}\beta A_t} |\widehat{Y}_t|^{p}\right] +\mathbb{E}\left[\left(\int_{0}^{T} e^{\frac{1}{2} \beta A_s}\big\|\widehat{U}_s\big\|_{\mathbb{L}^1_{Q}+\mathbb{L}^2_{Q}}ds\right)^p\right] <+\infty.
	\end{split}
\end{equation}
Using Young's inequality once again, together with the definition of the process $A$ given in \textsc{(H5)}, we get
\begin{equation*}
	\begin{split}
		&p\mathbb{E}\left[\int_{0}^{T}e^{\frac{p}{2} \beta A_s}\big|\widehat{Y}_s\big|^{p-1}\left|\widehat{f}(s,Y^2_s,Z^2_s,U_s^2)\right|ds\right]\\
		&=p\mathbb{E}\left[\int_{0}^{T}e^{\frac{p}{2} \beta A_s}\big|\widehat{Y}_s\big|^{p-1}|a_s|\left|\frac{\widehat{f}(s,Y^2_s,Z^2_s,U_s^2)}{a_s}\right|ds\right]\\
		& \leq (p-1)\mathbb{E}\left[\int_{0}^{T}e^{\frac{p}{2} \beta A_s}\big|\widehat{Y}_s\big|^{p} dA_s\right]+\mathbb{E}\left[\int_{0}^{T}e^{\frac{p}{2} \beta A_s}\left|\frac{\widehat{f}(s,Y^2_s,Z^2_s,U_s^2)}{a_s}\right|^p ds\right].
	\end{split}
\end{equation*}
Using assumption \textsc{(H3)}, the fact that $a_s\geq \eta_s\vee\delta_s$, and the elementary inequality \eqref{Dizzy}, we obtain
\begin{equation*}
	\begin{split}
		&\mathbb{E}\left[\int_{0}^{T}e^{\frac{p}{2} \beta A_s}\left|\frac{\widehat{f}(s,Y^2_s,Z^2_s,U_s^2)}{a_s}\right|^p ds\right]\\
		&\leq 2^{p-1}\left(\mathbb{E}\left[\int_{0}^{T}e^{\frac{p}{2} \beta A_s}\left|\frac{{f}_1(s,Y^2_s,Z^2_s,U_s^2)-{f}_1(s,Y^2_s,0,0)}{a_s}\right|^p ds\right]+\mathbb{E}\left[\int_{0}^{T}e^{\frac{p}{2} \beta A_s}\left|\frac{\widehat{f}(s,Y^2_s,0,0)}{a_s}\right|^p ds\right]\right.\\
		&\qquad\qquad\qquad\left.+\mathbb{E}\left[\int_{0}^{T}e^{\frac{p}{2} \beta A_s}\left|\frac{{f}_2(s,Y^2_s,0,0)-{f}_2(s,Y^2_s,Z^2_s,U_s^2)}{a_s}\right|^p ds\right]\right)\\
		&\leq 2^{p-1} \left(2\mathbb{E}\left[\int_{0}^{T}e^{\frac{p}{2} \beta A_s}\big|{Z}^2_s\big|^{p}ds\right]+2\mathbb{E}\left[\int_{0}^{T}e^{\frac{p}{2} \beta A_s}\big\|{U}^2_s\big\|^{p}_{\mathbb{L}^1_{Q}+\mathbb{L}^2_{Q}} ds\right]+\mathbb{E}\left[\int_{0}^{T}e^{\frac{p}{2} \beta A_s}\left|\frac{\widehat{f}(s,Y^2_s,0,0)}{a_s}\right|^p ds\right]\right).
	\end{split}
\end{equation*}
Using inequality \eqref{Impo_Ine} and the fact that $\widehat{Z}\in\mathcal{H}^{p}_\beta$, we get
$$
\mathbb{E}\left[\int_{0}^{T}e^{\frac{p}{2} \beta A_s}\big|{Z}^2_s\big|^{p}ds\right] \leq T^{1-\frac{p}{2}}\mathbb{E}\left[\left(\int_{0}^{T}e^{\beta A_s}\big|{Z}^2_s\big|^{2}ds\right)^{\frac{p}{2}}\right]<+\infty.
$$
From Corollary \ref{cor L1L2 control} and the fact that $\widehat{U}\in\mathcal{L}^{p}_{N,\beta}$, we deduce that there exists a constant $\mathfrak{C}_{p,T}$ such that
$$
\mathbb{E}\left[\int_{0}^{T}e^{\frac{p}{2} \beta A_s}\big\|{U}^2_s\big\|^{p}_{\mathbb{L}^1_{Q}+\mathbb{L}^2_{Q}} ds\right] \leq \mathfrak{C}_{p,T} \left(\mathbb{E}\left[\left(\int_0^T\int_{{E}}e^{\beta A_s}|U^2_s(e)|^2N(ds,de)\right)^\frac{p}{2}\right]\right)<+\infty.
$$
Finally, by assumptions \textsc{(H4)}--\textsc{(H6)}, we have $a_s\geq \sqrt{\phi_s}$. Hence, using the fact that $\widehat{Y}\in\mathcal{S}^{p,A}_{\beta}$, we also obtain
\begin{equation*}
	\begin{split}
		\mathbb{E}\left[\int_{0}^{T}e^{\frac{p}{2} \beta A_s}\left|\frac{\widehat{f}(s,Y^2_s,0,0)}{a_s}\right|^p ds\right] \leq \frac{2}{\epsilon^p}\left[\int_{0}^{T}e^{\frac{p}{2} \beta A_s}\left|\varphi_s\right|^p ds\right]+\frac{2}{\epsilon^{\frac{p(3-p)}{2(p-1)}}}\mathbb{E}\left[\int_{0}^{T}e^{\frac{p}{2} \beta A_s}\left|Y^2_s\right|^p dA_s\right]<+\infty.
	\end{split}
\end{equation*}
Hence, we deduce that
\begin{equation}\label{O1}
	\mathbb{E}\left[\int_{0}^{T}e^{\frac{p}{2} \beta A_s}\big|\widehat{Y}_s\big|^{p-1}\left|\widehat{f}(s,Y^2_s,Z^2_s,U_s^2)\right|ds\right]<+\infty.
\end{equation}

Next, we observe that the stochastic integrals appearing on the right-hand side of \eqref{N control new} are uniformly integrable martingales. Indeed, by the BDG inequality (see Remark \ref{rmq p}) and Young's inequality, we have
\begin{equation*}
	\begin{split}
		&\mathbb{E}\left[\sup_{t \in [0,T]}\left|\int_{0}^{t}e^{\frac{p}{2}\beta A_s+\mu s} \big|\widehat{Y}_s\big|^{p-1} \check{\widehat{Y}}_s \widehat{Z}_s dW_s\right|\right]\\
		&\leq e^{\mu T} \mathfrak{C}_T \mathbb{E}\left[\left( \sup_{t \in [0,T]}e^{\frac{p-1}{2}\beta A_s}\big|\widehat{Y}_s\big|^{p-1}\right) \left(\int_{0}^{T}e^{\beta A_s}\big\|\widehat{Z}_s\big\|^2 ds\right)^{\frac{1}{2}}\right]\\
		&\leq \mathfrak{C}_T\left(\frac{p}{p-1}\mathbb{E}\left[\sup_{t \in [0,T]}e^{\frac{p}{2}\beta A_s}\big|\widehat{Y}_s\big|^{p}\right] +\frac{1}{p}\mathbb{E}\left[\left( \int_{0}^{T}e^{\beta A_s}\big\|\widehat{Z}_s\big\|^2 ds\right)^{\frac{p}{2}}\right] \right)<+\infty.
	\end{split}
\end{equation*}
Similarly, we have
\begin{equation*}
	\begin{split}
		&\mathbb{E}\left[\sup_{t \in [0,T]}\left|\int_{0}^{t}e^{\frac{p}{2}\beta A_s+\mu s}\int_{E} \big|\widehat{Y}_{s-}\big|^{p-1} \check{\widehat{Y}}_{s-} \widehat{U}_s(e)\tilde{N}(ds,de)\right|\right]\\
		&\leq \mathfrak{c} \mathbb{E}\left[\left(e^{p\beta A_s+2\mu s}\int_{E} \big|\widehat{Y}_{s-}\big|^{2(p-1)}  \big|\widehat{U}_s(e)\big|^2 N(ds,de)\right)^{\frac{1}{2}}\right]\\
		&\leq \mathfrak{C}_T\left(\frac{p}{p-1}\mathbb{E}\left[\sup_{t \in [0,T]}e^{\frac{p}{2}\beta A_s}\big|\widehat{Y}_s\big|^{p}\right] +\frac{1}{p}\mathbb{E}\left[\left( \int_{0}^{T}e^{\beta A_s}\big|\widehat{U}_s(e)\big|^2 N(ds,de)\right)^{\frac{p}{2}}\right] \right)<+\infty.
	\end{split}
\end{equation*}
Returning to \eqref{N control new} and using \eqref{O0}, \eqref{O1}, the martingale property of the stochastic integrals, together with the fact that $\widehat{Y}\in\mathcal{S}^p_{\beta}$ (or $\widehat{Y}\in\mathcal{S}^{p,A}_{\beta}$), which ensures the $p$-integrability of the Lebesgue integral, we derive
$$
\mathbb{E}\left[\int_{0}^{T} e^{\frac{p}{2}\beta A_s+\mu s}\int_{E}\left\{\big|\widehat{Y}_{s-} +\widehat{U}_s(e)\big|^p -\big|\widehat{Y}_{s-}\big|^p-p \big|\widehat{Y}_{s-}\big|^{p-1}\check{\widehat{Y}}_{s-}\widehat{U}_s(e)\right\}N(ds,de)\right]<+\infty.
$$
In particular, it follows from \cite[Lemma 3.67]{Jacod1979} that
$$
\mathbb{E}\left[\int_{0}^{T} e^{\frac{p}{2}\beta A_s+\mu s}\int_{E}\left\{\big|\widehat{Y}_{s-} +\widehat{U}_s(e)\big|^p -\big|\widehat{Y}_{s-}\big|^p-p \big|\widehat{Y}_{s-}\big|^{p-1}\check{\widehat{Y}}_{s-}\widehat{U}_s(e)\right\}Q(de)ds\right]<+\infty.
$$
Then, for any $t\in[0,T]$, each $\tau\in\mathcal{T}_{0,T}$, and any $\mu\in\mathbb{R}$, we can write
\begin{equation}\label{O2}
	\begin{split}
		&\int_{t \wedge \tau}^{\tau} e^{\frac{p}{2}\beta A_s+\mu s}\int_{E}\left\{\big|\widehat{Y}_{s-} +\widehat{U}_s(e)\big|^p -\big|\widehat{Y}_{s-}\big|^p-p \big|\widehat{Y}_{s-}\big|^{p-1}\check{\widehat{Y}}_{s-}\widehat{U}_s(e)\right\}N(ds,de)\\
		&=\int_{t \wedge \tau}^{\tau} e^{\frac{p}{2}\beta A_s+\mu s}\int_{E}\left\{\big|\widehat{Y}_{s-} +\widehat{U}_s(e)\big|^p -\big|\widehat{Y}_{s-}\big|^p-p \big|\widehat{Y}_{s-}\big|^{p-1}\check{\widehat{Y}}_{s-}\widehat{U}_s(e)\right\}\tilde{N}(ds,de)\\
		&\quad - \int_{t \wedge \tau}^{\tau} e^{\frac{p}{2}\beta A_s+\mu s}\int_{E}\left\{\big|\widehat{Y}_{s-} +\widehat{U}_s(e)\big|^p -\big|\widehat{Y}_{s-}\big|^p-p \big|\widehat{Y}_{s-}\big|^{p-1}\check{\widehat{Y}}_{s-}\widehat{U}_s(e)\right\}Q(de)ds.
	\end{split}
\end{equation}
Applying Lemma \ref{Lemma betap for inequality} to the jump part in \eqref{Ito for p<2--v1.1} and using equality \eqref{O2}, we can write
\begin{equation}\label{Ito for p<2--v1}
	\begin{split}
		&e^{\frac{p}{2}\beta A_{t \wedge \tau}+\mu (t \wedge \tau) }\big|\widehat{Y}_{t \wedge \tau}\big|^p+\frac{p}{2}\beta\int_{t \wedge \tau}^{\tau}e^{\frac{p}{2}\beta A_s+\mu s} \big|\widehat{Y}_s\big|^p dA_s
		+\frac{\mathfrak{b}_p}{2}\int_{t \wedge \tau}^{\tau}e^{\frac{p}{2}\beta A_s+\mu s} \big|\widehat{Y}_s\big|^{p-2}\big\| \widehat{Z}_s \big\|^2 \mathds{1}_{\{\widehat{Y}_{s} \neq 0\}} ds\\
		&+\frac{\mathfrak{b}_p}{2}\int_{t \wedge \tau}^\tau\int_{E}e^{\frac{p}{2}\beta A_s+\mu s}|\widehat{U}_s(e)|^2\left(|\widehat{Y}_{s-}|^{2}\vee|\widehat{Y}_{s-}+U_s(e)|^2\right)^{\frac{p-2}{2}}\mathds{1}_{\{|\widehat{Y}_{s-}|\vee|\widehat{Y}_{s-}+U_s(e)|\neq 0\}}N(ds,de)\\
		\leq& e^{\frac{p}{2}\beta A_{\tau}+\mu \tau}\big|\widehat{Y}_{ \tau}\big|^p+p\int_{t \wedge \tau}^{\tau}e^{\frac{p}{2}\beta A_s-\mu s} \big|\widehat{Y}_s\big|^{p-1} \big\|\widehat{U}_s\big\|_{\mathbb{L}^1_{Q}+\mathbb{L}^2_{Q}}\delta_sds \\
		&+p\int_{t \wedge \tau}^{\tau}e^{\frac{p}{2}\beta A_s+\mu s} \big|\widehat{Y}_s\big|^{p-1}\left|\widehat{f}(s,Y^2_s,Z^2_s,U_s^2)\right|ds -\mu \int_{t \wedge \tau}^{\tau}e^{\frac{p}{2}\beta A_s+\mu s} \big|\widehat{Y}_s\big|^p ds\\
		&-p\int_{t \wedge \tau}^{\tau}e^{\frac{p}{2}\beta A_s+\mu s} \big|\widehat{Y}_s\big|^{p-1} \check{\widehat{Y}}_s \widehat{Z}_s dW_s
		-p\int_{t \wedge \tau}^{\tau}e^{\frac{p}{2}\beta A_s+\mu s}\int_{E} \big|\widehat{Y}_{s-}\big|^{p-1} \check{\widehat{Y}}_{s-} \widehat{U}_s(e)\tilde{N}(ds,de)\\
		&-\frac{1}{2}\int_{t \wedge \tau}^{\tau} e^{\frac{p}{2}\beta A_s+\mu s}\int_{E}\left\{\big|\widehat{Y}_{s-} +\widehat{U}_s(e)\big|^p -\big|\widehat{Y}_{s-}\big|^p-p \big|\widehat{Y}_{s-}\big|^{p-1}\check{\widehat{Y}}_{s-}\widehat{U}_s(e)\right\}\tilde{N}(ds,de)\\
		&-\frac{1}{2}\int_{t \wedge \tau}^{\tau} e^{\frac{p}{2}\beta A_s+\mu s}\int_{E}\left\{\big|\widehat{Y}_{s-} +\widehat{U}_s(e)\big|^p -\big|\widehat{Y}_{s-}\big|^p-p \big|\widehat{Y}_{s-}\big|^{p-1}\check{\widehat{Y}}_{s-}\widehat{U}_s(e)\right\}Q(de)ds
	\end{split}
\end{equation}	
We now use the technical estimates stated in Lemma \ref{lemma mrd}: for each $k>0$, let $\varrho_{p,k}$ be chosen as in that lemma and set $\varrho_{p,k}=\varrho(\varrho_{p,k},p)|Y_{s-}|\mathds{1}_{\{Y_{s-}\neq0\}}+\varrho\mathds{1}_{\{Y_{s-}=0\}}$, where $\varrho>0$ is arbitrary; then, by assumption (H3), the definition of the norm in $\mathbb{L}^1_Q+\mathbb{L}^2_Q$, and Young's inequality, we obtain, for any $\varepsilon>0$,
\begin{equation*}
	\begin{split}
		&p\int_{t \wedge \tau}^\tau e^{\frac{p}{2}\beta A_s+\mu s}  |\widehat{Y}_{s}|^{p-1}\|\widehat{U}_{s}\|_{\mathbb{L}^1_Q+\mathbb{L}^2_Q} \delta_s ds\\
		& \leq p  \int_{t \wedge \tau}^\tau e^{\frac{p}{2}\beta A_s+\mu s}  | \widehat{Y}_{s}|^{p-1} \left(\| \widehat{U}_{s}\mathds{1}_{|\widehat{U}_{s}| < \varrho} \|_{\mathbb{L}^2_Q}+\| \widehat{U}_{s}\mathds{1}_{|\widehat{U}_{s}| \geq  \varrho} \|_{\mathbb{L}^1_Q}\right) \delta_s ds\\
		& \leq \frac{p}{2 \epsilon^{\frac{2-p}{2(p-1)}} \varepsilon} \int_{t \wedge \tau}^\tau e^{\frac{p}{2}\beta A_s} |\bar Y_{s}|^{p}dA_s +\frac{p \varepsilon}{2} \int_{t \wedge \tau}^\tau e^{\frac{p}{2}\beta A_s} |\widehat{Y}_{s-}|^{p-2} \mathds{1}_{\{\widehat{Y}_{s-} \neq 0\}} \| \widehat{U}_{s}\mathds{1}_{| \widehat{U}_{s}| < \varrho_{p,k}} \|^2_{\mathbb{L}^2_\lambda}ds\\
		&+  p  \int_{t \wedge \tau}^\tau e^{\frac{p}{2}\beta A_s} |\widehat{Y}_{s}|^{p-1} \|\widehat{U}_{s}\mathds{1}_{|\widehat{U}_{s}| \geq  \varrho_{p,k}} \|_{\mathbb{L}^1_Q} \delta_s ds.
	\end{split}
\end{equation*}
Plugging this into \eqref{Ito for p<2--v1}, we get
\begin{equation}\label{Ito for p<2--v1.2}
	\begin{split}
		&e^{\frac{p}{2}\beta A_{t \wedge \tau}+\mu (t \wedge \tau) }\big|\widehat{Y}_{t \wedge \tau}\big|^p+\frac{p}{2}\left(\beta-\frac{1}{\epsilon^{\frac{2-p}{2(p-1)}} \varepsilon}\right)\int_{t \wedge \tau}^{\tau}e^{\frac{p}{2}\beta A_s+\mu s} \big|\widehat{Y}_s\big|^p dA_s
		+\frac{\mathfrak{b}_p}{2}\int_{t \wedge \tau}^{\tau}e^{\frac{p}{2}\beta A_s+\mu s} \big|\widehat{Y}_s\big|^{p-2}\big\| \widehat{Z}_s \big\|^2 \mathds{1}_{\{\widehat{Y}_{s} \neq 0\}} ds\\
		&+\frac{\mathfrak{b}_p}{2}\int_{t \wedge \tau}^\tau\int_{E}e^{\frac{p}{2}\beta A_s+\mu s}|\widehat{U}_s(e)|^2\left(|\widehat{Y}_{s-}|^{2}\vee|\widehat{Y}_{s-}+U_s(e)|^2\right)^{\frac{p-2}{2}}\mathds{1}_{\{|\widehat{Y}_{s-}|\vee|\widehat{Y}_{s-}+U_s(e)|\neq 0\}}N(ds,de)\\
		\leq& e^{\frac{p}{2}\beta A_{\tau}+\mu \tau}\big|\widehat{Y}_{ \tau}\big|^p
		+p\int_{t \wedge \tau}^{\tau}e^{\frac{p}{2}\beta A_s+\mu s} \big|\widehat{Y}_s\big|^{p-1}\left|\widehat{f}(s,Y^2_s,Z^2_s,U_s^2)\right|ds -\mu \int_{t \wedge \tau}^{\tau}e^{\frac{p}{2}\beta A_s+\mu s} \big|\widehat{Y}_s\big|^p ds\\
		&-p\int_{t \wedge \tau}^{\tau}e^{\frac{p}{2}\beta A_s+\mu s} \big|\widehat{Y}_s\big|^{p-1} \check{\widehat{Y}}_s \widehat{Z}_s dW_s
		-p\int_{t \wedge \tau}^{\tau}e^{\frac{p}{2}\beta A_s+\mu s}\int_{E} \big|\widehat{Y}_{s-}\big|^{p-1} \check{\widehat{Y}}_{s-} \widehat{U}_s(e)\tilde{N}(ds,de)\\
		&-\frac{1}{2}\int_{t \wedge \tau}^{\tau} e^{\frac{p}{2}\beta A_s+\mu s}\int_{E}\left\{\big|\widehat{Y}_{s-} +\widehat{U}_s(e)\big|^p -\big|\widehat{Y}_{s-}\big|^p-p \big|\widehat{Y}_{s-}\big|^{p-1}\check{\widehat{Y}}_{s-}\widehat{U}_s(e)\right\}\tilde{N}(ds,de)\\
		&-\frac{1}{2}\int_{t \wedge \tau}^{\tau} e^{\frac{p}{2}\beta A_s+\mu s}\int_{E}\left\{\big|\widehat{Y}_{s-} +\widehat{U}_s(e)\big|^p -\big|\widehat{Y}_{s-}\big|^p-p \big|\widehat{Y}_{s-}\big|^{p-1}\check{\widehat{Y}}_{s-}\widehat{U}_s(e)\right\}Q(de)ds\\
		&+\frac{p \varepsilon}{2} \int_{t \wedge \tau}^\tau e^{\frac{p}{2}\beta A_s} |\widehat{Y}_{s-}|^{p-2} \mathds{1}_{\{\widehat{Y}_{s-} \neq 0\}} \| \widehat{U}_{s}\mathds{1}_{| \widehat{U}_{s}| < \varrho_{p,k}} \|^2_{\mathbb{L}^2_\lambda}ds+  p  \int_{t \wedge \tau}^\tau e^{\frac{p}{2}\beta A_s} |\widehat{Y}_{s}|^{p-1} \|\widehat{U}_{s}\mathds{1}_{|\widehat{U}_{s}| \geq  \varrho_{p,k}} \|_{\mathbb{L}^1_Q} \delta_s ds.
	\end{split}
\end{equation}	
Consider the sequence of stopping times $(\tau_k)_{k\geq1}$ defined in \eqref{SPR} and explained in Remark \ref{kech}. We then take $\tau:=\tau_k\wedge\vartheta_k$, where $(\vartheta_k)_{k\geq1}$ is a localization sequence for the local martingale
$$
\frac{1}{2}\int_{0}^{\cdot} e^{\frac{p}{2}\beta A_s+\mu s}\int_{E}\left\{\big|\widehat{Y}_{s-} +\widehat{U}_s(e)\big|^p -\big|\widehat{Y}_{s-}\big|^p-p \big|\widehat{Y}_{s-}\big|^{p-1}\check{\widehat{Y}}_{s-}\widehat{U}_s(e)\right\}\tilde{N}(ds,de).
$$
First, we point out that, by applying Lemma \ref{lemma mrd} to the last three terms in inequality \eqref{Ito for p<2--v1.2} and choosing $\varepsilon>0$ such that $\varepsilon\leq \rho_{p,k}$ for every $k\geq1$, with $\rho_{p,k}\in\left(0,\frac{p-1}{2}\right)$, we see that, for each $k\geq1$,
\begin{equation}\label{neg}
	\begin{split}
		&-\frac{1}{2}\int_{t \wedge \tau}^{\tau} e^{\frac{p}{2}\beta A_s+\mu s}\int_{E}\left\{\big|\widehat{Y}_{s-} +\widehat{U}_s(e)\big|^p -\big|\widehat{Y}_{s-}\big|^p-p \big|\widehat{Y}_{s-}\big|^{p-1}\check{\widehat{Y}}_{s-}\widehat{U}_s(e)\right\}Q(de)ds\\
		&+\frac{p \varepsilon}{2} \int_{t \wedge \tau}^\tau e^{\frac{p}{2}\beta A_s} |\widehat{Y}_{s-}|^{p-2} \mathds{1}_{\{\widehat{Y}_{s-} \neq 0\}} \| \widehat{U}_{s}\mathds{1}_{| \widehat{U}_{s}| < \varrho} \|^2_{\mathbb{L}^2_\lambda}ds+  p  \int_{t \wedge \tau}^\tau e^{\frac{p}{2}\beta A_s} |\widehat{Y}_{s}|^{p-1} \|\widehat{U}_{s}\mathds{1}_{|\widehat{U}_{s}| \geq  \varrho} \|_{\mathbb{L}^1_Q} \delta_s ds\\
		&\leq -\frac{1}{2}\int_{t \wedge \tau}^\tau \int_{E}e^{\frac{p}{2}\beta A_s}\left[| \widehat{Y}_{s-}+\widehat{U}_s(e)|^{p}-|\widehat{Y}_{s-}|^p-p|\widehat{Y}_{s-}|^{p-1}\check{ \widehat{Y}}_{s-} \widehat{U}_s(e)\right]Q(de)ds\\
		&+\frac{p \rho_{p,k}}{2} \int_{t \wedge \tau}^\tau e^{\frac{p}{2}\beta A_s} |Y_{s-}|^{p-2} \|\widehat{U}_{s}\mathds{1}_{|\widehat{U}_{s}| < \varrho_{p,k}} \|_{\mathbb{L}^2_Q}ds
		+  p k \int_{t \wedge \tau}^\tau e^{\frac{p}{2}\beta A_s} | \widehat{Y}_{s}|^{p-1} \|\bar V_{s}\mathds{1}_{|\widehat{U}_{s}| \geq  \varrho_{p,k}} \|_{\mathbb{L}^1_Q}ds \\
		&\leq 0.
	\end{split}
\end{equation} 
In addition, by Young's inequality, we have
\begin{equation}\label{O3}
	\begin{split}
	&p\int_{t \wedge \tau}^{\tau}e^{\frac{p}{2}\beta A_s+\mu s} \big|\widehat{Y}_s\big|^{p-1}\left|\widehat{f}(s,Y^2_s,Z^2_s,U_s^2)\right|ds\\
	&=p\int_{t \wedge \tau}^{\tau}e^{\frac{p-1}{2}\beta A_s+\frac{p-1}{p}\mu s} \big|\widehat{Y}_s\big|^{p-1}\left(e^{\frac{p}{2}\beta A_s+\frac{1}{p}\mu s}\left|\widehat{f}(s,Y^2_s,Z^2_s,U_s^2)\right|\right)ds\\
	&\leq (p-1)\int_{t \wedge \tau}^{\tau}e^{\frac{p}{2}\beta A_s+\mu s}\big|\widehat{Y}_s\big|^{p}ds+\int_{t \wedge \tau}^{\tau}e^{\frac{p}{2}\beta A_s+\mu s} \left|\widehat{f}(s,Y^2_s,Z^2_s,U_s^2)\right|^pds.
	\end{split}
\end{equation}
Plugging \eqref{neg} and \eqref{O3} into \eqref{Ito for p<2--v1.2}, taking expectations on both sides, and using the martingale property of the stochastic integrals on $[t\wedge\tau,\tau]$, where $\tau=\tau_k\wedge\vartheta_k$ as above, we obtain, after passing to the limit as $k\to+\infty$ with $t=0$,\footnote{Note that $e^{\frac{p}{2}\beta A_{\tau_k}}|\widehat{Y}_{\tau_k}|^p\to e^{\frac{p}{2}\beta A_T}|\widehat{\xi}|^p$ $\mathbb{P}$-a.s. as $k\to+\infty$. By the integrability condition (H6), we have $\mathbb{E}\left[e^{\frac{p}{2}\beta A_T}|\widehat{\xi}|^p\right]<+\infty$. Hence, by the dominated convergence theorem, we deduce that $e^{\frac{p}{2}\beta A_{\tau_k}}|\widehat{Y}_{\tau_k}|^p\to e^{\frac{p}{2}\beta A_T}|\widehat{\xi}|^p$ in $\mathbb{L}^1(\Omega)$ as $k\to+\infty$.} using the monotone convergence theorem and choosing $\beta\geq \frac{2}{p}\left(\frac{1}{\epsilon^{\frac{2-p}{2(p-1)}}\varepsilon}+1\right)=:\beta^\ast_{p,\epsilon}$ and $\mu>p-1$, the following estimate:
\begin{equation}\label{Ito for p<2--v1.3}
	\begin{split}
		&\mathbb{E}\left[ \int_{0}^{T}e^{\frac{p}{2}\beta A_s+\mu s} \big|\widehat{Y}_s\big|^p dA_s\right] 
		+\frac{\mathfrak{b}_p}{2}\mathbb{E}\left[ \int_{0}^{T}e^{\frac{p}{2}\beta A_s+\mu s} \big|\widehat{Y}_s\big|^{p-2}\big\| \widehat{Z}_s \big\|^2 \mathds{1}_{\{\widehat{Y}_{s} \neq 0\}} ds\right] \\
		&+\frac{\mathfrak{b}_p}{2}\mathbb{E}\left[ \int_{0}^T e^{\frac{p}{2}\beta A_s+\mu s}\int_{E}|\widehat{U}_s(e)|^2\left(|\widehat{Y}_{s-}|^{2}\vee|\widehat{Y}_{s-}+U_s(e)|^2\right)^{\frac{p-2}{2}}\mathds{1}_{\{|\widehat{Y}_{s-}|\vee|\widehat{Y}_{s-}+\widehat{U}_s(e)|\neq 0\}}N(ds,de)\right] \\
		&\leq \mathbb{E}\left[  e^{\frac{p}{2}\beta A_{T}+\mu T}\big|\widehat{\xi}\big|^p\right] 
		+\mathbb{E}\left[ \int_{0}^{T}e^{\frac{p}{2}\beta A_s+\mu s} \left|\widehat{f}(s,Y^2_s,Z^2_s,U_s^2)\right|^p ds\right] .
	\end{split}
\end{equation}	
Coming back to \eqref{Ito for p<2--v1.1}, using this formula with $\mu=0$ and $\tau=T$, and then taking the supremum and expectation on both sides, we get
\begin{equation}\label{sup for p less than 2}
	\begin{split}
		\mathbb{E}\left[\sup_{t \in [0,T]} e^{\frac{p}{2}\beta A_{t  }}\big|\widehat{Y}_{t}\big|^p \right]
		\leq& \mathbb{E}\left[ e^{\frac{p}{2}\beta A_{T}}\big|\widehat{\xi}\big|^p\right] +p\mathbb{E}\int_{0}^{T}e^{\frac{p}{2}\beta A_s} \big|\widehat{Y}_s\big|^{p-1}\left|\widehat{f}(s,Y^2_s,Z^2_s,U_s^2)\right|ds  \\
		&+p\mathbb{E}\left[\sup_{t \in [0,T]}\left|\int_{t}^{T}e^{\frac{p}{2}\beta A_s} \big|\widehat{Y}_s\big|^{p-1} \check{\widehat{Y}}_s \widehat{Z}_s dW_s\right|\right]\\
		& +p\mathbb{E}\left[\sup_{t \in [0,T]}\left|\int_{t}^{T}e^{\frac{p}{2}\beta A_s}\int_{E} \big|\widehat{Y}_{s-}\big|^{p-1} \check{\widehat{Y}}_{s-} \widehat{U}_s(e)\tilde{N}(ds,de)\right|\right].
	\end{split}
\end{equation}
Using the BDG inequality, we have
\begin{equation*}
	\begin{split}
		&p\mathbb{E}\left[\sup_{t \in [0,T]}\left|\int_{t}^{T}e^{\frac{p}{2}\beta A_s} \big|\widehat{Y}_s\big|^{p-1} \check{\widehat{Y}}_s \widehat{Z}_s dW_s\right|\right]\\
		&\leq p \mathfrak{c} \mathbb{E}\left[\left(\int_{0}^{T}e^{p\beta A_s} \big|\widehat{Y}_s\big|^{2(p-1)} \big\|\widehat{Z}_s\big\|^2 \mathds{1}_{\{\widehat{Y}_s \neq 0\}} ds\right)^{\frac{1}{2}}\right]\\
		&\leq \mathbb{E}\left[\left( \sup_{t \in [0,T]}e^{\frac{p}{4}\beta}\big|\widehat{Y}_s\big|^{\frac{p}{2}}\right) \left( p^2 \mathfrak{c}^2\int_{0}^{T}e^{\frac{p}{2}\beta A_s} \big|\widehat{Y}_s\big|^{p-2} \big\|\widehat{Z}_s\big\|^2 \mathds{1}_{\{\widehat{Y}_s \neq 0\}} ds\right)^{\frac{1}{2}}\right]\\
		&\leq \frac{1}{6} \mathbb{E}\left[\sup_{t \in [0,T]}e^{\frac{p}{2}\beta}\big|\widehat{Y}_s\big|^{p}\right]+\frac{3}{2}p^2 \mathfrak{c}^2\mathbb{E}\left[\int_{0}^{T}e^{\frac{p}{2}\beta A_s} \big|\widehat{Y}_s\big|^{p-2} \big\|\widehat{Z}_s\big\|^2 \mathds{1}_{\{\widehat{Y}_s \neq 0\}} ds\right]
	\end{split}
\end{equation*}
The last line follows from the elementary inequality $ab\leq \frac{1}{6}a^2+\frac{3}{2}b^2$, where we recall that $\mathfrak{c}>0$ denotes the universal BDG constant. By a similar computation, and using the localizing sequence of stopping times considered in Lemma \ref{Lemma betap for inequality}, we get
\begin{equation*}
	\begin{split}
	&\mathbb{E}\left[ \sup_{0\leq t\leq T}\left|\int_0^te^{\frac{p}{2}\beta A_s}\int_{E}|\widehat{Y}_{s-}|^{p-1}\check{\widehat{Y}}_{s-} \widehat{U}_s(e)\tilde{N}(ds,de)\right|\right] \\
	&\leq\mathfrak{c}\mathbb{E}\left[\left(\int_0^Te^{p\beta A_s} \int_{E}\left(|\widehat{Y}_{s-}|^{2}\vee|\widehat{Y}_{s-}+\widehat{U}_s(e)|^2\right)^{p-1}\mathds{1}_{\{|\widehat{Y}_{s-}|\vee|\widehat{Y}_{s-}+\widehat{U}_s(e)|\neq 0\}}|\widehat{U}_s(e)|^2 N(ds,de)\right)^{\frac{1}{2}}\right]\\
	&\leq\mathfrak{c}\mathbb{E}\left[\left( \sup_{0\leq t\leq T}e^{\frac{p}{4}\beta A_t}|\widehat{Y}_{t}|^{\frac{p}{2}}\right) \right. \\
	&\left. \qquad\qquad \times \left( \int_0^Te^{\frac{p}{2}\beta A_s}\int_{E}|\widehat{U}_s(e)|^2\left(|\widehat{Y}_{s-}|^{2}\vee|\widehat{Y}_{s-}+\widehat{U}_s(e)|^2\right)^{\frac{p-2}{2}}\mathds{1}_{\{|\widehat{Y}_{s-}|\vee|\widehat{Y}_{s-}+\widehat{U}_s(e)|\neq 0\}}N(ds,de)\right)^{\frac{1}{2}}\right]\\
	&\leq\frac{1}{6}\mathbb{E}\left[\sup_{0\leq t\leq T}e^{\frac{p}{2}\beta A_t}|\widehat{Y}_{t}|^{p}\right]\\
	&\qquad+\frac{3}{2}p^2 \mathfrak{c}^2\mathbb{E}\left[\int_0^Te^{\frac{p}{2}\beta A_s}\int_{E}|\widehat{U}_s(e)|^2\left(|\widehat{Y}_{s-}|^{2}\vee|\widehat{Y}_{s-}+\widehat{U}_s(e)|^2\right)^{\frac{p-2}{2}}\mathds{1}_{\{|\widehat{Y}_{s-}|\vee|\widehat{Y}_{s-}+\widehat{U}_s(e)|\neq 0\}}N(ds,de)\right].
	\end{split}
\end{equation*}
Finally, using Young's and Jensen's inequalities, we get
\begin{equation}\label{Paranoia}
	\begin{split}
		&p\mathbb{E}\left[\int_{0}^{T}e^{\frac{p}{2}\beta A_s} \big|\widehat{Y}_s\big|^{p-1}\left|\widehat{f}(s,Y^2_s,Z^2_s,U_s^2)\right|ds\right]\\
		&=p	\mathbb{E}\left[\int_{0}^{T}e^{\frac{p-1}{2}\beta A_s} \big|\widehat{Y}_s\big|^{p-1} e^{\frac{\beta}{2} A_s}\left|\widehat{f}(s,Y^2_s,Z^2_s,U_s^2)\right|ds\right]\\
		& \leq p\mathbb{E}\left[\left(\frac{1}{6(p-1)}\right)^{\frac{p-1}{p}}\left(\sup_{t \in [0,T]}e^{\frac{p-1}{2}\beta A_s} \big|\widehat{Y}_s\big|^{p-1}\right)\left(\frac{1}{6(p-1)}\right)^{\frac{1-p}{p}}\int_{0}^{T} e^{\frac{\beta}{2} A_s}\left|\widehat{f}(s,Y^2_s,Z^2_s,U_s^2)\right|ds\right]\\
		& \leq \frac{1}{6}\mathbb{E}\left[\sup_{t \in [0,T]}e^{\frac{p}{2}\beta A_s}\big|\widehat{Y}_t\big|^{p}\right]+\left(6 T(p-1)\right)^{p-1}\mathbb{E}\left[\int_{0}^{T} e^{\frac{p}{2}\beta A_s}\left|\widehat{f}(s,Y^2_s,Z^2_s,U_s^2)\right|^p ds\right].
	\end{split}
\end{equation}
	Plugging the three estimates above into \eqref{sup for p less than 2}, together with \eqref{Ito for p<2--v1.3}, we deduce that there exists a constant $\mathfrak{C}_{p,T}$ such that, for any $\beta\geq\beta^\ast_{p,\epsilon}$, we have
\begin{equation}\label{Y sup p less than 2}
	\begin{split}
		\mathbb{E}\left[\sup_{t \in [0,T]} e^{\frac{p}{2}\beta A_{t  }}\big|\widehat{Y}_{t}\big|^p \right]
		\leq & \mathfrak{C}_{p,T} \left(\mathbb{E}\left[ e^{\frac{p}{2}\beta A_{T}}\big|\widehat{\xi}\big|^p\right]  +\mathbb{E}\int_{0}^{T}e^{\beta A_s}\left|\widehat{f}(s,Y^2_s,Z^2_s,U_s^2)\right|^pds\right).
	\end{split}
\end{equation}
	Let $\nu_\varepsilon$ be the function defined on $\mathbb{R}$ as in the proof of Lemma \ref{Lemma betap for inequality}. Note that, from the definition of the function $\nu_\varepsilon$, it is easy to see that 
	\begin{equation*}
	\begin{split}
		e^{\frac{2-p}{4}\beta A_s}\left( \nu_\varepsilon\left( \big|\widehat{Y}_{s-}\big|\vee \big|\widehat{Y}_{s-}+\widehat{U}_s(e)\big|\right) \right) ^{\frac{2-p}{2}}&=\left( e^{\frac{1}{2}\beta A_s}\nu_\varepsilon\left( \big|\widehat{Y}_{s-}\big|\vee \big|\widehat{Y}_{s-}+\widehat{U}_s(e)\big|\right) \right) ^{\frac{2-p}{2}}\\
		&=\left[ \nu_{\varepsilon_p}\left( e^{\frac{1}{2}\beta A_s}\left( \big|\widehat{Y}_{s-}\big|\vee \big|\widehat{Y}_{s-}+\widehat{U}_s(e)\big|\right) \right) \right] ^{\frac{2-p}{2}}
	\end{split}
	\end{equation*}
with $\varepsilon_p:=\varepsilon e^{\frac{1}{2}\beta A_s}$ which tend to zero as $\varepsilon \downarrow 0$ since $e^{\frac{1}{2}\beta A_s} \leq e^{\beta A_T}<+\infty$ a.s. This last property is derived from assumption \textsc{(H4)} implying $\mathbb{E}\left[\int_{0}^{T}e^{\beta A_s} ds\right]<+\infty$. Additionally, for any $p \in (1,2)$, note that $\lim\limits_{\varepsilon \downarrow 0} \left(\nu_\varepsilon(y)\right)^{p}=\big|y\big|^p$ and $\lim\limits_{\varepsilon \downarrow 0} \left(\nu_\varepsilon(y)\right)^{p-2}=\big|y\big|^{p-2}\mathds{1}_{\{y \neq 0\}}$. Following this and using Young's inequality, we have
\begin{equation}\label{O5}
	\begin{split}
		&\mathbb{E}\left[\left(\int_{0}^{T}e^{\beta A_s}\int_{E} \big|\widehat{U}_s(e)\big|^2 N(ds,de)\right)^{\frac{p}{2}}\right]\\
		&=\mathbb{E}\left[\left(\int_{0}^{T}e^{\frac{2-p}{2}\beta A_s}\left(\nu_\varepsilon(\big|\widehat{Y}_{s-}\big|\vee \big|\widehat{Y}_{s-}+\widehat{U}_s(e)\big|\big)\right)^{2-p} \right.\right.\\
		&\left.\left. \qquad\qquad\qquad\qquad\qquad \times \left(\nu_\varepsilon(\big|\widehat{Y}_{s-}\big|\vee \big|\widehat{Y}_{s-}+\widehat{U}_s(e)\big|\big)\right)^{p-2}e^{\frac{p}{2} \beta A_s} \int_{E} \big|\widehat{U}_s(e)\big|^2 N(ds,de)\right)^{\frac{p}{2}}\right]\\
		&=\mathbb{E}\left[\left(\int_{0}^{T}\left(e^{\frac{2-p}{4}\beta A_s}\left(\nu_\varepsilon(\big|\widehat{Y}_{s-}\big|\vee \big|\widehat{Y}_{s-}+\widehat{U}_s(e)\big|\big)\right)^{\frac{2-p}{2}}\right)^2\right.\right.\\
		&\left.\left. \qquad\qquad\qquad\qquad\qquad \times \left(\nu_\varepsilon(\big|\widehat{Y}_{s-}\big|\vee \big|\widehat{Y}_{s-}+\widehat{U}_s(e)\big|\big)\right)^{p-2}e^{\frac{p}{2} \beta A_s} \int_{E} \big|\widehat{U}_s(e)\big|^2 N(ds,de)\right)^{\frac{p}{2}}\right]\\
		&=\mathbb{E}\left[\left(\int_{0}^{T}\left(\nu_{\varepsilon_p}(e^{\frac{1}{2}\beta A_s}\big|\widehat{Y}_{s-}\big|\vee \big|\widehat{Y}_{s-}+\widehat{U}_s(e)\big|)\right)^{2-p}\right.\right.\\
		&\left.\left. \qquad\qquad\qquad\qquad\qquad \times \left(\nu_\varepsilon(|\widehat{Y}_{s-}\big|\vee \big|\widehat{Y}_{s-}+\widehat{U}_s(e)\big|)\right)^{p-2}e^{\frac{p}{2} \beta A_s} \int_{E} \big|\widehat{U}_s(e)\big|^2 N(ds,de)\right)^{\frac{p}{2}}\right]\\
		&\leq \mathbb{E}\left[\left(\nu_{\varepsilon_p}\left(\left( e^{\frac{1}{2}\beta A} \widehat{Y}\right)_\ast\right)\right)^{p\frac{2-p}{2}}\left(\int_{0}^{T} e^{\frac{p}{2}\beta A_s}\left(\nu_\varepsilon\big(\big|\widehat{Y}_{s-}\big|\vee \big|\widehat{Y}_{s-}+\widehat{U}_s(e)\big|\big)\right)^{p-2} \int_{E} \big|\widehat{U}_s(e)\big|^2 N(ds,de)\right)^{\frac{p}{2}}\right]\\
		& \leq \frac{2-p}{2}\mathbb{E}\left[\left(\nu_{\varepsilon_p}\left(\left( e^{\frac{1}{2}\beta A} \widehat{Y}\right)_\ast\right)\right)^{p}\right]+\frac{p}{2} \mathbb{E}\left[\int_{0}^{T} e^{\frac{p}{2}\beta A_s}\left(\nu_\varepsilon\big(\big|\widehat{Y}_{s-}\big|\vee \big|\widehat{Y}_{s-}+\widehat{U}_s(e)\big|\big)\right)^{p-2} \int_{E} \big|\widehat{U}_s(e)\big|^2 N(ds,de)\right].
	\end{split}
\end{equation}
Let $\varepsilon$ tend to zero.  We can use a convergence theorem, which is a consequence of the estimations \eqref{Y sup p less than 2} and \eqref{Ito for p<2--v1.3} to derive:
\begin{equation*}
	\begin{split}
		&\mathbb{E}\left[\left(\int_{0}^{T}e^{\beta A_s}\int_{E} \big|\widehat{U}_s(e)\big|^2 N(ds,de)\right)^{\frac{p}{2}}\right]\\
		& \leq \frac{2-p}{2}\mathbb{E}\left[\sup_{t \in [0,T]} e^{\frac{p}{2}\beta A_{t  }}\big|\widehat{Y}_{t}\big|^p\right]\\
		&\qquad+\frac{p}{2} \mathbb{E}\left[\int_{0}^{T}e^{\frac{p}{2}\beta A_s}\int_{E}\big|\widehat{U}_s(e)\big|^2\left(|\widehat{Y}_{s-}|^{2}\vee|\widehat{Y}_{s-}+\widehat{U}_s(e)|^2\right)^{\frac{p-2}{2}}\mathds{1}_{\{|\widehat{Y}_{s-}|\vee|\widehat{Y}_{s-}+\widehat{U}_s(e)|\neq 0\}}N(ds,de)\right]\\
		& \leq  \mathfrak{C}_{p,T} \left(\mathbb{E}\left[ e^{\frac{p}{2}\beta A_{T}}\big|\widehat{\xi}\big|^p\right]  +\mathbb{E}\int_{0}^{T}e^{\beta A_s}\left|\widehat{f}(s,Y^2_s,Z^2_s,U_s^2)\right|^pds\right).
	\end{split}
\end{equation*}
	To conclude the proof of the proposition, it remains to show the estimation for the remaining term 
$$
\mathbb{E}\left[\left(\int_{0}^{T}e^{\beta A_s} \big\|\widehat{Z}_s\big\|^2 ds\right)^{\frac{p}{2}}\right].
$$
To this end, note that if $\widehat{Y}=0$ then after writing the BSDE \eqref{basic BSDE} forwardly, we derive that the continuous martingale part $\left(\int_{0}^{t}\widehat{Z}_s dW_s\right)_{t \leq T}$ have finite variation over $[0,T]$. Since it is predictable, we derive from \cite[Chapter I. Corollary 3.16]{jacodshiryaev2003} that $\int_{0}^{\cdot}\widehat{Z}_s dW_s=0$ up to an evanescent set. Then, using this with Young's inequality and \eqref{Ito for p<2--v1.3}, \eqref{Y sup p less than 2},  we have
\begin{equation*}
	\begin{split}
		\mathbb{E}\left[\left(\int_{0}^{T}e^{\beta A_s} \big\|\widehat{Z}_s\big\|^2 ds\right)^{\frac{p}{2}}\right]
		&=\mathbb{E}\left[\left(\int_{0}^{T}e^{\beta A_s} \big\|\widehat{Z}_s\big\|^2 \mathds{1}_{\{\widehat{Y}_s\neq0\}} ds\right)^{\frac{p}{2}}\right]\\
		& \leq \mathbb{E}\left[\left( \sup_{t \in [0,T]}\left(e^{\frac{p(2-p)}{4}\beta A_s} \big|\widehat{Y}_s\big|^{\frac{p(2-p)}{2}}\right)\right) \left(\int_{0}^{T} \big\|\widehat{Z}_s\big\|^2 \big|\widehat{Y}_s\big|^{p-2} \mathds{1}_{\{\widehat{Y}_s\neq0\}} ds\right)^{\frac{p}{2}}\right]\\
		& \leq \frac{2-p}{2}\mathbb{E}\left[\sup_{t \in [0,T]} e^{\frac{p}{2}\beta A_{t  }}\big|\widehat{Y}_{t}\big|^p\right]+\frac{p}{2} \mathbb{E}\left[\int_{0}^{T}e^{\frac{p}{2}\beta A_s}\big|\widehat{Y}_{s}\big|^{p-2} \big\|\widehat{Z}_s\big\|^2 \mathds{1}_{\{\widehat{Y}_{s}\neq 0\} }ds\right]\\
		& \leq  \mathfrak{C}_{p,T} \left(\mathbb{E}e^{\frac{p}{2}\beta A_{T}}\big|\widehat{\xi}\big|^p +\mathbb{E}\int_{0}^{T}e^{\beta A_s}\left|\widehat{f}(s,Y^2_s,Z^2_s,U_s^2)\right|^pds\right).
	\end{split}
\end{equation*}
Completing the proof.

\end{proof}
\begin{remark}
	We recall that, if one follows the classical argument with the jump component taken in the space $\mathbb{L}^2_Q$, then the application of Lemma \ref{Lemma betap for inequality} leads to the appearance, on the left-hand side of estimate \eqref{Ito for p<2--v1.1}, of the term
	$$
	\mathfrak{N}:=
	\left(
	\mathfrak{b}_p
	\int_{t \wedge \tau}^{\tau}
	e^{\frac{p}{2}\beta A_s+\mu s}
	\int_E
	|\widehat{U}_s(e)|^2
	\big(
	|\widehat{Y}_{s-}|^2
	\vee
	|\widehat{Y}_{s-}+\widehat{U}_s(e)|^2
	\big)^{\frac{p-2}{2}}
	\mathds{1}_{\{|\widehat{Y}_{s-}|\vee|\widehat{Y}_{s-}+\widehat{U}_s(e)|\neq0\}}
	N(ds,de)
	\right)_{t\leq T}.
	$$
	At the same time, the stochastic-Lipschitz dependence of the generator with respect to the jump component gives rise to a term of the form
	$$
	p\int_{t \wedge \tau}^{\tau}
	e^{\frac{p}{2}\beta A_s+\mu s}
	|\widehat{Y}_s|^{p-1}
	\|\widehat{U}_s\|_{\mathbb{L}^2_Q}\delta_s ds,
	$$
	which is obtained through computations similar to those leading to \eqref{O4}. In the case $p\geq2$, this term can be handled by arguments of the same type as those used to obtain \eqref{O5}. However, when $p\in(1,2)$, the counterexample given in \cite[p. 2]{Kruse2017} shows that this strategy fails in general. More precisely, the process
	$$
	\mathfrak{M}:=
	\left(
	\mathfrak{b}_p
	\int_{t \wedge \tau}^{\tau}
	e^{\frac{p}{2}\beta A_s+\mu s}
	\int_E
	|\widehat{U}_s(e)|^2
	\big(
	|\widehat{Y}_{s-}|^2
	\vee
	|\widehat{Y}_{s-}+\widehat{U}_s(e)|^2
	\big)^{\frac{p-2}{2}}
	\mathds{1}_{\{|\widehat{Y}_{s-}|\vee|\widehat{Y}_{s-}+\widehat{U}_s(e)|\neq0\}}
	\widetilde{N}(ds,de)
	\right)_{t\leq T}
	$$
	cannot, in general, be justified as a local martingale under the sole $\mathbb{L}^2_Q$-framework. This is precisely the property needed in order to pass from $\mathfrak{N}$ to its predictable compensator.
	
	Since this passage is not available in general, the functional space for the jump component has to be modified. This explains the main difficulty compared with the case $p\geq2$. In particular, the a priori estimates for $p\in(1,2)$ cannot be obtained by a direct adaptation of the classical $\mathbb{L}^2_Q$ argument; they require working in the space $\mathbb{L}^1_Q+\mathbb{L}^2_Q$ together with the auxiliary estimates recalled above. The obstruction comes from the fact that, for $p<2$, the dual predictable projection of the martingale associated with the stochastic integral with respect to $\widetilde{N}$ is not controlled by the corresponding quadratic projection, as already discussed in Remark \ref{rmq p}.
\end{remark}
\section{Existence and uniqueness of $\mathbb{L}^p$-solutions for $p\in(1,+\infty)$}
\label{sec4}

Recall that, if $p\in(1,2)$, then $\mathscr{L}^p_Q=\mathbb{L}^1_Q+\mathbb{L}^2_Q$, whereas if $p\geq2$, then $\mathscr{L}^p_Q=\mathbb{L}^2_Q$.

\subsection{$\mathbb{L}^p$-solutions for $p \in [2,+\infty)$}
From Propositions \ref{Propo p=2} and \ref{Propo p sup a 2}, we can easily deduce the following results:
\begin{corollary}\label{Uniquenss p geq 2}
	Let $p \geq 2$. Under assumptions \textsc{(H2)}--\textsc{(H6)} on the data $(\xi, f)$, the BSDE \eqref{basic BSDE} associated with $(\xi,f)$ admits at most one $\mathbb{L}^p$-solution.
\end{corollary}
\begin{corollary}\label{Integrability property}
	Let $p \geq 2$, $(\xi, f)$ be a pair of data satisfying assumptions \textsc{(H2)}--\textsc{(H6)}, and $(Y, Z, U)$ an $\mathbb{L}^p$-solution of the BSDE \eqref{basic BSDE} associated with $(\xi, f)$. Then,
	\begin{itemize}
		\item For any ${\beta} >0$, there exists a constants $\mathfrak{c}_{p,\epsilon,T}$ such that  
		\begin{equation}
			\begin{split}
				&\mathbb{E}\left[\sup_{t \in [0,T]}e^{{\beta} A_t}\big|{Y}_{t}\big|^{p}\right]+\mathbb{E}\int_{0}^{T} e^{\beta A_s} \big|{Y}_s\big|^p dA_s
				\leq \mathfrak{c}_{p,\epsilon,T} \left(\mathbb{E}\left[ e^{{\beta} A_T} \big|{\xi}\big|^p\right]  + \mathbb{E}\int_{0}^{T}e^{{\beta} A_s} \left| \varphi_s\right|^p ds\right).
			\end{split}
		\label{Sofii}
		\end{equation}
		
		\item For any $\beta >0$ such that \eqref{Sofii} holds, we have
		\begin{equation*}
			\begin{split}
				&\mathbb{E}\left[\left(\int_{0}^{T}e^{\beta A_s}\big\|{Z}_s\big\|^2ds\right)^{\frac{p}{2}}\right]+\mathbb{E}\left[\left(\int_{0}^{T}e^{\beta A_s} \big\|{U}_s\big\|^2_{\mathbb{L}^2_Q} ds\right)^{\frac{p}{2}}\right]+\mathbb{E}\left[\left(\int_{0}^{T}e^{\beta A_s} \int_{E}\big|{U}_s(e)\big|^2N(ds,de)\right)^{\frac{p}{2}}\right]\\
				&\leq \mathfrak{C}_{p,\epsilon,T}  \left(\mathbb{E}\left[ e^{(p-1){\beta} A_T} \big|{\xi}\big|^p\right]  + \mathbb{E}\int_{0}^{T}e^{(p-1){\beta} A_s} \left| \varphi_s\right|^p ds\right).
			\end{split}
		\end{equation*}
	\end{itemize}
\end{corollary}

Now, we can state the main result of the current section.
\begin{theorem}\label{thm1}
	Suppose that \textsc{(H1)}--\textsc{(H6)} hold. Then the BSDE \eqref{basic BSDE} has a unique $\mathbb{L}^p$-solution for $p \geq 2$.
\end{theorem}

\begin{proof}
	For any $\omega \in \Omega$, $n \geq 1$, define\footnote{This is exactly the sequence of stopping times defined in Remark \ref{SPR} and given by \eqref{kech}.}
	$$
	\tau_n(\omega):=\inf\left\{t \geq 0 : a_t(\omega) \geq n \right\} \wedge T.
	$$
	As $(a_t)_{t \leq T}$ is a progressively measurable process and $\mathbb{F}$ satisfies the usual conditions of right-continuity and completeness, the Début theorem \cite[Chapter IV. Theorem 50]{DellacherieMeyer1975} infers that the random variables $\left\{\tau_n\right\}_{n \geq 1}$ is an increasing sequence of stopping times such that $\tau_n \nearrow T$ as $n\rightarrow+\infty$.\\
	Next, for each $n \geq 1$ and all $(\omega,t,y,z,u) \in \Omega \times [0,T] \times \mathbb{R}^d \times \mathbb{R}^{d \times k}\times \mathbb{L}^2_Q$, we set
	$$
	f_n(t,y,z,u):=\mathds{1}_{\{t \leq \tau_n(\omega)\}}f(\omega,t,y,z,u).
	$$
	Since $\mathds{1}_{\llbracket 0,\tau_n\rrbracket}$ is predictable, we derive that the process $(\omega,t) \mapsto f_n(\omega,t,y,z,u)$ is progressively measurable.\\ 
	Let $n \geq 1$ be fixed. It is clear that the Lebesgue integrator measure $dt$ for the coefficient $f$ does not jump, then the dependence on $\mathds{1}_{\{t \leq \tau_n\}}$ or $\mathds{1}_{\{t < \tau_n\}}$
	does not matter as $\mathds{1}_{\llbracket 0,\tau_n \rrbracket}$ has one jump time and thus of $dt$-measure zero. Therefore, then for $d\mathbb{P} \otimes dt$-almost every $(\omega,t)$, it holds true that $\mathds{1}_{\{t \leq \tau_n(\omega)\}}=\mathds{1}_{\{t < \tau_n(\omega)\}}$. \\
	Then, for any $n \geq 1$ and all $t \in [0,T]$, $y \in \mathbb{R}^d$, $z,z' \in \mathbb{R}^{d \times k}$, $u, u' \in \mathbb{L}^2_Q$, $d\mathbb{P}\otimes dt$-a.e., we have
	$$
	\big|	f_n(t,y,z,u)-	f_n(t,y,z',u')\big| \leq n\left(\big|z-z'\big|+\big\|u-u'\big\|_{\mathbb{L}^2_Q}\right).
	$$
	Moreover, from \textsc{(H2)} and Remark \ref{rmq essential}, for all $t \in [0,T]$, $y,y' \in \mathbb{R}^d$, $z \in \mathbb{R}^{d \times k}$, $u \in \mathbb{L}^2_Q$, $d\mathbb{P}\otimes dt$-a.e., we have
	\begin{equation}\label{fk mon}
		\left(y-y'\right)\left(f_n(t,y,z,u)-f_n(t,y',z,u)\right) \leq 0
		.
	\end{equation}
	On the other hand, for all $t \in [0,T]$, $y, \in \mathbb{R}^d$, $z,z' \in \mathbb{R}^{d \times k}$, $u \in \mathbb{L}^2_Q$, $d\mathbb{P}\otimes dt$-a.e., we have $\left|f_n(t,y,0,0)\right| \leq \left|f(t,y,0,0)\right| \leq \varphi_t+\phi_t |y|$ and the function $y \mapsto f_n(t,y,z,u)$ is continuous by assumption (H1).\\
	Then, for $p \geq 2$, we have 
	$$
	\mathbb{E}\left[\int_{0}^{T}\big|f_n(s,0,0,0)\big|^p ds\right] \leq \mathbb{E}\left[\int_{0}^{T}\big|\varphi_s\big|^p ds\right] < +\infty,
	$$
	and for any $\mathsf{r} > 0$, we have
	\begin{equation*}
		\begin{split}
			\sup_{|y| \leq \mathsf{r}} \big|f_n(t,y,0,0) - f_n(t,0,0,0)\big|&=\sup_{|y| \leq \mathsf{r}} \mathds{1}_{\{t \leq \tau_n \}} \big|f(t,y,0,0) - f(t,0,0,0)\big|\\
			& \leq  \mathds{1}_{\{t \leq \tau_n \}} \left(\sup_{|y| \leq \mathsf{r}} \big|f_n(t,y,0,0)\big| + \big|{f}_n(t,0,0,0)\big|\right)\\
			& \leq   \left(2\varphi_t + \mathsf{r} \phi_t\right)\mathds{1}_{\{t \leq \tau_n \}}.
		\end{split}
	\end{equation*}
From the definition of the stopping time $\tau_n$ and the definition of 
$(a_t)_{t \leq T}$ given in assumption \textsc{(H5)}, we obtain
$$
\mathbb{E}\left[\int_{0}^{\tau_n}\phi_s ds\right]
\leq n^2T<+\infty .
$$
Moreover, since the exponential weight is greater than or equal to one, assumption 
\textsc{(H6)} implies
$$
\mathbb{E}\int_0^T |\varphi_s|^p ds<+\infty.
$$
Hence, by Hölder's inequality on $\Omega\times[0,T]$, we have
$$
\mathbb{E}\left[\int_{0}^{\tau_n}\varphi_s ds\right]
\leq
\mathbb{E}\left[\int_{0}^{T}\varphi_s ds\right]
\leq
T^{\frac{p-1}{p}}
\mathbb{E}\left[
\left(\int_{0}^{T}|\varphi_s|^pds\right)^{\frac1p}
\right]
\leq
T^{\frac{p-1}{p}}
\left(
\mathbb{E}\int_{0}^{T}|\varphi_s|^pds
\right)^{\frac1p}
<+\infty .
$$
	Thus, for every $\mathsf{r} > 0$, we obtain
	$$
	\sup_{|y| \leq \mathsf{r}} \big|f_n(\cdot, y, 0, 0) - f_n(\cdot, 0, 0, 0)\big| \in \mathbb{L}^1\left(\Omega \times [0,T], d\mathbb{P} \otimes dt\right).
	$$
	Finally, note that from \textsc{(H6)}, we have
	$$
	\mathbb{E}\left[\left|\xi\right|^p\right] \leq \mathbb{E}\left[e^{\beta A_T}\left|\xi\right|^p\right] < +\infty.
	$$
	Therefore, the data $\left(\xi, f_n\right)$ satisfy the assumptions of Theorem 1 and Proposition 2 in \cite{Kruse2015}. Moreover, Theorem \ref{representation thm} ensures that the martingale representation property holds for the Brownian motion $W$ and the integer-valued random measure $N$. As a result, on each interval $[0, \tau_n]$, the framework and conditions of Propositions 1--2 and Theorem 1 in \cite{Kruse2015} are fulfilled. Consequently, for any $n \geq 1$, there exists a unique triplet $(Y^n, Z^n, U^n) = (Y^n_t, Z^n_t, U^n_t)_{t \leq T}$ such that
	\begin{equation}\label{Approx BSDE k p geq 2}
		Y^n_t = \xi + \int_{t}^{T} f_n(s, Y^n_s, Z^n_s, U^n_s) ds - \int_{t}^{T} Z^n_s dW_s - \int_{t}^{T} \int_{E} U^n_s(e) \tilde{N}(ds,de).
	\end{equation}
   Moreover, it is clear that the newly introduced driver $f_n$ satisfies the assumptions of Proposition \ref{Propo p sup a 2}, where the resulting estimation for the driver remains independent of $n \geq 1$. Indeed, using assumptions (H2)--(H3) and Remark \ref{rmq essential}, we obtain, for all $t \in [0,T]$, $y, y' \in \mathbb{R}^d$, $z, z' \in \mathbb{R}^{d \times k}$, and $u, u' \in \mathbb{L}^2_Q$, that $d\mathbb{P} \otimes dt$-a.e.,
   \begin{equation}\label{fk prp}
   	\left\lbrace 
   	\begin{split}
   		& \left(y-y'\right)\left(f_n(t,y,z,u)-f_n(t,y',z,u)\right) \leq \mathds{1}_{\{t \leq \tau_n\}} \alpha_t \left|y-y'\right|^2, \\  
   		& \left| f_n(t,y,z,u)-f_n(t,y,z',u')\right|  \leq \mathds{1}_{\{t \leq \tau_n\}} \left(\eta_t \left\|z-z'\right\|+\delta_t \left\|u-u'\right\|_{\mathbb{L}^2_Q}\right) \leq \eta_t \left\|z-z'\right\|+\delta_t \left\|u-u'\right\|_{\mathbb{L}^2_Q},\\  
   		& \mathds{1}_{\{t \leq \tau_n\}}\left(\alpha_t +\varepsilon  a_t^2\right) \leq 0, \quad \forall t \in [0,T], \text{ and for any } \varepsilon \geq 0.
   	\end{split}
   	\right. 
   \end{equation}
	Consequently, we deduce from Corollary \ref{Integrability property} that, for any $\beta > 0$ and $n \geq 1$, there exists a constant $\mathfrak{C}_{p,\epsilon,T}$ (independent of $n$) such that
	\begin{equation}\label{UI for p=2}
		\begin{split}
			&\mathbb{E}\left[\sup_{t \in [0,T]}e^{{\beta} A_t}\big|{Y}^n_{t}\big|^{p}\right] + \mathbb{E}\int_{0}^{T} e^{\beta A_s} \big|{Y}^n_s\big|^p dA_s + \mathbb{E}\left[\left(\int_{0}^{T}e^{\beta A_s}\big\|{Z}^n_s\big\|^2ds\right)^{\frac{p}{2}}\right] \\
			&+ \mathbb{E}\left[\left(\int_{0}^{T}e^{\beta A_s} \big\|{U}^n_s\big\|^2_{\mathbb{L}^2_Q} ds\right)^{\frac{p}{2}}\right]+\mathbb{E}\left[\left(\int_{0}^{T}e^{\beta A_s} \int_{E}\big|{U}^n_s(e)\big|^2N(ds,de)\right)^{\frac{p}{2}}\right] \\
			&\leq \mathfrak{C}_{p,\epsilon,T} \left(\mathbb{E}\left[ e^{(p-1){\beta} A_T} \big|{\xi}\big|^p\right]  + \mathbb{E}\int_{0}^{T}e^{(p-1){\beta} A_s} \left| \varphi_s\right|^p ds\right).
		\end{split}
	\end{equation}
	Here, no boundedness assumption on $A_T$ is imposed. The argument is justified by localization. More precisely, one first applies the arguments of Corollary \ref{Integrability property} on the localized intervals associated with the sequence $\{\sigma_k\}_{k\geq1}$ defined in Remark \ref{rmq p}, for which the stopped increasing process is bounded. This gives the desired estimates locally. Then, by passing to the limit as $k\to+\infty$ and using Fatou's lemma, we obtain \eqref{UI for p=2} on the whole interval $[0,T]$. Next, we aim to show that the approximating BSDE \eqref{Approx BSDE k p geq 2} converges toward the solution $(Y, Z, U)$ of the BSDE \eqref{basic BSDE} in the space $\mathcal{E}^2_\beta(0, T)$. Then, we demonstrate that the limiting process belongs to $\mathcal{E}^p_\beta(0, T)$, so that the limiting process of the BSDE \eqref{Approx BSDE k p geq 2} is the unique $\mathbb{L}^p$-solution of the BSDE \eqref{basic BSDE}.\\
	Let us fix some integers $n, m \geq 1$ such that $n \leq m$. Let $(Y^n, Z^n, U^n)$ and $(Y^m, Z^m, U^m)$ be two $\mathbb{L}^2$-solutions of the BSDE \eqref{Approx BSDE k p geq 2} associated with the parameters $(\xi, f_n)$ and $(\xi, f_m)$, respectively. Define $\widehat{\mathcal{R}}^{n,m} = \mathcal{R}^n - \mathcal{R}^m$, where $\mathcal{R} \in \{Y, Z, U\}$. By applying It\^o's formula, we have
	\begin{equation*}
		\begin{split}
			&e^{\beta A_t} \big|\widehat{Y}^{n,m}_t\big|^2+\beta \int_{t}^{T}e^{\beta A_s} \big|\widehat{Y}^{n,m}_s\big|^2 dA_s+\int_{t}^{T}e^{\beta A_s }\big\|\widehat{Z}^{n,m}_s\big\|^2ds+\int_{t}^{T}e^{\beta A_s} \big\|\widehat{U}^{n,m}_s\big\|^2_{\mathbb{L}^2_Q}ds\\
			&=2\int_{t}^{T}e^{\beta A_s}   \widehat{Y}^{n,m}_s \left(f_{n}(s,Y^{n}_s,Z^n_s,U^n_s)-f_m(s,Y^m_s,Z^m_s,U^m_s)\right) ds\\
			&-2\int_{t}^{T}e^{\beta A_s} \widehat{Y}^{n,m}_{s-}  \widehat{Z}^{n,m}_s dW_s -\int_{t}^{T}e^{\beta A_s} \int_{E}\big(\big|\widehat{Y}^{n,m}_{s-}  - \widehat{U}^{n,m}_s(e)\big|^2-\big|\widehat{Y}^{n,m}_{s-}\big|^2\big) \tilde{N}(ds,de) .
		\end{split}
	\end{equation*} 
	Using \eqref{fk mon} and \eqref{fk prp}, we derive
	\begin{equation}\label{Generator p=2}
		 \resizebox{\textwidth}{!}{$
		\begin{split}
			&2\widehat{Y}^{n,m}_s \left(f_{n}(s,Y^{n}_s,Z^n_s,U^n_s)-f_m(s,Y^m_s,Z^m_s,U^m_s)\right)ds\\ 
			&\leq \left(2\eta_s \big|\widehat{Y}_s\big| \big\|\widehat{Z}_s\|+2\delta_s \big|\widehat{Y}^{n,m}_s\big| \big\|\widehat{U}^{n,m}_s\big\|_{\mathbb{L}^2_Q}+2 \big|\widehat{Y}^{n,m}_s\big| \big|{f}_n\left(s,Y^m_s,Z^m_s,U^m_s\right)-{f}_m\left(s,Y^m_s,Z^m_s,U^m_s\right)|\right)ds\\
			& \leq 3 \big|\widehat{Y}^{n,m}_s\big|^2 dA_s+\frac{1}{2}\left(\big\|\widehat{Z}^{n,m}_s\|^2+\big\|\widehat{U}^{n,m}_s\big\|^2_{\mathbb{L}^2_Q}\right)ds+\frac{\big|f_n\left(s,Y^m_s,Z^m_s,U^m_s\right)-f_m\left(s,Y^m_s,Z^m_s,U^m_s\right)\big|^2}{a^2_s} ds.
		\end{split}
	$}
	\end{equation}
	Henceforth, we have
	\begin{equation*}
		\begin{split}
			&e^{\beta A_t} \big|\widehat{Y}^{n,m}_t\big|^2+(\beta-3) \int_{t}^{T}e^{\beta A_s} \big|\widehat{Y}^{n,m}_s\big|^2 dA_s+\frac{1}{2}\int_{t}^{T}e^{\beta A_s }\big\|\widehat{Z}_s\big\|^2ds+\frac{1}{2}\int_{t}^{T}e^{\beta A_s} \big\|\widehat{U}^{n,m}_s\big\|^2_{\mathbb{L}^2_Q}ds\\
			&\leq \int_{t}^{T}e^{\beta A_s }\frac{\big|f_n\left(s,Y^m_s,Z^m_s,U^m_s\right)-f_m\left(s,Y^m_s,Z^m_s,U^m_s\right)\big|^2}{a^2_s} ds
			+2\int_{t}^{T}e^{\beta A_s} \widehat{Y}^{n,m}_{s-}  \widehat{Z}^{n,m}_s dW_s\\ &\qquad-\int_{t}^{T}e^{\beta A_s} \int_{E}\big(\big|\widehat{Y}_{s-}  - \widehat{U}^{n,m}_s(e)\big|^2-\big|\widehat{Y}^{n,m}_{s-}\big|^2\big) \tilde{N}(ds,de).
		\end{split}
	\end{equation*}
	Using similar reasoning as in Proposition \ref{Propo p sup a 2}, with a straightforward adjustment incorporating \eqref{Generator p=2}, we obtain the following result for some constant $\mathfrak{c}_\beta$ (independent of $n$ and $m$):
	\begin{equation*}\label{Ito p=2--Existence}
		\begin{split}
			&\mathbb{E}\left[\sup_{t \in [0,T]}e^{{\beta} A_t}\big|\widehat{Y}^{n,m}_{t}\big|^{2}\right]+\mathbb{E}\int_{0}^{T} e^{\beta A_s} \big|\widehat{Y}^{n,m}_s\big|^2 dA_s+\mathbb{E}\int_{0}^{T}e^{\beta A_s}\big\|\widehat{Z}^{n,m}_s\big\|^2ds+\mathbb{E}\int_{0}^{T}e^{\beta A_s} \big\|\widehat{U}^{n,m}_s\big\|^2_{\mathbb{L}^2_Q} ds\\
			&\leq \mathfrak{c}_\beta \mathbb{E}\int_{0}^{T}e^{\beta A_s }\frac{\big|f_n\left(s,Y^m_s,Z^m_s,U^m_s\right)-f_m\left(s,Y^m_s,Z^m_s,U^m_s\right)\big|^2}{a^2_s} ds.
		\end{split}
	\end{equation*}
	On the other hand, we have
	\begin{equation*}
		\begin{split}
			\big|f_n\left(s,Y^m_s,Z^m_s,U^m_s\right)-f_m\left(s,Y^m_s,Z^m_s,U^m_s\right)\big|&=\big|\mathds{1}_{[0,\tau_n]}(s)f(s,Y^m_s,Z^m_s,U^m_s)-\mathds{1}_{[0,\tau_m]}f(s,Y^m_s,Z^m_s,U^m_s)\big|\\
			&=\mathds{1}_{]\tau_n,\tau_m]}(s) \big|f(s,Y^m_s,Z^m_s,U^m_s)\big|.
		\end{split}
	\end{equation*}
	Therefore, using \eqref{fk prp}, we have
	\begin{equation*}
		\begin{split}
			&\big|f_n\left(s,Y^m_s,Z^m_s,U^m_s\right)-f_m\left(s,Y^m_s,Z^m_s,U^m_s\right)\big|^2 ds \\
			&\leq 2\mathds{1}_{]\tau_n,\tau_m]}(s) \left(\big|f(s,Y^m_s,Z^m_s,U^m_s)-f(s,Y^m_s,0,0)\big|^2+\big|f(s,Y^m_s,0,0)\big|^2\right)\\
			&\leq4\mathds{1}_{]\tau_k,\tau_m]}(s) \left(\eta^2_s\big\|Z^m_s\big\|^2+\delta^2_s\big\|U^m_s\big\|^2_{\mathbb{L}^2_Q}+\big|\varphi_s\big|^2+\big\|Y^m_s\big\|^2 \phi^2_s \right)ds
		\end{split}
	\end{equation*}
	and then
	\begin{equation*}
		\begin{split}
			&\mathbb{E}\int_{0}^{T}e^{\beta A_s }\frac{\big|f_n\left(s,Y^m_s,Z^m_s,U^m_s\right)-f_m\left(s,Y^m_s,Z^m_s,U^m_s\right)\big|^2}{a^2_s} ds\\
			&\leq \mathbb{E}\int_{\tau_n}^{\tau_m}e^{\beta A_s } \left\{\big\|Y^m_s\big\|^2dA_s+\left(\big\|Z^m_s\big\|^2+\big\|U^m_s\big\|^2_{\mathbb{L}^2_Q}+\left|\frac{\varphi_s}{a_s}\right|^2\right)ds\right\}.
		\end{split}
	\end{equation*}
	Passing to the limit as $n \rightarrow +\infty$, along with the uniform estimation \eqref{UI for p=2} and the dominated convergence theorem, we obtain
	$$
	\lim\limits_{n \rightarrow +\infty} \mathbb{E} \int_{0}^{T} e^{\beta A_s} \frac{\big|f_n\left(s,Y^m_s,Z^m_s,U^m_s\right)-f_m\left(s,Y^m_s,Z^m_s,U^m_s\right)\big|^2}{a^2_s} ds = 0.
	$$
	Thus, the sequence $\left\{\big(Y^n, Z^n, U^n\big)\right\}_{n \geq 1}$ is a Cauchy sequence in $\mathcal{E}^2_\beta(0, T)$. Consequently, there exists a triplet $(Y, Z, U) \in \mathcal{E}^2_\beta(0, T)$ that corresponds to the limit of the sequence $\left\{\big(Y^n, Z^n, U^n\big)\right\}_{n \geq 1}$ in the same space. In other words, we have
	\begin{equation}\label{CV p geq 2}
		\lim\limits_{n \rightarrow+\infty}\left(\big\|Y^n-Y\big\|^2_{\mathcal{S}^2_\beta}+\big\|Y^n-Y\big\|^2_{\mathcal{S}^{2,A}_\beta}+\big\|Z^n-Z\big\|^2_{\mathcal{H}^2_\beta}+\big\|U^n-U\big\|^2_{\mathcal{L}^2_{Q,\beta}}\right)=0.
	\end{equation}
	Finally, it remains to verify that the limiting process $(Y, Z, U)$ satisfies the BSDE \eqref{basic BSDE}. This will be achieved by passing to the limit term by term.\\
	Using the convergence results \eqref{CV p geq 2} for the processes $\left\{(Z^n, U^n)\right\}_{n \geq 1}$ associated with the martingale part of the BSDE \eqref{Approx BSDE k p geq 2}, and applying the BDG inequality, we obtain
	\begin{equation*}
		\left\lbrace 
		\begin{split}
			&\lim\limits_{n \rightarrow+\infty}\mathbb{E}\left[\sup_{0\leq t\leq T} \left|\int_{t}^{T}Z^n_s dW_s-\int_{t}^{T}Z_s dW_s\right|\right] =0.\\
			&\lim\limits_{n \rightarrow+\infty}\mathbb{E}\left[\sup_{0\leq t\leq T} \left|\int_{t}^{T} \int_E U^n_s(e) \tilde{N}(ds,de)-\int_{t}^{T}\int_E U_s(e) \tilde{N}(ds,de)\right|\right]=0.
		\end{split}
		\right. 
	\end{equation*}
	For the coefficient part, we have
	\begin{equation*}
		\begin{split}
			&\mathbb{E}\int_{0}^{T}e^{\beta A_s}\dfrac{\big|f_n(s,Y^n_s,Z^n_s,U^n_s)-f(s,Y_s,Z_s,U_s)\big|^2}{a_s^2} ds\\
			& \leq2 \mathbb{E}\int_{0}^{T}e^{\beta A_s}\dfrac{\big|f_n(s,Y^n_s,Z^n_s,U^n_s)-f_n(s,Y^n_s,Z_s,U_s)\big|^2}{a_s^2} ds+2\mathbb{E}\int_{0}^{T}e^{\beta A_s}\dfrac{\big|f_n(s,Y^n_s,Z_s,U_s)-f(s,Y_s,Z_s,U_s)\big|^2}{a_s^2} ds
		\end{split}
	\end{equation*}
	Using the same argument as above, along with \eqref{CV p geq 2}, we can conclude that
	\begin{equation*}
		\lim\limits_{n \rightarrow +\infty} \mathbb{E}\int_{0}^{T} e^{\beta A_s} \dfrac{\big|f_n(s, Y^n_s, Z^n_s, U^n_s) - f_n(s, Y^n_s, Z_s, U_s)\big|^2}{a_s^2} ds = 0.
	\end{equation*}
	For the second term, using the assumption of continuity of the coefficient $f_n$ in the $y$-variable and the definition of $f_k$, we can deduce that $\lim\limits_{n \rightarrow +\infty} f_n(s, Y^n_s, Z_s, U_s) = f(s, Y_s, Z_s, U_s)$ $d\mathbb{P} \otimes dt$-a.s. \\
	Again, employing \eqref{fk prp}, \eqref{UI for p=2}, and \eqref{CV p geq 2}, along with the Dominated Convergence Theorem, we obtain
	\begin{equation*}
		\lim\limits_{n \rightarrow +\infty} \mathbb{E} \int_{0}^{T} e^{\beta A_s} \dfrac{\big|f_n(s, Y^n_s, Z_s, U_s) - f(s, Y_s, Z_s, U_s)\big|^2}{a_s^2} ds = 0.
	\end{equation*}
	Thus, passing to the limit in \eqref{Approx BSDE k p geq 2}, we get
	$$
	Y_t = \xi + \int_{t}^{T} f(s, Y_s, Z_s, U_s) ds - \int_{t}^{T} Z_s dW_s - \int_{t}^{T} \int_{E} U_s(e) \tilde{N}(ds, de), \quad t \in [0, T].
	$$
	To show that the solution $(Y, Z, U)$ belongs to $\mathcal{E}^p_\beta(0, T)$, we use the convergence results \eqref{CV p geq 2}, which imply that $\big(Y^n-Y\big)_{\ast} \rightarrow 0$, $
	\left(\int_{0}^{\cdot} Z^n_s dW_s-\int_{0}^{\cdot} Z_s dW_s\right)_\ast \rightarrow 0$, and $
	\big(\int_{0}^{\cdot} Z^n_s dW_s-\int_{0}^{\cdot} Z_s dW_s\big)_\ast \rightarrow 0$, and $\big(\int_{0}^{\cdot} \int_{E} U^n_s(e) \tilde{N}(ds, de)-\int_{0}^{\cdot} \int_{E} U_s(e) \tilde{N}(ds, de)\big)_\ast \rightarrow 0$   a.s.	Finally, by passing to a subsequence that converges almost surely (if necessary), and using \eqref{UI for p=2} along with the Lebesgue dominated convergence theorem, we conclude that $(Y, Z, U)$ belongs to $\mathcal{E}^p_\beta(0, T)$ for $p \geq 2$.
\end{proof}

\subsection{$\mathbb{L}^p$-solutions for $p \in (1,2)$}
A direct consequence of Proposition \ref{Estimation p less stricly than 2} is given in the following two corollaries:
\begin{corollary}\label{Coro 1 p less than 2}
	Consider a pair of data $(\xi,f)$ satisfying assumptions \textsc{(H2)}--\textsc{(H6)}. Then the BSDE \eqref{basic BSDE} associated with $(\xi,f)$ has at most one $\mathbb{L}^p$-solution.
\end{corollary}
\begin{corollary}\label{coro 2 p less than 2}
	Let $(Y,Z,U)$ be a solution of the BSDE \eqref{basic BSDE} associated with the data $(\xi,f)$ satisfying the assumptions \textsc{(H2)}--\textsc{(H6)}. Then there exist constants $\beta^\ast_{p,\epsilon}$ and $\mathfrak{C}_{p,T}$ such that, for any $\beta\geq\beta^\ast_{p,\epsilon}$, we have
	\begin{equation*}
		\begin{split}
			&\mathbb{E}\left[\sup_{t \in [0,T]} e^{\frac{p}{2}\beta A_{t  }}\big|{Y}_{t}\big|^p \right]+\mathbb{E}\left[\left(\int_{0}^{T}e^{\beta A_s} \big\|{Z}_s\big\|^2 ds\right)^{\frac{p}{2}}\right]
			+\mathbb{E}\left[\left(\int_{0}^{T}e^{\beta A_s}\int_{E} \big|{U}_s(e)\big|^2 N(ds,de)\right)^{\frac{p}{2}}\right]\\
			&\leq  \mathfrak{C}_{p,T} \left(\mathbb{E}\left[ e^{\frac{p}{2}\beta A_{T}}\big|{\xi}\big|^p \right] +\mathbb{E}\int_{0}^{T}e^{\beta A_s}\left|\varphi_s\right|^pds\right).
		\end{split}
	\end{equation*}
\end{corollary}

\subsubsection{The case where the generator $f$ does not depend on the $(z,u)$-variables}
In what follows, we treat the case where the coefficient $f$ is independent of the variables $(z,u)$, that is,
$$
f(t,y,z,u)=f(t,y),
$$
for any $(t,y,z,u)\in[0,T]\times\mathbb{R}^d\times\mathbb{R}^{d\times k}\times\left(\mathbb{L}^1_Q+\mathbb{L}^2_Q\right)$. 

Then, we consider the following version of the BSDE \eqref{basic BSDE}
\begin{equation}\label{basic BSDE--p less than 2 f inde}
	Y_t=\xi+\int_{t}^{T}{f}(s,Y_s)ds-\int_{t}^T Z_s d W_s-\int_{t}^{T}\int_{E}U_s(e) \tilde{N}(ds,de),\quad t \in [0,T].
\end{equation}	
\begin{proposition}\label{Propo 1--Existence p less than 2}
 Suppose that \textsc{(H1)}--\textsc{(H6)} hold. Then the  BSDE \eqref{basic BSDE--p less than 2 f inde} has a unique $\mathbb{L}^p$-solution.
\end{proposition}

\begin{proof}
Throughout this proof, we use the convention introduced after Remark \ref{rmq essential}. In particular, after the change of variables and keeping the same notation for the transformed generator, the generator satisfies the non-increasing condition in the $y$-variable.

Therefore, assumption \textsc{(H2)} can be rewritten as follows: for all $y,y'\in\mathbb{R}^d$, $d\mathbb{P}\otimes dt$-a.e.,
$$
\left(y - y^{\prime}\right)\left({f}\left(t,y\right) - {f}\left(t,y^{\prime}\right)\right) \leq 0.
$$
Then, from assumptions (H1), (H4), and (H6), it is clear that the data $\left(\xi, {f}\right)$ satisfies the conditions of Theorem 2 in \cite{Kruse2015}. Moreover, since $\mathfrak{f}$ does not depend on the $u$-variable, we can deduce from \cite[Theorem 2]{Kruse2015}, \cite[Proposition 3]{Kruse2015}, and \cite[page 2]{Kruse2017} that there exists a unique solution $(Y,Z,U)$ of the BSDE \eqref{basic BSDE--p less than 2 f inde} such that, for some constant $\mathfrak{C}_{p,T}$, we have
\begin{equation*}
	\begin{split}
		&\mathbb{E}\left[\sup_{t \in [0,T]} \big|{Y}_{t}\big|^p \right]+\mathbb{E}\left[\left(\int_{0}^{T} \big\|{Z}_s\big\|^2 ds\right)^{\frac{p}{2}}\right]
		+\mathbb{E}\left[\left(\int_{0}^{T}\int_{E} \big|{U}_s(e)\big|^2 N(ds,de)\right)^{\frac{p}{2}}\right]\\
		&\leq  \mathfrak{C}_{p,T} \left(\mathbb{E}\left[ \big|{\xi}\big|^p\right]  +\mathbb{E}\int_{0}^{T}\left|\varphi_s\right|^pds\right).
	\end{split}
\end{equation*}

Now, because the solution $(Y,Z,U)$ obtained from the results of \cite{Kruse2015,Kruse2017} may not have the required integrability in our setting, we need to use a localization procedure. Indeed, the works \cite{Kruse2015,Kruse2017} deal with Lipschitz coefficients with respect to $(z,u)$, whereas we consider stochastic Lipschitz coefficients. This lack of integrability can be overcome by an argument inspired by the proof of Proposition 54.2 in \cite{Pardoux1997}, or equivalently by the localization procedure that is used throughout this paper. Since, in the present case, the generator $f$ does not depend on the variables $(z,u)$, the situation is simpler than that of Proposition \ref{Estimation p less stricly than 2}. Therefore, we may use the same arguments as in Proposition \ref{Estimation p less stricly than 2}, first when the random variable $A_T$ is bounded. In the general case, we introduce the localizing sequence $(\tau_k)_{k\geq1}$ defined by
$
\tau_k:=\inf\left\{t\geq0:A_t\geq k\right\}\wedge T.
$
Then, by applying Fatou's lemma and using Corollary \ref{coro 2 p less than 2}, we deduce that there exists a constant $\mathfrak{C}_{p,T}$ such that, for any $\beta>0$, we have
\begin{equation}\label{UI for p less than 2}
	\begin{split}
		&\mathbb{E}\left[\sup_{t \in [0,T]} e^{\frac{p}{2}\beta A_{t  }}\big|{Y}_{t}\big|^p \right]+\mathbb{E}\int_{0}^{T}e^{\frac{p}{2}\beta A_{s}}\big|{Y}_{s}\big|^p dA_s+\mathbb{E}\left[\left(\int_{0}^{T}e^{\beta A_s} \big\|{Z}_s\big\|^2 ds\right)^{\frac{p}{2}}\right]\\
		&+\mathbb{E}\left[\left(\int_{0}^{T}e^{\beta A_s}\int_{E} \big|{U}_s(e)\big|^2 N(ds,de)\right)^{\frac{p}{2}}\right]\\
		&\leq  \mathfrak{C}_{p,T} \left(\mathbb{E}\left[ e^{\frac{p}{2}\beta A_{T}}\big|{\xi}\big|^p\right]  +\mathbb{E}\int_{0}^{T}e^{\beta A_s}\left|\varphi_s\right|^pds\right).
	\end{split}
\end{equation}
Completing the proof.
\end{proof}

By the convention introduced after Remark \ref{rmq essential}, we may work with a generator satisfying the non-increasing condition in the $y$-variable.
\begin{proposition}\label{OO7}
	Let $(z,u) \in \mathcal{H}^p_\beta \times \mathcal{L}^p_{N,\beta}$ and $(\xi,f)$ a given pair of data. Suppose that $(\xi,f)$ verifies \textsc{(H1)}--\textsc{(H6)}. Then the BSDE 
	\begin{equation}\label{general case}
		Y_t=\xi+\int_{t}^{T}f (s,Y_s,z_s,u_s)-\int_{t}^{T}Z_s dW_s-\int_{t}^{T} \int_{E}U_s(e)\tilde{N}(ds,de)
	\end{equation}
has a unique $\mathbb{L}^p$-solution.
\end{proposition}

\begin{proof}
 For every $n \geq 1$, we define
 $$
 \xi^n := q_n(\xi), \quad \mathfrak{f}_n(\omega, t, y) := f(\omega, t, y, z_t, u_t) - f(\omega, t, 0, z_t, u_t) + q_n(f(\omega, t, 0, z_t, u_t)),
 $$
 with $q_n(x) = \frac{x n}{|x| \vee n}$. Clearly, for any $y \in \mathbb{R}^d$, the process $\mathfrak{f}_k(\cdot, y)$ is progressively measurable. Moreover, from Remark \ref{rmq essential}, the driver $\mathfrak{f}_k$ satisfies, for all $t \in [0,T]$, $y, y' \in \mathbb{R}^d$, and $d\mathbb{P} \otimes dt$-a.e.,
 \begin{equation}\label{fk new mon}
 	\left(y - y'\right)\left(\mathfrak{f}_k(t, y) - \mathfrak{f}_k(t, y')\right) \leq 0.
 \end{equation}
 Furthermore, it is evident that $\big|q_n(\xi)\big| \leq n$ and 
 $$
 \big|\mathfrak{f}_n(\omega, t, 0)\big|^p = \big|q_n(f(\omega, t, 0, z_t, u_t))\big|^p \leq n^p,
 $$
 which implies, by using the fact that $\varphi \geq 1$ together with assumption \textsc{(H4)}, we obtain
 $$
 \mathbb{E}\int_{0}^{T} e^{\beta A_s} \big|\mathfrak{f}_n(s,0)\big|^p ds \leq  n^p \mathbb{E}\int_{0}^{T} e^{\beta A_s} ds
 \leq
 n^p \mathbb{E}\int_{0}^{T} e^{\beta A_s} |\varphi_s|^p ds
 <+\infty .
 $$
 Using Proposition \ref{Propo 1--Existence p less than 2}, we deduce that, for each $n \geq 1$, there exists a unique triplet $(Y^n, Z^n, U^n)$ that belongs to $\mathcal{E}^p_\beta(0, T)$ and solves the following BSDE:
 \begin{equation}\label{Dhr}
 	Y^n_t = \xi^n + \int_{t}^{T} \mathfrak{f}_n(s, Y^n_s) \, ds - \int_{t}^{T} Z^n_s \, dW_s - \int_{t}^{T} \int_{E} U^n_s(e) \tilde{N}(ds, de),\quad t \in [0,T].
 \end{equation}
 Let us fix integers $ m >n \geq 1$ such that $k \leq m$. Let $(Y^n, Z^n, U^n)$ and $(Y^m, Z^m, U^m)$ be two $\mathbb{L}^2$-solutions of the BSDE \eqref{Approx BSDE k p geq 2} associated with the parameters $(\xi^n, \mathfrak{f}_n)$ and $(\xi^m, \mathfrak{f}_m)$, respectively. Define $\widehat{\mathcal{R}}^{n,m} = \mathcal{R}^n - \mathcal{R}^m$, where $\mathcal{R} \in \{\xi, Y, Z, U\}$. Our goal is to show that the sequence $\left\{\left(Y^n, Z^n, U^n\right)\right\}_{n \geq 1}$ is a Cauchy sequence in $\mathcal{E}^p_\beta(0,T)$. To this end, we first provide estimations for the generator part of the semimartingale $\big(e^{\frac{p}{2}\beta A_t} \big|\widehat{Y}_t\big|\big)_{t \leq T}$. Using Young's inequality, we have
 \begin{equation}\label{new--generator}
 	\begin{split}
 		& p e^{\frac{p}{2}\beta A_s} \big|\widehat{Y}^{n,m}_s\big|^{p-1} \smash{\check{\widehat{Y}}}^{n,m}_s
 		 \left(\mathfrak{f}_n(s, Y^n_s) - \mathfrak{f}_m(s, Y^m_s)\right) \, ds \\
 		& \leq (p-1) e^{\frac{p}{2}\beta A_s} \big|\widehat{Y}^{n,m}_s\big|^p \, dA_s + e^{\frac{p}{2}\beta A_s} \left|\frac{\left|q_n(f(s, 0, z_s, u_s)) - q_m(f(s, 0, z_s, u_s))\right|}{a_s}\right|^p \, ds.
 	\end{split}
 \end{equation}
 Following similar arguments as those used in Proposition \ref{Estimation p less stricly than 2}, with \eqref{new--generator} and an appropriate choice of $\beta$, we deduce the existence of a constant $\mathfrak{C}_{p,T,\epsilon}$ such that
	\begin{equation}\label{UI p less than 2}
	\begin{split}
		&\mathbb{E}\left[\sup_{t \in [0,T]} e^{\frac{p}{2}\beta A_{t  }}\big|\widehat{Y}^{n,m}_{t}\big|^p \right]+\mathbb{E}\int_{0}^{T}e^{\frac{p}{2}\beta A_{s}}\big|\widehat{Y}^{n,m}_{s}\big|^p dA_s+\mathbb{E}\left[\left(\int_{0}^{T}e^{\beta A_s} \big\|\widehat{Z}^{n,m}_s\big\|^2 ds\right)^{\frac{p}{2}}\right]\\
		&\qquad+\mathbb{E}\left[\left(\int_{0}^{T}e^{\beta A_s}\int_{E} \big|\widehat{U}^{n,m}_s(e)\big|^2 N(ds,de)\right)^{\frac{p}{2}}\right]\\
		&\leq  \mathfrak{C}_{p,T,\epsilon} \left(\mathbb{E}\left[ e^{\frac{p}{2}\beta A_T}\big|\widehat{\xi}^{n,m}\big|^p\right] +\mathbb{E}\int_{0}^{T}e^{\frac{p}{2}\beta A_s} \left|\frac{\left|q_n(f(s,0,z_s,u_s))-q_m(f(s,0,z_s,u_s))\right|}{a_s}\right|^p ds\right).
	\end{split}
\end{equation}
On the other hand, by using the basic inequality \eqref{Dizzy}
and assumption \text{(H3)} on $f$, we have
\begin{equation*}
	\begin{split}
		\left|\frac{\left|q_n(f(s,0,z_s,u_s))-q_m(f(s,0,z_s,u_s))\right|}{a_s}\right|^p
		& \leq \mathfrak{C}_p\left(\|z\|^p+\|u_s\|^p_{\mathbb{L}^1_Q+\mathbb{L}^2_Q}+\left|\frac{\varphi_s}{a_s}\right|^p\right)
	\end{split}
\end{equation*}
Then, passing to the Lebesgue-Stieltjes integral and applying Jensen's inequality, we obtain
\begin{equation}\label{Generator p less than 2}
	\begin{split}
		&\int_{0}^{T}e^{\frac{p}{2}\beta A_s} \left|\frac{\left|q_n(f(s,0,z_s,u_s))-q_m(f(s,0,z_s,u_s))\right|}{a_s}\right|^p ds\\
		& \leq \mathfrak{C}_{p,T,\epsilon}\left(  \left(\int_{0}^{T}e^{\beta A_s}\big\|z_s\big\|^2 ds \right)^{\frac{p}{2}}+ \int_{0}^{T}e^{\frac{p}{2}\beta A_s}\big\|u_s\big\|^p_{\mathbb{L}^1_Q+\mathbb{L}^2_Q}ds+ \int_{0}^{T}e^{\beta A_s}\left|\varphi_s\right|^p ds\right) .
	\end{split}
\end{equation}
From the definitions of $q_n$ and $q_m$, we have $\lim\limits_{n,m \rightarrow +\infty} \left|q_n(f(s,0,z_s,u_s)) - q_m(f(s,0,z_s,u_s))\right| = 0$ a.s. Then, using the estimation above, along with the facts that $z \in \mathcal{H}^2_\beta$, $u \in \mathcal{L}^p_{N,\beta}$ with Corollary \ref{cor L1L2 control}, assumption \textsc{(H4)}, and the Lebesgue dominated convergence theorem, we conclude that $\left\{\left(Y^n,Z^n,U^n\right)\right\}_{n \geq 1}$ is a Cauchy sequence in $\mathcal{E}^p_\beta(0,T)$. Thus, there exists a triplet of stochastic processes $(Y,Z,U)$ that represents the limiting process of the sequence $\left\{(Y^n,Z^n,U^n)\right\}_{n \geq 1}$ in the space $\mathcal{E}^p_\beta(0,T)$. By the BDG inequality, we have
\begin{equation*}
 \begin{split}
 	\lim\limits_{n \rightarrow+\infty}\mathbb{E}\left[\sup_{0\leq t\leq T} \left|\int_{t}^{T}Z^n_s dW_s-\int_{t}^{T}Z_s dW_s\right|^p\right] \leq \mathfrak{c}\lim\limits_{n \rightarrow+\infty}\mathbb{E}\left[\left(\int_{0}^{T}\big\|Z^n_s-Z_s\big\|^2ds\right)^{\frac{p}{2}}\right] =0.
 \end{split}
\end{equation*}
Similarly, we derive
\begin{equation*}
	\begin{split}
		\lim\limits_{n \rightarrow+\infty}\mathbb{E}\left[\sup_{0\leq t\leq T} \left|\int_{t}^{T} \int_E U^n_s(e) \tilde{N}(ds,de)-\int_{t}^{T}\int_E U_s(e) \tilde{N}(ds,de)\right|^p\right]=0.
	\end{split}
\end{equation*}
Note that $\big|q_n(\xi)\big|^p \leq \big|\xi\big|^p$ and $\lim\limits_{n \rightarrow +\infty} q_n(\xi) = \xi$ a.s. Then, using assumption \textsc{(H1)} and the dominated convergence theorem, we obtain 
$$
\lim\limits_{n \rightarrow +\infty} \mathbb{E}\left[e^{\frac{p}{2}\beta A_T} \big|q_n(\xi) - \xi\big|^p\right] = 0.
$$
Next, applying the estimation \eqref{UI p less than 2}, and passing to the limit as $n \rightarrow +\infty$, along with Fatou's Lemma and estimation \eqref{Generator p less than 2}, we obtain
	\begin{equation}\label{UI p less than 2--limmiting}
		\resizebox{\textwidth}{!}{$
	\begin{split}
		&\mathbb{E}\left[\sup_{t \in [0,T]} e^{\frac{p}{2}\beta A_{t  }}\big|{Y}_{t}\big|^p \right]+\mathbb{E}\int_{0}^{T}e^{\frac{p}{2}\beta A_{s}}\big|{Y}_{s}\big|^p dA_s+\mathbb{E}\left[\left(\int_{0}^{T}e^{\beta A_s} \big\|{Z}_s\big\|^2 ds\right)^{\frac{p}{2}}\right]\\
		&\qquad+\mathbb{E}\left[\left(\int_{0}^{T}e^{\beta A_s}\int_{E} \big|{U}_s(e)\big|^2 N(ds,de)\right)^{\frac{p}{2}}\right]\\
		&\leq  \mathfrak{C}_{p,T,\epsilon} \left(\mathbb{E}\left[ e^{\frac{p}{2}\beta A_T}\big|{\xi}\big|^p\right] +\mathbb{E}\left[\left(\int_{0}^{T}e^{\beta A_s}\big\|z_s\big\|^2 ds \right)^{\frac{p}{2}}\right] +\mathbb{E}\left[\int_{0}^{T}e^{\frac{p}{2}\beta A_s}\big\|u_s\big\|^p_{\mathbb{L}^1_Q+\mathbb{L}^2_Q} ds\right]+\mathbb{E}\int_{0}^{T}e^{\beta A_s}\left|\varphi_s\right|^p ds\right).
	\end{split}
$}
\end{equation}
Using the continuity of the function $f$ w.r.t. to the $y$-variable (assumption \textsc{(H1)}), we deduce that 
$$
\lim\limits_{n \rightarrow +\infty} \mathfrak{f}_n(t,Y^n_t) = f(t,Y_t,z_t,u_t) \quad \text{a.s.}
$$
Passing to the limit term by term in \eqref{Dhr}, we deduce that $(Y,Z,U)$ is the unique solution of the BSDE \eqref{Dhr}. Moreover, the solution belongs to $\mathcal{E}^p_\beta(0,T)$ due to \eqref{UI p less than 2--limmiting}, completing the proof.
\end{proof}
\subsubsection{The general case where the generator $f$ depends on the $(z,u)$-variables}
Now, we state the main result of this section.
\begin{theorem}
	Suppose that \textsc{(H1)}--\textsc{(H6)} hold. Then there exists a constant $\beta^\ast_{p,\epsilon,T}$ such that, for every $\beta\geq\beta^\ast_{p,\epsilon,T}$ for which the integrability assumptions in \textsc{(H1)} are satisfied, the BSDE \eqref{basic BSDE} admits a unique $\mathbb{L}^p$-solution in $\mathcal{E}^p_\beta(0,T)$.
\end{theorem}

\begin{proof}
	Let $(z, u) \in \mathcal{H}^p_\beta \times \mathcal{L}^p_{N,\beta}$ and set $\mathfrak{f}(t, y) := f(t, y, z_t,u_t)$. From Proposition \ref{OO7}, it follows that the BSDE
	\begin{equation*}
		Y_t = \xi + \int_{t}^{T} f(s, Y_s, z_s, u_s) \, ds - \int_{t}^{T} Z_s \, dW_s - \int_{t}^{T} \int_{E} U_s(e) \tilde{N}(ds, de)
	\end{equation*}
	has a unique solution $(Y, Z, U) \in \mathcal{E}^p_\beta(0, T)$.\\
	Then, for each $(y, z, u) \in \mathcal{B}^p_\beta(0, T):=\mathcal{S}^{p,A}_\beta \times \mathcal{H}^p_\beta \times \mathfrak{L}^p_{\beta}$, we can define a mapping $\Phi(y, z, u) = (Y, Z, U)$ from $\mathcal{B}^p_\beta(0, T)$ into itself, where $(Y, Z, U)$ is the unique solution of (\ref{general case}) associated with $(\xi, f(\cdot, \cdot, z, u))$.\\
	Consider two triplets of stochastic processes $(y^1, z^1, u^1)$ and $(y^2, z^2, u^2)$ in $\mathcal{B}^p_\beta(0, T)$.\\
	Let $(Y^1, Z^1, U^1)$ and $(Y^2, Z^2, U^2)$ be two solutions of the BSDE \eqref{general case} associated with the parameters $(\xi, f(\cdot, \cdot, z^1, u^1))$ and $(\xi, f(\cdot, \cdot, z^2, u^2))$. Define $\widehat{\mathcal{R}} = \mathcal{R}^1 - \mathcal{R}^2$ for $\mathcal{R} \in \left\{Y, Z, U, y, z, u\right\}$. We aim to prove that the mapping $\Phi$ is a strict contraction on $\mathcal{B}^p_\beta(0, T)$ equipped with the norm
	$$
	\|(Y, Z, U)\|_{\mathcal{B}^p_\beta(0, T)}^p
	:= \|Y\|_{\mathcal{S}^{p,A}_\beta}^p + \|Z\|_{\mathcal{H}^{p}_\beta}^p + \|U\|_{\mathcal{L}^{p}_{N,\beta}}^p + \|U\|_{\mathcal{L}^{p}_{Q,\beta}}^p.
	$$
    By applying Lemmas \ref{Lemma Ito for p less than 2} and \ref{Lemma betap for inequality}, we obtain
	\begin{equation}\label{Ito for p<2--v1--Lot of calc}
	\begin{split}
		&e^{\frac{p}{2}\beta A_{t}}\big|\widehat{Y}_{t}\big|^p+\frac{p}{2}\beta\int_{t}^{T}e^{\frac{p}{2}\beta A_s} \big|\widehat{Y}_s\big|^p dA_s
		+\mathfrak{b}_p\int_{t}^{T}e^{\frac{p}{2}\beta A_s} \big|\widehat{Y}_s\big|^{p-2}\big\| \widehat{Z}_s \big\|^2 \mathds{1}_{\{\widehat{Y}_{s} \neq 0\}} ds\\
		&+\mathfrak{b}_p\int_{t}^T\int_{E}e^{\frac{p}{2}\beta A_s}|\widehat{U}_s(e)|^2\left(|\widehat{Y}_{s-}|^{2}\vee|\widehat{Y}_{s-}+\widehat{U}_s(e)|^2\right)^{\frac{p-2}{2}}\mathds{1}_{\{|\widehat{Y}_{s-}|\vee|\widehat{Y}_{s-}+\widehat{U}_s(e)|\neq 0\}}N(ds,de)\\
		\leq& p\int_{t}^{T}e^{\frac{p}{2}\beta A_s} \big|\widehat{Y}_s\big|^{p-1}\check{\widehat{Y}}_s  \left(f(s,Y^1_s, z^1_s,u^1_s)-f(s,Y^2_s, z^2_s,u^2_s)\right)ds \\
		&-p\int_{t}^{T}e^{\frac{p}{2}\beta A_s} \big|\widehat{Y}_s\big|^{p-1} \check{\widehat{Y}}_s \widehat{Z}_s dW_s
		-p\int_{t}^{T}e^{\frac{p}{2}\beta A_s}\int_{E} \big|\widehat{Y}_{s-}\big|^{p-1} \check{\widehat{Y}}_{s-} \widehat{U}_s(e)\tilde{N}(ds,de)
	\end{split}
\end{equation}	 
By employing a similar localization procedure as in Proposition \ref{Estimation p less stricly than 2} and taking the expectation on both sides of \eqref{Ito for p<2--v1--Lot of calc}, we get
	\begin{equation}\label{general case Ito less than 2}
	\begin{split}
		&\mathbb{E}e^{\frac{p}{2}\beta A_{t} }\big|\widehat{Y}_{t}\big|^p+\frac{p}{2}\beta\mathbb{E}\int_{t  }^{T}e^{\frac{p}{2}\beta A_s} \big|\widehat{Y}_s\big|^p dA_s
		+\mathfrak{b}_p\int_{t}^{T}e^{\frac{p}{2}\beta A_s} \big|\widehat{Y}_s\big|^{p-2}\big\| \widehat{Z}_s \big\|^2 \mathds{1}_{\{\widehat{Y}_{s} \neq 0\}} ds\\
		&+\mathfrak{b}_p\mathbb{E}\int_{t}^T\int_{E}e^{\frac{p}{2}\beta A_s}|\widehat{U}_s(e)|^2\left(|\widehat{Y}_{s-}|^{2}\vee|\widehat{Y}_{s-}+\widehat{U}_s(e)|^2\right)^{\frac{p-2}{2}}\mathds{1}_{\{|\widehat{Y}_{s-}|\vee|\widehat{Y}_{s-}+\widehat{U}_s(e)|\neq 0\}}N(ds,de)\\
		\leq& p\mathbb{E}\int_{t }^{T}e^{\frac{p}{2}\beta A_s} \big|\widehat{Y}_s\big|^{p-1}\check{\widehat{Y}}_s  \left(f(s,Y^1_s, z^1_s,u^1_s)-f(s,Y^2_s, z^2_s,u^2_s)\right)ds
	\end{split}
\end{equation}	
Using assumptions \textsc{(H2)}--\textsc{(H3)} and Remark \ref{rmq essential}, we get 
\begin{equation}\label{general case gene p less than 2}
	\begin{split}
		&\big|\widehat{Y}_s\big|^{p-1}\check{\widehat{Y}_s}  \left(f(s,Y^1_s, z^1_s,u^1_s)-f(s,Y^2_s, z^2_s,u^2_s)\right)\\
		&\leq \big|\widehat{Y}_s\big|^{p-2}\mathds{1}_ {\{\widehat{Y}_s \neq 0\}}\left(\alpha_s \big|\widehat{Y}_s\big|^2+ \big|\widehat{Y}_s\big| \big\|\widehat{z}_s\big\|\delta_s+ \big|\widehat{Y}_s\big| \big\|\widehat{u}_s\big\|_{\mathbb{L}^2_Q} \eta_s \right)\\
		& \leq \alpha_s \big|\widehat{Y}_s\big|^{p}  +\big|\widehat{Y}_s\big|^{p-1} \big\|\widehat{z}_s\big\|\delta_s+ \big|\widehat{Y}_s\big|^{p-1} \big\|\widehat{u}_s\big\|_{\mathbb{L}^2_Q}\eta_s \\
		& \leq \big|\widehat{Y}_s\big|^{p-1} \big\|\widehat{z}_s\big\|\delta_s+ \big|\widehat{Y}_s\big|^{p-1} \big\|\widehat{u}_s\big\|_{\mathbb{L}^1_Q+\mathbb{L}^2_Q}\eta_s .
	\end{split}
\end{equation}
On the other hand, using Hölder's inequality, Young's and Jensen's inequalities, for any $\varrho >0$, we obtain 
\begin{equation*}
	\begin{split}
		\int_{t}^{T}e^{\frac{p}{2}\beta A_s }\big|\widehat{Y}_s\big|^{p-1} \big|\widehat{z}_s\big| \delta_s ds 
		&\leq  \left(\int_{t}^{T}e^{\frac{p}{2}\beta A_s }\big|\widehat{Y}_s\big|^{p}dA_s \right)^{\frac{p-1}{p}}\left(\int_{t}^{T}e^{\frac{p}{2} \beta A_s} \big|\widehat{z}_s\big|^p ds\right)^{\frac{1}{p}}\\
		&\leq \frac{(p-1)\varrho^{\frac{1}{p-1}}}{p} \int_{t}^{T}e^{\frac{p}{2}\beta A_s }\big|\widehat{Y}_s\big|^{p}dA_s+\frac{(T-t)^{\frac{2-p}{2}}}{p\varrho} \left(\int_{t}^{T}e^{\beta A_s} \big\|\widehat{z}_s\big\|^2 ds\right)^{\frac{p}{2}}.
	\end{split}
\end{equation*}
Similarly, we can show that 
$$
\int_{t}^{T}e^{\frac{p}{2}\beta A_s }\big|\widehat{Y}_s\big|^{p-1} \big\|\widehat{u}_s\big\|_{\mathbb{L}^2_Q} \eta_s ds \leq  \frac{(p-1)\varrho^{\frac{1}{p-1}}}{p} \int_{t}^{T}e^{\frac{p}{2}\beta A_s }\big|\widehat{Y}_s\big|^{p}dA_s+\frac{1}{p\varrho} \int_{t}^{T}e^{\frac{p}{2} \beta A_s} \big\|\widehat{u}_s\big\|^p_{\mathbb{L}^1_Q+\mathbb{L}^2_Q} ds.
$$
Now, taking the expectation in the above two inequalities and applying Young's inequality and Corollary \ref{cor L1L2 control}, we obtain, for any $\varrho > 0$,
\begin{equation}\label{genera well estimates}
	\begin{split}
		&\mathbb{E}\int_{t}^{T}e^{\frac{p}{2}\beta A_s }\big|\widehat{Y}_s\big|^{p-1} \big|\widehat{z}_s\big| \delta_s ds+\mathbb{E}\int_{t}^{T}e^{\frac{p}{2}\beta A_s }\big|\widehat{Y}_s\big|^{p-1} \big\|\widehat{u}_s\big\|_{\mathbb{L}^2_Q} \eta_s ds\\
		& \leq 2\frac{(p-1)\varrho^{\frac{1}{p-1}}}{p}\mathbb{E}\int_{t}^{T}e^{\frac{p}{2}\beta A_s }\big|\widehat{Y}_s\big|^{p} dA_s+\frac{(T-t)^{\frac{2-p}{2}} \vee 1}{p\varrho} \left(\mathbb{E}\left[\left(\int_{t}^{T}e^{\beta A_s} \big|\widehat{z}_s\big|^2 ds\right)^{\frac{p}{2}}\right]  \right.\\
		&\left.\qquad+\mathbb{E}\left[\int_{t}^{T}e^{\frac{p}{2}\beta A_s} \big\|\widehat{u}_s\big\|^p_{\mathbb{L}^1_Q+\mathbb{L}^2_Q} ds\right] \right)\\
		& \leq 2\frac{(p-1)\varrho^{\frac{1}{p-1}}}{p}\mathbb{E}\int_{t}^{T}e^{\frac{p}{2}\beta A_s }\big|\widehat{Y}_s\big|^{p} dA_s+\frac{(T-t)^{\frac{2-p}{2}} \vee 1 \vee \mathfrak{C}_{p,T}}{p\varrho} \left(\mathbb{E}\left[\left(\int_{t}^{T}e^{\beta A_s} \big|\widehat{z}_s\big|^2 ds\right)^{\frac{p}{2}}\right]  \right.\\
		&\left.\qquad +\mathbb{E}\left[\left(\int_{t}^{T}e^{\beta A_s} \int_{E} \big|\widehat{u}_s(e)\big|^2 N(ds,de)\right)^{\frac{p}{2}}\right] \right),
	\end{split}
\end{equation}
where $\mathfrak{C}_{p,T}>0$ is a constant that stems from Corollary \ref{cor L1L2 control} and is independent of $\varrho$.\\
Coming back to \eqref{general case Ito less than 2}, and using \eqref{general case gene p less than 2} and \eqref{genera well estimates}, we obtain
\begin{equation*}
	\begin{split}
		&\mathbb{E}\left[ e^{\frac{p}{2}\beta A_{t} }\big|\widehat{Y}_{t}\big|^p\right] +\frac{p}{2}\beta\mathbb{E}\int_{t  }^{T}e^{\frac{p}{2}\beta A_s} \big|\widehat{Y}_s\big|^p dA_s
		+\mathfrak{b}_p\int_{t}^{T}e^{\frac{p}{2}\beta A_s} \big|\widehat{Y}_s\big|^{p-2}\big\| \widehat{Z}_s \big\|^2 \mathds{1}_{\{\widehat{Y}_{s} \neq 0\}} ds\\
		&+\mathfrak{b}_p\mathbb{E}\int_t^Te^{\frac{p}{2}\beta A_s}|\widehat{U}_s(e)|^2\left(|\widehat{Y}_{s-}|^{2}\vee|\widehat{Y}_{s-}+\widehat{U}_s(e)|^2\right)^{\frac{p-2}{2}}\mathds{1}_{\{|\widehat{Y}_{s-}|\vee|\widehat{Y}_{s-}+\widehat{U}_s(e)|\neq 0\}}N(ds,de)\\
		&\leq 2(p-1)\varrho^{\frac{1}{p-1}} \mathbb{E}\int_{t}^{T}e^{\frac{p}{2}\beta A_s }\big|\widehat{Y}_s\big|^{p-1} dA_s+\frac{(T-t)^{\frac{2-p}{2}} \vee 1 \vee \mathfrak{C}_{p,T}}{\varrho} \left(\mathbb{E}\int_{t}^{T}e^{\frac{p}{2}\beta A_s} \big|\widehat{y}_s\big|^p dA_s\right.\\
		&\left.\qquad+\mathbb{E}\left[\left(\int_{t}^{T}e^{\beta A_s} \big|\widehat{z}_s\big|^2 ds\right)^{\frac{p}{2}}\right] +\mathbb{E}\left[\left(\int_{t}^{T}e^{\beta A_s} \int_{E} \big|\widehat{u}_s(e)\big|^2 N(ds,de)\right)^{\frac{p}{2}}\right] \right).
	\end{split}
\end{equation*}
It is worth noting that the constant $\mathfrak{C}_{p,T}$ appearing below is still independent of $\varrho$. The parameter $\varrho$ appears only through the Young-type estimates above, and its contribution is absorbed by choosing the exponential parameter $\beta$ large enough. More precisely, for any fixed $\varrho>0$, if
$
\beta \geq 1+2(p-1)\varrho^{\frac{1}{p-1}}=:\beta^\ast_{p,\varrho},
$
then we deduce from the above estimation and a similar localization approach as in Proposition \ref{Estimation p less stricly than 2} that
\begin{equation}\label{Using in global case}
	\begin{split}
		&\mathbb{E}\int_{0}^{T}e^{\frac{p}{2}\beta A_s} \big|\widehat{Y}_s\big|^p dA_s
		+\mathfrak{b}_p\mathbb{E}\int_{0}^{T}e^{\frac{p}{2}\beta A_s} \big|\widehat{Y}_s\big|^{p-2}\big\| \widehat{Z}_s \big\|^2 \mathds{1}_{\{\widehat{Y}_{s} \neq 0\}} ds\\
		&+\mathfrak{b}_p\mathbb{E}\int_0^Te^{\frac{p}{2}\beta A_s}|\widehat{U}_s(e)|^2\left(|\widehat{Y}_{s-}|^{2}\vee|\widehat{Y}_{s-}+\widehat{U}_s(e)|^2\right)^{\frac{p-2}{2}}\mathds{1}_{\{|\widehat{Y}_{s-}|\vee|\widehat{Y}_{s-}+\widehat{U}_s(e)|\neq 0\}}N(ds,de)\\
		&\leq \frac{\mathfrak{C}_{p,T}}{\varrho}  \left(\mathbb{E}\int_{0 }^{T}e^{\frac{p}{2}\beta A_s} \big|\widehat{y}_s\big|^p dA_s+\mathbb{E}\left[\left(\int_{0}^{T}e^{\beta A_s} \big|\widehat{z}_s\big|^2 ds\right)^{\frac{p}{2}}\right] +\mathbb{E}\left[\left(\int_{t}^{T}e^{\beta A_s} \int_{E} \big|\widehat{u}_s(e)\big|^2 N(ds,de)\right)^{\frac{p}{2}}\right] \right).
	\end{split}
\end{equation}
Next, it suffices to repeat the arguments used in Proposition \ref{Estimation p less stricly than 2}, taking into account \eqref{Using in global case}. For the reader's convenience, we briefly outline them. Returning to \eqref{Ito for p<2--v1--Lot of calc} and applying the BDG inequality, we obtain
	\begin{equation*}
	\begin{split}
		&p\mathbb{E}\left[\sup_{t \in [0,T]}\left|\int_{t}^{T}e^{\frac{p}{2}\beta A_s} \big|\widehat{Y}_s\big|^{p-1} \check{\widehat{Y}}_s \widehat{Z}_s dW_s\right|\right]\\
		&\leq \frac{1}{4} \mathbb{E}\left[\sup_{t \in [0,T]}e^{\frac{p}{2}\beta}\big|\widehat{Y}_s\big|^{p}\right]+p^2 \mathfrak{c}^2\mathbb{E}\left[\int_{0}^{T}e^{\frac{p}{2}\beta A_s} \big|\widehat{Y}_s\big|^{p-2} \big\|\widehat{Z}_s\big\|^2 \mathds{1}_{\{\widehat{Y}_s \neq 0\}} ds\right].
	\end{split}
\end{equation*}
The last line follows from the basic inequality $ab \leq \frac{1}{2 \varrho} a^2 + \frac{\varrho}{2} b^2$ for any $\varrho > 0$ (recall that $\mathfrak{c} > 0$ is the universal constant in the BDG inequality). Similarly, by a comparable computation, we have
\begin{equation*}
	\begin{split}
		&p\mathbb{E}\left[\sup_{t \in [0,T]}\left|\int_{t}^{T}e^{\frac{p}{2}\beta A_s}\int_{E} \big|\widehat{Y}_{s-}\big|^{p-1} \check{\widehat{Y}}_{s-} \widehat{U}_s(e)\tilde{N}(ds,de)\right|\right]\\
		& \leq \frac{1}{4}\mathbb{E}\left[\sup_{0\leq t\leq T}e^{\frac{p}{2}\beta A_t}|\widehat{Y}_{t}|^{p}\right]\\
		&\qquad+p^2 \mathfrak{c}^2\mathbb{E}\int_0^T\int_{E}e^{\frac{p}{2}\beta A_s}\left(|\widehat{Y}_{s-}|^{2}\vee|\widehat{Y}_{s-}+\widehat{U}_s(e)|^2\right)^{\frac{p-2}{2}}\mathds{1}_{\{|\widehat{Y}_{s-}|\vee|\widehat{Y}_{s-}+\widehat{U}_s(e)|\neq 0\}}|\widehat{U}_s(e)|^2N(ds,de).
	\end{split}
\end{equation*}
Plugging this into \eqref{Ito for p<2--v1--Lot of calc} and taking the supremum on both sides, we obtain
	\begin{equation*}
	\begin{split}
		&\frac{1}{2}\mathbb{E}\left[\sup_{t \in [0,T]} e^{\frac{p}{2}\beta A_{t  }}\big|\widehat{Y}_{t}\big|^p \right]\\
		\leq&  {p^2 \mathfrak{c}^2}\left(\mathfrak{b}_p\mathbb{E}\left[\int_{0}^{T}e^{\frac{p}{2}\beta A_s} \big|\widehat{Y}_s\big|^{p-2} \big\|\widehat{Z}_s\big\|^2 \mathds{1}_{\{\widehat{Y}_s \neq 0\}} ds\right]\right. \\
		&\left.\qquad+\mathbb{E}\left[\int_0^T\int_{E}e^{\frac{p}{2}\beta A_s}\left(|\widehat{Y}_{s-}|^{2}\vee|\widehat{Y}_{s-}+\widehat{U}_s(e)|^2\right)^{\frac{p-2}{2}}\mathds{1}_{\{|\widehat{Y}_{s-}|\vee|\widehat{Y}_{s-}+\widehat{U}_s(e)|\neq 0\}}|\widehat{U}_s(e)|^2N(ds,de)\right]\right).
	\end{split}
\end{equation*}
This along with \eqref{Using in global case}, we deduce the existence of a constant $\mathfrak{C}_{p,T}$ such that
	\begin{equation}\label{sup for p less than 2--global}
	\begin{split}
		&\mathbb{E}\left[\sup_{t \in [0,T]} e^{\frac{p}{2}\beta A_{t  }}\big|\widehat{Y}_{t}\big|^p \right]\\
		&\leq \mathfrak{c}_p\left(\mathbb{E}\left[\int_{0}^{T}e^{\frac{p}{2}\beta A_s} \big|\widehat{Y}_s\big|^{p-2} \big\|\widehat{Z}_s\big\|^2 \mathds{1}_{\{\widehat{Y}_s \neq 0\}} ds\right] \right.\\
		&\left.  \qquad+\mathbb{E}\left[\int_0^T\int_{E}e^{\frac{p}{2}\beta A_s}\left(|\widehat{Y}_{s-}|^{2}\vee|\widehat{Y}_{s-}+\widehat{U}_s(e)|^2\right)^{\frac{p-2}{2}}\mathds{1}_{\{|\widehat{Y}_{s-}|\vee|\widehat{Y}_{s-}+\widehat{U}_s(e)|\neq 0\}}|\widehat{U}_s(e)|^2N(ds,de)\right]\right)\\
		&\leq \frac{\mathfrak{C}_{p,T}}{\varrho} \left(\mathbb{E}\int_{0 }^{T}e^{\frac{p}{2}\beta A_s} \big|\widehat{y}_s\big|^p dA_s+	\mathbb{E}\left[\left(\int_{0}^{T}e^{\beta A_s} \big|\widehat{z}_s\big|^2 ds\right)^{\frac{p}{2}}\right]   +\mathbb{E}\left[\left(\int_{t}^{T}e^{\beta A_s} \int_{E} \big|\widehat{u}_s(e)\big|^2 N(ds,de)\right)^{\frac{p}{2}}\right] \right).
	\end{split}
\end{equation}
From these estimations, and again using Proposition \ref{Estimation p less stricly than 2}, we obtain
\begin{equation*}
	\begin{split}
		&\mathbb{E}\left[\left(\int_{0}^{T}e^{\beta A_s}\int_{E} \big|\widehat{U}_s(e)\big|^2 N(ds,de)\right)^{\frac{p}{2}}\right]\\
		& \leq \frac{2-p}{2}\mathbb{E}\left[\sup_{t \in [0,T]} e^{\frac{p}{2}\beta A_{t  }}\big|\widehat{Y}_{t}\big|^p\right]\\
		&\quad+\frac{p}{2} \mathbb{E}\left[\int_{0}^{T}e^{\frac{p}{2}\beta A_s}\int_{E}\big|\widehat{U}_s(e)\big|^2\left(|\widehat{Y}_{s-}|^{2}\vee|\widehat{Y}_{s-}+\widehat{U}_s(e)|^2\right)^{\frac{p-2}{2}}\mathds{1}_{\{|\widehat{Y}_{s-}|\vee|\widehat{Y}_{s-}+\widehat{U}_s(e)|\neq 0\}}N(ds,de)\right],
	\end{split}
\end{equation*}
and
	\begin{equation*}
	\begin{split}
		\mathbb{E}\left[\left(\int_{0}^{T}e^{\beta A_s} \big\|\widehat{Z}_s\big\|^2 ds\right)^{\frac{p}{2}}\right]
		 &\leq \frac{2-p}{2}\mathbb{E}\left[\sup_{t \in [0,T]} e^{\frac{p}{2}\beta A_{t}}\big|\widehat{Y}_{t}\big|^p\right]\\
		 &\quad+\frac{p}{2}\mathbb{E}\left[\int_{0}^{T}e^{\frac{p}{2 }\beta A_s}\big|\widehat{Y}_{s}\big|^{p-2} \big\|\widehat{Z}_s\big\|^2 \mathds{1}_{\{\widehat{Y}_{s}\neq 0\} }ds\right].
	\end{split}
\end{equation*}
After summing the three estimates above, and with the same choice of $\beta$ as before, we obtain the existence of a constant $\mathfrak{C}_{p,T,\epsilon}>0$\footnote{Recall that the constant may vary from line to line.} such that
\begin{equation*}
	\begin{split}
		&\mathbb{E}\left[\int_{0}^{T}e^{\frac{p}{2}\beta A_s} \big|\widehat{Y}_s\big|^p dA_s\right]+\mathbb{E}\left[\left(\int_{0}^{T}e^{\beta A_s} \big\|\widehat{Z}_s\big\|^2 ds\right)^{\frac{p}{2}}\right]
		+\mathbb{E}\left[\left(\int_{0}^{T}e^{\beta A_s}\int_{E} \big|\widehat{U}_s(e)\big|^2 N(ds,de)\right)^{\frac{p}{2}}\right]\\
		&\leq \frac{\mathfrak{C}_{p,T,\epsilon}}{\varrho} \left(\mathbb{E}\int_{0 }^{T}e^{\frac{p}{2}\beta A_s} \big|\widehat{y}_s\big|^p dA_s+\left(\int_{0}^{T}e^{\beta A_s} \big|\widehat{z}_s\big|^2 ds\right)^{\frac{p}{2}} +\mathbb{E}\left[\left(\int_{0}^{T}e^{\beta A_s}\int_{E} \big|\widehat{u}_s(e)\big|^2 N(ds,de)\right)^{\frac{p}{2}}\right]\right).
	\end{split}
\end{equation*}
The constant $\mathfrak{C}_{p,T,\epsilon}$ in the last estimate is obtained from the preceding a priori estimates and is independent of $\varrho$. Hence, we may choose $\varrho>0$ large enough so that
$
\frac{\mathfrak{C}_{p,T,\epsilon}}{\varrho}<1.
$
For this fixed choice of $\varrho$, we set
$
\beta^\ast_{p,T,\epsilon}:=1+2(p-1)\varrho^{\frac{1}{p-1}}.
$
Consequently, for every $\beta\geq\beta^\ast_{p,T,\epsilon}$, the functional $\Phi$ is a strict contraction mapping on $\mathcal{B}^p_\beta(0,T)$. Hence, there exists a unique fixed point $(Y,Z,U)$ of $\Phi$, which is the unique $\mathbb{L}^p$-solution, for $p\in(1,2)$, of the BSDE \eqref{basic BSDE} associated with the parameters $(\xi,f)$. Moreover, we already know that $(Y,Z,U)\in\mathcal{B}^p_\beta(0,T)$. It is straightforward to see that the state process $Y$ belongs to $\mathcal{S}^p_\beta$ by virtue of Corollary \ref{coro 2 p less than 2}.\\
The proof is now complete.
\end{proof}

\begin{remark}
	The results of this paper for the case $p\geq2$ remain valid if the standard Poisson random measure is replaced by an integer-valued random measure $N$ on $([0,T]\times E,\mathcal{B}([0,T])\otimes\mathcal{E})$ with compensator $\upsilon(\omega;dt,du)=Q(\omega,t;du)\eta(\omega,t)dt$ satisfying $N(\omega;\{0\}\times E)=N(\omega;(0,T]\times\{0\})=\upsilon(\omega;(0,T]\times\{0\})=0$, where $\eta:\Omega\times[0,T]\to[0,\infty)$ is a predictable process and $Q$ is a kernel from $(\Omega\times[0,T],{\mathcal{P}})$ into $(E,\mathcal{E})$ such that $\int_0^t\int_E |e|^2Q(\omega,s,de)\eta(\omega,s)ds<+\infty$, $\mathbb{P}$-a.s.; in this case, the compensated random measure $\widetilde{N}(\omega;dt,du)=N(\omega; dt,du)-Q(\omega,t;du)\eta(\omega,t)dt$ is an $\mathbb{F}$-martingale measure, whereas the case $p\in(1,2)$ is more delicate than the previous one, since it relies on the auxiliary estimates developed in Section \ref{hrss}, in particular the Bichteler--Jacod inequality, whose use in the present form requires a deterministic compensating kernel $Q$, as discussed in \cite{Marinelli2014}.
\end{remark}

\section*{Disclosure statement}
No potential conflict of interest was reported by the authors.

\section*{Data availability statement}
No new data were created or analyzed in this study. Data sharing is not applicable to this article.

\section*{Funding}
The first author (Badr ELMANSOURI) was supported by the National Center for Scientific and Technical Research (CNRST), Morocco.

\end{document}